\documentclass[11pt]{article}
\usepackage{fullpage}
\usepackage[utf8]{inputenc}
\usepackage{amssymb}
\usepackage{amsmath}
\usepackage{amsthm}
\usepackage{amsfonts}
\usepackage{graphicx}
\usepackage{subcaption}
\usepackage[toc,page]{appendix}

\usepackage{epstopdf}
\usepackage{comment}
\usepackage[retainorgcmds]{IEEEtrantools}

\newtheorem{remark}{Remark}
\newtheorem{theorem}{Theorem}
\newtheorem{proposition}{Proposition}
\newtheorem{corollary}{Corollary}
\newtheorem{lemma}{Lemma}
\newtheorem{assumption}{Assumption}

\def\balpha{\boldsymbol{\alpha}}

\def\bmu{\boldsymbol{\mu}}
\def\bnu{\boldsymbol{\nu}}
\def\bmubar{\boldsymbol{\bar{\mu}}}
\def\bnubar{\boldsymbol{\bar{\nu}}}

\newcommand\cL{\mathcal L}

\newcommand\cP{\mathcal P}
\newcommand\cS{\mathcal S}

\newcommand\EE{\mathbb E}

\title{Price of Anarchy for Mean Field Games}
\author{
Ren\'{e} Carmona\thanks{Operations Research and Financial Engineering, Princeton University, Partially supported by NSF \#DMS-1716673 and ARO \#W911NF-17-1-0578} 
\and 
Christy V. Graves\thanks{Program in Applied and Computational Mathematics, Princeton University, Partially supported by NSF \#DMS-1515753 and NSF GRFP}
\and
Zongjun Tan \thanks{Operations Research and Financial Engineering, Princeton University, Partially supported by NSF \#DMS-1515753} 
}

\begin{document}

\maketitle

\begin{abstract}
    The price of anarchy, originally introduced to quantify the inefficiency of selfish behavior in routing games, is extended to mean field games. The price of anarchy is defined as the ratio of a worst case social cost computed for a mean field game equilibrium to the optimal social cost as computed by a central planner. We illustrate properties of such a price of anarchy on linear quadratic extended mean field games, for which explicit computations are possible. A sufficient and necessary condition to have no price of anarchy is presented. Various asymptotic behaviors of the price of anarchy are proved for limiting behaviors of the coefficients in the model and numerics are presented.
\end{abstract}

\section{\textbf{Introduction}} \label{se:introduction}
The concept of the `price of anarchy' was introduced to quantify the inefficiency of selfish behavior in finite player games \cite{christodoulou2005price2}\cite{christodoulou2005price}\cite{koutsoupias1999worst}\cite{Roughgarden}\cite{roughgarden2002bad}\cite{zhu2010price}. In this report, we extend the notion of price of anarchy to mean field games (MFG). Mean field games were introduced by Lasry and Lions \cite{Lasry_Lions} and Caines and his collaborators \cite{Huang} to describe the limiting regime of large symmetric games when the number of players, $N$, tends to infinity. A mean field game equilibrium characterizes the analogue of a Nash equilibrium in the $N=\infty$ regime. Thus, as in the finite player case, it is possible that the mean field game equilibrium is inefficient. In fact, in the paper of Balandat and Tomlin \cite{balandat2013efficiency}, they present a numerical example that shows that mean field game equilibria are not efficient, in general. The suboptimality of a mean field game equilibrium is also illustrated numerically for a congestion model in a paper of Achdou and Lauri\`{e}re \cite{achdou2015system}. More recently Cardaliaguet and Rainer gave in \cite{pierre_poa} a partial differential equation based thorough analysis of the \emph{(in)efficiency} of the mean field game equilibria.

In this report, the goal is to define the price of anarchy in the context of mean field games, and to compute it for a class of linear quadratic mean field game models, which can be solved explicitly. In fact, we consider an even more general class of games by allowing for interaction between the players through their controls, in addition to interaction through their states. This is often referred in the literature as extended mean field game, or mean field game of control. We compare the social cost of a mean field game equilibrium to the cost incurred when the players execute a strategy computed centrally.

We consider a system of $N$ players whose private states are denoted at time $t$ by $X^1_t$, $X^2_t$, $\cdots$, $X^N_t$. To keep the presentations simple, we assume the state space is $\mathbb{R}$. We denote by $\mu^N_t$ the empirical distribution of the states, namely:
\begin{equation*}
\mu^N_t=\frac1N\sum_{i=1}^N\delta_{X^i_t}.
\end{equation*}
We assume that these states evolve in continuous time under the influences of controls  $\alpha^1_t$, $\alpha^2_t$, $\cdots$ , $\alpha^N_t$ $\in \mathbb{A}$, where the set of admissible controls, $\mathbb{A}$, will be defined later. Let $\nu^N_{t}$ denote the empirical measure of the controls:
\begin{equation*}
\nu^N_t=\frac1N\sum_{i=1}^N\delta_{\alpha^i_t}.
\end{equation*}
We also assume that if and when interactions between these states and controls are present, they are of a mean field type, i.e. through $\mu^N_t$ and $\nu^N_t$. The time evolution of the state for player $i$ is given by the It\^{o} dynamics:
\begin{equation*}
    dX^i_t=b(t,X^i_t,\mu^N_t,\alpha^i_t,\nu^N_t)dt+\sigma dW_t.
\end{equation*}
We work over the interval $[0,T]$ limited by a finite time horizon $T \in \mathbb{R}^+$. We assume the drift function $b:[0,T] \times \mathbb{R} \times \cP(\mathbb{R}) \times \mathbb{A} \times \cP(\mathbb{A}) \ni (t,x,\mu,\alpha,\nu)  \rightarrow \mathbb{R}$ is Lipschitz in each of it's inputs. For the sake of simplicity, we assume that the volatility, $\sigma$, is a positive constant.

\subsubsection*{\textbf{Cost Functionals}}
We assume that we are given two functions $f:[0,T] \times \mathbb{R} \times \cP(\mathbb{R}) \times \mathbb{A} \times \cP(\mathbb{A}) \ni (t,x,\mu,\alpha, \nu) \rightarrow \mathbb{R}$ and $g:\mathbb{R}\times \cP(\mathbb{R}) \ni (x,\mu) \rightarrow \mathbb{R}$ which we call running and terminal cost functions, respectively. We assume $f$ and $g$ are Lipschitz in each of their arguments. The goal of player $i$ is to minimize their expected cost as given by:
\begin{equation*}
J^i(\balpha^1,\cdots,\balpha^N)=\EE\bigg[\int_0^Tf(t,X^i_t,\mu^N_{t},\alpha^i_t,\nu^N_{t})\,dt +g(X^i_T,\mu^{N}_{T})\bigg].
\end{equation*}

\subsubsection*{\textbf{Social Cost}}
We restrict ourselves to Markovian control strategies $\balpha=(\alpha_t)_{0\le t\le T}$ given by feedback functions in the form $\alpha_t=\phi(t,X_{t})$ and we let $\mathbb{A}$ denote the set of such controls. If the $N$ players use distributed Markovian control strategies of the form $\alpha^i_t=\phi(t,X^i_t)$, we define the cost (per player) to the system as the quantity $J^{(N)}_\phi$:
\begin{equation*}
J^{(N)}_\phi=\frac1N\sum_{i=1}^NJ^i(\balpha^1,\cdots,\balpha^N).
\end{equation*}
We shall compute this social cost in the limit $N\to\infty$ when all the players use the distributed control strategies given by the same feedback function $\phi$ identified by solving an optimization problem in the limit $N\to\infty$. We take the social cost to be the limit as $N\to\infty$ of $J^{(N)}_\phi$, namely:
\begin{equation*}
\begin{split}
\lim_{N\to\infty}J^{(N)}_\phi
&=\lim_{N\to\infty}\frac1N\sum_{i=1}^NJ^i(\balpha^1,\cdots,\balpha^N)\\
&=\lim_{N\to\infty}\frac1N\sum_{i=1}^N\EE\bigg[\int_0^Tf(t,X^i_t,\mu^{N}_{t},\phi(t,X^i_t),\nu^{N}_{t})\,dt +g(X^i_T,\mu^{N}_{T})
\bigg],\\
&=\lim_{N\to\infty}\EE\bigg[\int_0^T<f(t,\,\cdot\,,\mu^{N}_{X_t},\phi(t,\,\cdot\,),\nu^{N}_{t}),\mu^N_{t}>\,dt +<g(\,\cdot\,,\mu^{N}_{X_T}),\mu^{N}_{T}>
\bigg],
\end{split}
\end{equation*}
if we use the notation $<\varphi,\rho>$ for the integral $\int \varphi(z)\rho(dz)$ of the function $\varphi$ with respect to the measure $\rho$. Now if we assume that in the limit $N\to\infty$ the empirical distributions $\mu^N_{t}$ converge toward a measure $\mu_t$, and thus $\nu^N_{t}=\frac1N\sum_{i=1}^N\delta_{\phi(t,X^i_t)}$ also converges toward a measure $\nu_t$, then the social cost of the feedback function $\phi$ becomes:
\begin{equation*}
SC(\phi)=\int_0^T<f(t,\,\cdot\,,\mu_t,\phi(t,\,\cdot\,),\nu_t),\mu_t>\,dt \;+\; <g(\,\cdot\,,\mu_T),\mu_T>,
\end{equation*}
with the expectation, $\EE$, disappearing when the limiting flows $\bmu=(\mu_t)_{0\le t\le T}$ and $\bnu=(\nu_t)_{0\le t\le T}$ are deterministic.

We would like to evaluate $SC(\phi)$ in the $N=\infty$ regime directly, without having to construct the deterministic measure flows $\bmu$ and $\bnu$ as limits of the finite player empirical measures. To do this, we assume that propagation of chaos holds and that the states of the $N$ players become asymptotically independent in the limit as $N \rightarrow \infty$. We consider a representative agent whose state is given by $\boldsymbol{X^{\phi}}=(X^{\phi}_t)_{0 \leq t \leq T}$, the continuous time solution of the stochastic differential equation of McKean-Vlasov type:
\begin{equation}
    dX^{\phi}_t=b(t,X^{\phi}_t,\cL(X^{\phi}_t),\phi(t,X^{\phi}_t),\cL(\phi(t,X^{\phi}_t))dt+\sigma dW_t
\label{eq:X}
\end{equation}
controlled by $\phi$. Then we can identify $\bmu$ as the law of a representative agent using the feedback function $\phi$, i.e. $\mu_t=\cL(X^{\phi}_t)$, and similarly, we can identify $\bnu$ as the law of the control, such that $\nu_t=\cL(\phi(t,X^{\phi}_t))$. Thus, in the $N=\infty$ regime, we rewrite the social cost as:
\begin{equation*}
SC(\phi)=\int_0^T<f(t,\,\cdot\,,\cL(X^{\phi}_t),\phi(t,\,\cdot\,),\cL(\phi(t,X^{\phi}_t))),\cL(X^{\phi}_t)>\,dt+<g(\,\cdot\,,\cL(X^{\phi}_T)),\cL(X^{\phi}_T)>,
\end{equation*}
where $\boldsymbol{X^{\phi}}$ satisfies equation (\ref{eq:X}). For the remainder of the paper, we work in the $N=\infty$ regime. As mentioned earlier, $\phi$ should be identified by solving an optimal control problem. We consider two distinct problems:
\begin{itemize}\itemsep=-2pt
\item $\phi$ is a feedback function providing a mean field game equilibrium. We detail more precisely what is meant by $\phi$ providing a mean field game equilibrium in section \ref{sub:MFG_formulation}.
\item $\phi$ is the feedback function minimizing the social cost $SC(\phi)$, without having to be a mean field game equilibrium, in which case we use the notation $SC^{MKV}$ for $SC(\phi)$. This is a control problem of McKean-Vlasov type, which is detailed more precisely in section \ref{sub:MKV_formulation}.
\end{itemize}
The two problems are detailed more precisely in sections \ref{sub:MFG_formulation} and \ref{sub:MKV_formulation}. In section \ref{Price_Of_Anarchy}, we define the price of anarchy based on these two problem formulations. The class of linear quadratic models is explored in section \ref{sec:PoA_LQ}, where we provide some theoretical results on the price of anarchy for this class of games. This includes our main result, Theorem \ref{thm:PoA_equivalent_1}, which provides a sufficient and necessary condition to have no price of anarchy. In section \ref{sec:PoA_LQ}, we also prove some limiting cases and show numerical results. We conclude in section \ref{sec:conclusion}.

\subsection{\textbf{Nash Equilibrium: Mean Field Game Formulation}}
\label{sub:MFG_formulation}
The goal of this subsection is to articulate what is meant by a feedback function providing a mean field game equilibrium. To begin, we define what we call the \textit{mean field environment}. By symmetry of the players, we suppose all of the players in the mean field game use the same feedback function, $\phi$. Then the \textit{mean field environment} specified by $\phi$ is characterized by $\cL(X^{\phi}_t)_{0\le t\le T}$ and $\cL(\phi(t,X^{\phi}_t))_{0\le t\le T}$ where the dynamics of $(X^{\phi}_t)_{0\le t\le T}$ are given by equation (\ref{eq:X}). Since we search for a Nash equilibrium, we consider a representative agent who wishes to find their best response, $\phi'$, to the mean field environment specified by $\phi$, in which case their state is given by $\boldsymbol{X^{\phi',\phi}}=(X^{\phi',\phi}_t)_{0\le t\le T}$ solving the standard stochastic differential equation:
\begin{equation*}
    dX^{\phi',\phi}_t=b(t,X^{\phi',\phi}_t,\cL(X^{\phi}_t),\phi'(t,X^{\phi',\phi}_t),\cL(\phi(t,X^{\phi}_t)))dt+\sigma dW_t.
\end{equation*}
Consider the function:
\begin{equation*}
    \cS(\phi',\phi)=\biggl[\int_0^T<f(t,\cdot,\cL(X^{\phi}_t),\phi'(t,\cdot),\cL(\phi(t,X^{\phi}_t)),\cL(X^{\phi',\phi}_t))> dt +<g(\cdot,\cL(X^{\phi}_T)),\cL(X^{\phi',\phi}_t)>\biggr].
\end{equation*}
The best response for the representative agent in the mean field environment specified by $\phi$ is the feedback function minimizing this cost, namely $\phi^*=\arg\inf_{\phi'}\cS(\phi',\phi)$. Assuming the minimizer is unique (which will be the case for the models we consider), this defines a mapping $\Phi:\phi \rightarrow \phi^*$. If there is a $\hat{\phi}$ such that $\Phi(\hat{\phi})=\hat{\phi}$, then the players are in a mean field game equilibrium.

Thus, the search for a feedback function providing a mean field game equilibrium can be summarized as the following set of two successive steps:
\begin{enumerate}
\item For each feedback function $\phi:[0,T] \times \mathbb{R} \ni (t,x) \rightarrow \mathbb{R}$, solve the optimal control problem
\begin{equation*}
    \phi^*=\arg \inf_{\phi'}\cS(\phi',\phi).
\end{equation*}
Define the mapping $\Phi(\phi):=\phi^*$.
\item Find a fixed point $\hat{\phi}$ of $\Phi$ such that $\Phi(\hat{\phi})=\hat{\phi}$.
\end{enumerate}

When these two steps can be taken successfully, we say that $\hat{\phi}$ provides a mean field game equilibrium. Note that $\boldsymbol{X^{\hat{\phi},\hat{\phi}}}=\boldsymbol{X^{\hat{\phi}}}$ and therefore $\cS(\hat{\phi},\hat{\phi})=SC(\hat{\phi})$ gives the social cost for the mean field game equilibrium provided by $\hat{\phi}$. Notice that there could possibly be many feedback functions providing a mean field game equilibrium. Let $\mathcal{N}$ denote the set of all such feedback functions providing mean field game equilibria, as detailed above, i.e.
\begin{equation*}
    \mathcal{N}=\{\phi:[0,T] \times \mathbb{R} \ni (t,x) \rightarrow \mathbb{R}\ |\ \Phi(\phi)=\phi \}.
\end{equation*}

\subsection{\textbf{Centralized Control: Optimal Control of McKean-Vlasov Type}}\label{sub:MKV_formulation}
The goal of this subsection is to articulate how to compute the cost associated with the control problem of McKean-Vlasov type, $SC^{MKV}$. The central planner considers the following control problem:
\begin{equation*}
\begin{split}
   \hat{\phi}&=\arg \inf_{\phi}SC(\phi) \\
   &=\arg \inf_{\phi}\left[\int_0^T<f(t,\,\cdot\,,\cL(X^{\phi}_t),\phi(t,\,\cdot\,),\cL(\phi(t,X^{\phi}_t))),\cL(X^{\phi}_t)>\,dt \;+\; <g(\,\cdot\,,\cL(X^{\phi}_T)),\cL(X^{\phi}_T)>\right].
\end{split}
\end{equation*}
Thus, the cost of the solution to the optimal control problem of McKean-Vlasov is given by:
\begin{equation*}
    SC^{MKV}=SC(\hat{\phi}).
\end{equation*}
\begin{remark}
We are not concerned with uniqueness for the control of McKean-Vlasov type problem, because $SC^{MKV}=SC(\phi_1)=SC(\phi_2)$ is still well defined even if there are two different optimal feedback functions $\phi_1$ and $\phi_2$ minimizing $SC(\phi)$.
\end{remark}

\subsection{\textbf{Price of Anarchy}}\label{Price_Of_Anarchy}
We have described two approaches to compute the optimal feedback function $\phi$. In the mean field game formulation, we require $\phi \in \mathcal{N}$, where $\mathcal{N}$ denotes the set of feedback functions providing mean field game equilibria. In the optimal control of McKean-Vlasov type formulation, the optimal control to be adopted by all players is computed by a central planner, who optimizes the social cost function $SC(\phi)$ directly. Thus, we necessarily have:
\begin{equation*}
    SC^{MKV} \leq SC(\phi),\ \forall \phi \in \mathcal{N}.
\end{equation*}
In other words, there is a `price of anarchy' associated with allowing players to choose their controls selfishly. We thus define the price of anarchy (denoted $PoA$) as the ratio between the worst case cost for a mean field game equilibrium and the optimal cost computed by a central planner:
\begin{equation*}
    PoA=\frac{\sup_{\phi \in \mathcal{N}}SC(\phi)}{SC^{MKV}}.
\end{equation*}

\section{\textbf{Price of Anarchy for Linear Quadratic Extended Mean Field Games}}\label{sec:PoA_LQ}
The class of linear quadratic extended mean field games is a class of problems for which explicit solutions can be computed analytically, and thus, we can compute the price of anarchy explicitly. To the best of our knowledge, the case of linear quadratic extended mean field games has not been explored in the literature, as well as computing the price of anarchy for this class of games.

To begin, we need to describe in more detail the two problems that will be used to compute the price of anarchy: the linear quadratic extended mean field game, and the linear quadratic control problem of McKean-Vlasov type with dependence on the law of the control. To specify the problems, we only need to specify the drift and cost functions, $b$, $f$, and $g$ introduced in section \ref{se:introduction}. For the linear quadratic models, we take the drift to be linear:
\begin{equation*}
    b(t,x,\mu,\alpha,\nu)=b_1(t)x+\bar{b}_1(t) \bar{\mu}+b_2(t) \alpha+\bar{b}_2(t) \bar{\nu},
\end{equation*}
where $\bar{\mu}$ denotes the mean of the measure $\mu$, namely, $\bar{\mu}=\int_{\mathbb{R}}xd\mu(x)$, and similarly for $\bar{\nu}$. We take the running and terminal costs to be quadratic:
\begin{equation*}
\begin{split}
    f(t,x,\mu,\alpha,\nu)&=\frac{1}{2}\left[q(t)x^2+\bar{q}(t)(x-s(t)\bar{\mu})^2 +r(t)\alpha^2+\bar{r}(t)(\alpha-\bar{s}(t)\bar{\nu})^2\right], \\
    g(x,\mu)&=\frac{1}{2}\left[q_T x^2+\bar{q}_T (x-s_T \bar{\mu})^2\right].
\end{split}
\end{equation*}
\begin{remark}
If $\bar{b}_2(t)\equiv0$ and $\bar{r}(t)\equiv0$, then we have the standard mean field game or control problem of McKean-Vlasov type. (See Theorem \ref{th:exist_uniq} for assumptions that provide existence and uniqueness.)
\label{remark_2}
\end{remark}

\subsection{\textbf{Linear Quadratic Extended Mean Field Games}}\label{sec:EMFG}
To solve the linear quadratic extended mean field game (LQEMFG), we begin by considering the reduced Hamiltonian for this problem:
\begin{equation*}
\begin{split}
    H(t,x,\bar{\mu},\alpha,\bar{\nu},y)=&\left[b_1(t)x+\bar{b}_1(t)\bar{\mu}+b_2(t) \alpha+\bar{b}_2(t)\bar{\nu}\right]y \\
    &+\frac{1}{2}\left[q(t)x^2+\bar{q}(t)(x-s(t)\bar{\mu})^2 +r(t)\alpha^2+\bar{r}(t)(\alpha-\bar{s}(t)\bar{\nu})^2\right],
\end{split}
\end{equation*}
and whenever the flows $\bmubar=(\bar{\mu}_t)_{0\leq t\leq T}$ and $\bnubar=(\bar{\nu}_t)_{0\leq t\leq T}$ are fixed, we consider for each control process $\balpha=(\alpha_t)_{0\leq t\leq T}$ the adjoint equation:
\begin{equation*}
    dY_t=- \partial_x H(t,X_t,\bar{\mu}_t,\alpha_t,\bar{\nu}_t,Y_t)dt+Z_tdW_t,
    \qquad
    Y_T=\partial_xg(X_T, \cL(X_T)).
\end{equation*}
According to the Pontryagin stochastic maximum principle, a sufficient condition for optimality is $\partial_{\alpha}H(t,X_t,\bar{\mu}_t,\hat{\alpha}_t,\bar{\nu}_t,Y_t)=0$. Thus, we find the optimal control:
\begin{equation}
    \hat{\alpha}_t=\frac{\bar{r}(t)\bar{s}(t)\bar{\nu}-b_2(t)Y_t}{r(t)+\bar{r}(t)}.
\label{eq:MFG_opt_alpha_Pontryagin}
\end{equation}
When solving the fixed point step, we identify $\bar{\nu}_t=\mathbb{E}(\hat{\alpha}_t)$. By taking the expectation, we find:
\begin{equation*}
    \bar{\nu}_t = \mathbb{E}(\hat{\alpha}_t)=c^{MFG}(t)\mathbb{E}(Y_t),
	\qquad 
	\text{with} 
	\quad
    c^{MFG}(t)=-\frac{b_2(t)}{r(t)+\bar{r}(t)(1-\bar{s}(t))}.
\end{equation*}
Thus, from equation (\ref{eq:MFG_opt_alpha_Pontryagin}) we have:
\begin{equation}
    \hat{\alpha}_t=a^{MFG}(t) Y_t+b^{MFG}(t)\mathbb{E}(Y_t),
\label{eq:MFG_opt_alpha_Y_expY}
\end{equation}
with:
\begin{equation*}
	a^{MFG}(t)=-\frac{b_2(t)}{r(t)+\bar{r}(t)},
	\qquad
	\text{and}
	\qquad
	b^{MFG}(t)=-\frac{\bar{r}(t)\bar{s}(t)b_2(t)}{(r(t)+\bar{r}(t))(r(t)+\bar{r}(t)(1-\bar{s}(t)))}.
\end{equation*}
Note that $c^{MFG}(t)=a^{MFG}(t)+b^{MFG}(t)$. The solution of the mean field game equilibrium problem is given by the solution to the FBSDE system:
\begin{equation}
\left\{
\begin{array}{lcl}
      \displaystyle  dX_t&=& \left[b_1(t)X_t+\bar{b}_1(t) \mathbb{E}X_t+a^{MFG}(t)b_2(t)Y_t+(b^{MFG}(t)b_2(t)+c^{MFG}(t)\bar{b}_2(t))\mathbb{E}Y_t\right]dt+\sigma dW_t \\[5pt]
        dY_t&=& -\left[(q(t)+\bar{q}(t))X_t-\bar{q}(t)s(t)\mathbb{E}X_t+b_1(t)Y_t \right]dt +Z_t dW_t,
\end{array}
\right.
\label{eq:FBSDE_EMFG}
\end{equation}
with initial condition $X_0=\xi$, a random variable with finite mean and variance, and terminal condition $Y_T=(q_T+\bar{q}_T)X_T-\bar{q}_Ts_T\mathbb{E}X_T$.

This is a linear FBSDE of McKean-Vlasov type, which can be solved explicitly under mild assumptions (or at least in the case of time-independent coefficients which we will consider later. See Appendix \ref{appendix_solving_fbsdes}). Let $\bar{\eta}_t^{MFG}$, $\eta_t^{MFG}$, $\bar{x}_t^{MFG}$, and $v^{MFG}_t$ denote the solutions for this problem as described in the appendix so that:
\begin{align*}
Y_t =\eta_t^{MFG}X_t+(\bar{\eta}_t^{MFG}-\eta_t^{MFG})\bar{x}_t^{MFG}, \quad &\mathbb{E}(Y_t) =\bar{\eta}_t^{MFG}\bar{x}_t^{MFG},\\
\mathbb{E}(X_t) =\bar{x}_t^{MFG}, \quad &Var(X_t)=v^{MFG},
\end{align*}
provide a solution to the LQEMFG problem. Then from the appendix, we have:
\begin{itemize}
	\item a scalar Riccati equation for $\bar{\eta}_t^{MFG}$:
	\begin{equation}
	\dot{\bar{\eta}}^{MFG}_t+ \left[ c^{MFG}(t)(b_2(t)+\bar{b}_2(t)) \right] \cdot (\bar{\eta}^{MFG}_t)^2+ \left(2b_1(t)+\bar{b}_1(t) \right) \cdot \bar{\eta}^{MFG}_t +q(t)+\bar{q}(t)(1-s(t)) =0,
	\label{eq:eta_bar_MFG}
	\end{equation}
	with terminal condition $\bar{\eta}^{MFG}_T = q_T+\bar{q}_T(1-s_T)$, 
	\item a linear first order ODE for $\bar{x}_t^{MFG}$:
	\begin{equation}
	\dot{\bar{x}}^{MFG}_t=\left[ b_1(t)+\bar{b}_1(t)+c^{MFG}(t)(b_2(t)+\bar{b}_2(t)) \cdot \bar{\eta}^{MFG}_t \right] \cdot \bar{x}^{MFG}_t,
	\label{eq:x_bar_MFG}
	\end{equation}
	with initial condition $\bar{x}^{MFG}_0=\mathbb{E}(\xi)$,
	\item a scalar Riccati equation for $\eta_t^{MFG}$:
	\begin{equation}
	\dot{\eta}^{MFG}_t+a^{MFG}(t)b_2(t)\cdot (\eta^{MFG}_t)^2 + 2b_1(t) \cdot \eta^{MFG}_t+q(t)+\bar{q}(t)=0,
	\end{equation}
	with terminal condition $\eta^{MFG}_T = q_T+\bar{q}_T$,
\end{itemize}
and where the dot is the standard ODE notation for a derivative. 
And thus, we obtain explicit solutions for $\bar{x}_t^{MFG}$ and $v_t^{MFG}$:
\begin{IEEEeqnarray}{rCl}
	\bar{x}^{MFG}_t &=& \mathbb{E}(\xi) e^{ \int_0^t\left( b_1(s)+\bar{b}_1(s)+ \left[ c^{MFG}(s)(b_2(s)+\bar{b}_2(s))\right] \cdot \bar{\eta}^{MFG}_s\right) ds} \label{eq:x_bar_MFG_explicit}, \\ 
	v^{MFG}_t &=& Var(\xi)e^{\int_0^t 2\left[b_1(s)+a^{MFG}(s)b_2(s) \cdot \eta^{MFG}_s\right]ds}+\sigma^2 \int_0^t e^{2 \int_s^t \left[b_1(u)+a^{MFG}(u)b_2(u)\eta^{MFG}_u\right] du}ds.
	\label{eq:v_t}
\end{IEEEeqnarray}
Let $SC^{MFG}:=SC(\phi^{MFG})$ in which $\phi^{MFG}=\phi^{MFG}(t,x)$ is the feedback function specified by this solution, namely, from equation (\ref{eq:MFG_opt_alpha_Y_expY}), we have:
\begin{equation*}
    \phi^{MFG}(t,x)=a^{MFG}(t)\eta_t^{MFG}x+\left[a^{MFG}(t)(\bar{\eta}_t^{MFG}-\eta_t^{MFG})+b^{MFG}(t)\bar{\eta}_t^{MFG} \right]\bar{x}_t^{MFG}.
\end{equation*}
Then we can compute the social cost as described in section \ref{sub:MFG_formulation}:
\begin{IEEEeqnarray}{rCl}
    SC^{MFG}&=&\frac{1}{2}\Big[(q_T+\bar{q}_T)v_T^{MFG}+(q_T+\bar{q}_T(1-s_T)^2)(\bar{x}_T^{MFG})^2 \nonumber \\
    && \quad +\int_0^T \left[q(t)+\bar{q}(t)+(r(t)+\bar{r}(t))(a^{MFG}(t)\eta_t^{MFG})^2\right]v_t^{MFG}dt
    \label{eq:SC_MFG_0} \\
    && \quad +\int_0^T\left[q(t)+\bar{q}(t)(1-s(t))^2+(r(t)+\bar{r}(t)(1-\bar{s}(t))^2)(c^{MFG}(t)\bar{\eta}_t^{MFG})^2\right](\bar{x}_t^{MFG})^2dt\Big], \nonumber
\end{IEEEeqnarray}
where we have used the fact that:
\begin{equation*}
    \mathbb{E}(\phi^{MFG}(t,X_t))=c^{MFG}(t)\cdot \bar{\eta}_t^{MFG}\bar{x}_t^{MFG}
	\qquad 
	\text{and}
	\qquad
    Var(\phi^{MFG}(t,X_t))=\left(a^{MFG}(t)\eta_t^{MFG} \right)^2 \cdot v^{MFG}_t.
\end{equation*}

\subsection{\textbf{Linear Quadratic Control of McKean-Vlasov Type Involving the Law of the Control}}\label{sec:EMKV}
To solve the linear quadratic optimal control problem of McKean-Vlasov type involving the law of the control (LQEMKV), we begin with the reduced Hamiltonian, which is the same as in the LQEMFG problem:
\begin{equation*}
\begin{split}
    H(t,x,\bar{\mu},\alpha,\bar{\nu},y)=&\left[b_1(t)x+\bar{b}_1(t)\bar{\mu}+b_2(t) \alpha+\bar{b}_2(t)\bar{\nu}\right]y\\
    &+\frac{1}{2}\left[q(t)x^2+\bar{q}(t)(x-s(t)\bar{\mu})^2 +r(t)\alpha^2+\bar{r}(t)(\alpha-\bar{s}(t)\bar{\nu})^2\right].
\end{split}
\end{equation*}
Since we require $\bar{\nu}_t$ to be equal to $\mathbb{E}(\alpha_t)$ throughout the optimization, it is not sufficient to minimize the Hamiltonian with respect to the $\alpha$ input alone in order to guarantee optimality. A sufficient condition for control problems of McKean-Vlasov type involving the law of the control is derived in \cite{carmona_acciaio}. Since we consider a Hamiltonian that depends on the means of $\bar{\mu}$ and $\bar{\nu}$ instead of the full distributions, the sufficient condition reduces to the following (see section 4 in \cite{carmona_acciaio}):
\begin{equation*}
    \partial_{\alpha}H(t,X_t,\mathbb{E}(X_t),\hat{\alpha}_t,\mathbb{E}(\hat{\alpha}_t),Y_t)+\tilde{\mathbb{E}} \left[\partial_{\bar{\nu}}H(t,\tilde{X}_t,\mathbb{E}(X_t),\hat{\alpha}_t,\mathbb{E}(\hat{\alpha}_t),\tilde{Y}_t) \right]=0,
\end{equation*}
where the adjoint equation is given by:
\begin{equation*}
\left\{
\begin{array}{lcl}
    dY_t&=&-\left[\partial_x H(t,X_t,\bar{\mu}_t,\alpha_t,\bar{\nu}_t,Y_t)+\tilde{\mathbb{E}} \left[\partial_{\bar{\mu}}H(t,\tilde{X}_t,\bar{\mu}_t,\tilde{\alpha}_t,\bar{\nu}_t,\tilde{Y}_t) \right] \right]dt+Z_tdW_t \\[5pt]
    Y_T&=&\partial_xg(X_T, \cL(X_T))+\tilde{\mathbb{E}} \left[\partial_{\bar{\mu}}g(\tilde{X}_T,\cL(X_T))(X_T) \right],
\end{array}
\right.
\end{equation*}
and where $(\boldsymbol{\tilde{X}},\boldsymbol{\tilde{Y}},\boldsymbol{\tilde{\alpha}})$ denotes an independent copy of $(\boldsymbol{X},\boldsymbol{Y},\boldsymbol{\alpha})$. In the present LQ case, the sufficient condition can be used to solve for:
\begin{equation}
    \hat{\alpha}_t=a^{MKV}(t)Y_t+b^{MKV}(t)\mathbb{E}(Y_t),
    \label{eq:alpha_hat_MKV}
\end{equation}
with:
\begin{equation*}
\begin{split}
    a^{MKV}(t)&=-\frac{b_2(t)}{r(t)+\bar{r}(t)}, \quad \text{and} \quad b^{MKV}(t)=-\frac{1}{r(t)+\bar{r}(t)}\left(\bar{b}_2(t)-\frac{\bar{r}(t)\bar{s}(t)(\bar{s}(t)-2)(b_2(t)+\bar{b}_2(t))}{r(t)+\bar{r}(t)(1-\bar{s}(t))^2} \right).
\end{split}
\end{equation*}
Then $\mathbb{E}(\hat{\alpha}_t)=c^{MKV}(t)\mathbb{E}(Y_t)$ with:
   $$ c^{MKV}(t)=a^{MKV}(t)+b^{MKV}(t)=-\frac{b_2(t)+\bar{b}_2(t)}{r(t)+\bar{r}(t)(1-\bar{s}(t))^2}.$$

So the solution of the optimal control problem of McKean-Vlasov type is given by the solution to the FBSDE system:
\begin{equation}
\left\{
\begin{split}
        dX_t=&\left[b_1(t)X_t+\bar{b}_1(t) \mathbb{E}X_t+a^{MKV}(t)b_2(t)Y_t+(b^{MKV}(t)b_2(t)+c^{MKV}(t)\bar{b}_2(t))\mathbb{E}Y_t\right]dt+\sigma dW_t \\[5pt]
        dY_t=&-\left[(q(t)+\bar{q}(t))X_t+s(t)\bar{q}(t)(s(t)-2)\mathbb{E}X_t+b_1(t)Y_t+\bar{b}_1(t)\mathbb{E}Y_t\right]dt +Z_t dW_t,
\end{split}
\right.
\label{eq:FBSDE_EMKV}
\end{equation}
with initial condition $X_0=\xi$, and terminal condition $Y_T=(q_T+\bar{q}_T)X_T+s_T\bar{q}_T(s_T-2)\mathbb{E}X_T$.

As in the previous section, this is a linear FBSDE of McKean-Vlasov type, which can be solved explicitly under mild assumptions (or at least in the case of time-independent coefficients which we will consider later. See Appendix \ref{appendix_solving_fbsdes}). Let $\bar{\eta}_t^{MKV}$, $\eta_t^{MKV}$, $\bar{x}_t^{MKV}$, and $v^{MKV}_t$ denote the solutions for this problem as described in the appendix so that:
\begin{align*}
Y_t =\eta_t^{MKV}X_t+(\bar{\eta}_t^{MKV}-\eta_t^{MKV})\bar{x}_t^{MKV}, \quad &\mathbb{E}(Y_t) =\bar{\eta}_t^{MKV}\bar{x}_t^{MKV},\\
\mathbb{E}(X_t) =\bar{x}_t^{MKV}, \quad &Var(X_t)=v^{MKV},
\end{align*}
provide a solution to the LQEMKV problem. Then from the appendix, we have:
\begin{itemize}
	\item a scalar Riccati equation for $\bar{\eta}_t^{MKV}$:
	\begin{equation}
	\dot{\bar{\eta}}^{MKV}_t+ \left[ c^{MKV}(t)(b_2(t)+\bar{b}_2(t)) \right] \cdot (\bar{\eta}^{MKV}_t)^2+ 2\left(b_1(t)+\bar{b}_1(t) \right) \cdot \bar{\eta}^{MKV}_t +q(t)+\bar{q}(t)(1-s(t))^2 =0,
	\label{eq:eta_bar_MKV}
	\end{equation}
	with terminal condition $\bar{\eta}^{MKV}_T = q_T+\bar{q}_T(1-s_T)^2$, 
	\item a linear first order ODE for $\bar{x}_t^{MKV}$:
	\begin{equation}
	\dot{\bar{x}}^{MKV}_t=\left[ b_1(t)+\bar{b}_1(t)+c^{MKV}(t)(b_2(t)+\bar{b}_2(t)) \cdot \bar{\eta}^{MKV}_t \right] \cdot \bar{x}^{MKV}_t,
	\label{eq:x_bar_MKV}
	\end{equation}
	with initial condition $\bar{x}^{MKV}_0=\mathbb{E}(\xi)$,
	\item a scalar Riccati equation for $\eta_t^{MKV}$:
	\begin{equation}
	\dot{\eta}^{MKV}_t+a^{MKV}(t)b_2(t)\cdot (\eta^{MKV}_t)^2 + 2b_1(t) \cdot \eta^{MKV}_t+q(t)+\bar{q}(t)=0,
	\end{equation}
	with terminal condition $\eta^{MKV}_T = q_T+\bar{q}_T$,
\end{itemize}
and where the dot is the standard ODE notation for a derivative. 
And thus, we obtain explicit solutions for $\bar{x}_t^{MKV}$ and $v_t^{MKV}$:
\begin{IEEEeqnarray}{rCl}
	\bar{x}^{MKV}_t &=& \mathbb{E}(\xi) e^{ \int_0^t\left( b_1(s)+\bar{b}_1(s)+ \left[ c^{MKV}(s)(b_2(s)+\bar{b}_2(s))\right] \cdot \bar{\eta}^{MKV}_s\right) ds}, \label{eq:x_bar_MKV_explicit} \\ 
	v^{MKV}_t &=& Var(\xi)e^{\int_0^t 2\left[b_1(s)+a^{MKV}(s)b_2(s) \cdot \eta^{MKV}_s\right]ds}+\sigma^2 \int_0^t e^{2 \int_s^t \left[b_1(u)+a^{MKV}(u)b_2(u)\eta^{MKV}_u\right] du}ds.
	\label{eq:v_t_MKV}
\end{IEEEeqnarray}
Then $SC^{MKV}=SC(\phi^{MKV})$ where $\phi^{MKV}$ is the feedback function specified by this solution, namely, from equation (\ref{eq:alpha_hat_MKV}), we have:
\begin{equation*}
    \phi^{MKV}(t,x)=a^{MKV}(t)\eta_t^{MKV}x+\left[a^{MKV}(t)(\bar{\eta}_t^{MKV}-\eta_t^{MKV})+b^{MKV}(t)\bar{\eta}_t^{MKV} \right]\bar{x}_t^{MKV}.
\end{equation*}
Then we can compute the social cost, denoted $SC^{MKV}$, as described in section \ref{sub:MKV_formulation}:
\begin{IEEEeqnarray}{rCl}
    SC^{MKV}&=&\frac{1}{2}\Big[(q_T+\bar{q}_T)v_T^{MKV}+(q_T+\bar{q}_T(1-s_T)^2)(\bar{x}_T^{MKV})^2 \nonumber \\
    && \quad +\int_0^T \left[q(t)+\bar{q}(t)+(r(t)+\bar{r}(t))(a^{MKV}(t)\eta_t^{MKV})^2\right]v_t^{MKV}dt \label{eq:SC_MKV_0} \\
    && \quad +\int_0^T\left[q(t)+\bar{q}(t)(1-s(t))^2+(r(t)+\bar{r}(t)(1-\bar{s}(t))^2)(c^{MKV}(t)\bar{\eta}_t^{MKV})^2\right](\bar{x}_t^{MKV})^2dt\Big], \nonumber
\end{IEEEeqnarray}
where we have used the fact that:
\begin{equation*}
\mathbb{E}(\phi^{MKV}(t,X_t))=c^{MKV}(t)\cdot \bar{\eta}_t^{MKV}\bar{x}_t^{MKV}
\qquad 
\text{and}
\qquad
Var(\phi^{MKV}(t,X_t))=\left(a^{MKV}(t)\eta_t^{MKV} \right)^2 \cdot v^{MKV}_t.
\end{equation*}

\subsection{\textbf{Theoretical Results}}
For the remainder of the paper, we assume the coefficients are independent of time and non-negative:
\begin{equation*}
    (b_1(t),\bar{b}_1(t),b_2(t),\bar{b}_2(t),q(t),\bar{q}(t),r(t),\bar{r}(t),s(t),\bar{s}(t))=(b_1,\bar{b}_1,b_2,\bar{b}_2,q,\bar{q},r,\bar{r},s,\bar{s})\in (\mathbb{R}^+)^{10},
\end{equation*}
$$    (q_T,\bar{q}_T,s_T)\in(\mathbb{R}^+)^{3},$$
and therefore,
\begin{equation*}
\begin{split}
    \left(a^{MFG}(t), b^{MFG}(t), c^{MFG}(t) \right)&=\left(a^{MFG}, b^{MFG}, c^{MFG} \right) \\
    \left(a^{MKV}(t), b^{MKV}(t), c^{MKV}(t) \right)&=\left(a^{MKV}, b^{MKV}, c^{MKV}\right).
\end{split}
\end{equation*}
Also, it will be convenient to denote:
\begin{equation}
\lambda := \frac{c^{MFG}}{c^{MKV}} = \frac{b_2}{b_2 + \bar{b}_2}\cdot \frac{ r + \bar{r}(1-\bar{s})^2 }{r + \bar{r}(1-\bar{s})},  \qquad u_t := \lambda \bar{\eta}_t^{MFG}, \qquad w_t := \bar{\eta}_t^{MKV}.
\label{eq:lambda}
\end{equation}
and to make the following observations:
\begin{equation}
a^{MFG}=a^{MKV}=:a, \qquad
\eta_t^{MFG} = \eta_t^{MKV}=:\eta_t, \qquad
v_t^{MFG}=v_t^{MKV}=:v_t.
\label{eq:observations}
\end{equation}

\begin{theorem} \label{th:exist_uniq}
Assume the following:
\begin{equation}
\begin{array}{rclcrclcrcl}
    b_2 &>& 0 & & b_2 + \bar{b}_2 &>& 0 & & \\
    r+\bar{r} &>&0 &&  r+\bar{r}(1-\bar{s}) &>&0 & & r+\bar{r}(1-\bar{s})^2 &>& 0 \\
    q+\bar{q} &>& 0 && q + \bar{q}(1-s) &>& 0 & & q+ \bar{q}(1-s)^2 &>& 0 \\
    q_T + \bar{q}_T &\geq& 0 & \quad & q_T + \bar{q}_T(1-s_T) &\geq& 0 & \quad & q_T + \bar{q}_T(1-s_T)^2 &\geq& 0.
\end{array}
\label{eq:assumptions}
\end{equation}
Then there exists a unique solution to the LQEMFG problem, and there exists a unique solution to the LQEMKV problem. And therefore, $PoA=\frac{SC^{MFG}}{SC^{MKV}}$ where $SC^{MFG}$ and $SC^{MKV}$ are given by equations (\ref{eq:SC_MFG_0}) and (\ref{eq:SC_MKV_0}), respectively.
\end{theorem}

\begin{remark}
    Note that existence in Theorem \ref{th:exist_uniq} follows from the explicit construction in Appendix \ref{appendix_solving_fbsdes}, because the above conditions provide existence to the solutions of the Riccati equations. Uniqueness comes from the connection between LQEMFG or LQEMKV and deterministic LQ optimal control. (See section 3.5.1 in \cite{Carmona_book}).
\end{remark}

\begin{proposition}
Assuming (\ref{eq:assumptions}), if furthermore,
\begin{IEEEeqnarray}{rCl}
    \left[\bar{b}_1\bar{\eta}_t^{MKV}+s\bar{q}(s-1)\right] \cdot \bar{x}_t^{MKV} &=& 0,\quad \forall \  t \in [0,T] 
    \label{eq:prop_1_1} \\ [3pt]
    \left(c^{MFG}-c^{MKV} \right) \cdot \bar{\eta}_t^{MKV}\bar{x}_t^{MKV}&=& 0,\quad \forall \  t \in [0,T] 
    \label{eq:prop_1_2} \\ [3pt]
    s_T\bar{q}_T(s_T-1) \cdot \bar{x}_T^{MKV}&=& 0, 
    \label{eq:prop_1_3}
\end{IEEEeqnarray}
then $PoA=1$.
\label{prop: lq_subtract_FBSDEs}
\end{proposition}

\begin{proof}
    Comparing the FBSDE systems (\ref{eq:FBSDE_EMFG}) and (\ref{eq:FBSDE_EMKV}), and using the fact that $a^{MFG}=a^{MKV}$ and $b^{MFG}-b^{MKV}=c^{MFG}-c^{MKV}$, the result is clear.
\end{proof}

\begin{remark}
	Recall from Remark \ref{remark_2} that in the standard mean field game, $\bar{b}_2=\bar{r}=0$, and thus, $\lambda=1$. Although Proposition \ref{prop: lq_subtract_FBSDEs} is a simple result, we will see shortly in Corollary \ref{corollary_3} that in the case when $\lambda=1$, the sufficient condition given by equations (\ref{eq:prop_1_1})-(\ref{eq:prop_1_3}) is also a necessary condition to have $PoA=1$. We can see that in the standard mean field game setting, Proposition \ref{prop: lq_subtract_FBSDEs} is similar to Theorem 3.4 in \cite{pierre_poa} which characterizes the global efficiency of mean field game equilibria in the case of a separated Hamiltonian. See also Remark 6.1 in \cite{nourian2013nash}, where it is noted that the mean field game and control of McKean-Vlasov type problems are the same for a particular model of flocking.
\end{remark}

\begin{corollary}
	Assuming (\ref{eq:assumptions}), if furthermore, $\bar{b}_1=0$, $s\bar{q}(s-1)=0$, $s_T\bar{q}_T(s_T-1)=0$, and $c^{MFG} = c^{MKV}$,
	then $PoA=1$.
	\label{corollary_1}
\end{corollary}

\begin{corollary}
	Assuming (\ref{eq:assumptions}), if the initial condition $\xi$ is such that $\mathbb{E}(\xi)=0$, then from equation (\ref{eq:x_bar_MKV_explicit}), $x^{MKV}_t =0$ for all $t\in[0,T]$, and thus, $PoA=1$.
	\label{corollary_2}
\end{corollary}

Using the observations in equation (\ref{eq:observations}), we can rewrite:
\begin{IEEEeqnarray}{rCl}
    SC^{MFG} &=& \frac{1}{2}(q_T+\bar{q}_T)v_T+\frac{1}{2}\left(q_T+\bar{q}_T(1-s_T)^2\right)(\bar{x}_T^{MFG})^2+\frac{1}{2}\int_0^T \left[q+\bar{q}+(r+\bar{r})(a\eta_t)^2\right] v_t dt \nonumber \\
    && +\frac{1}{2}\int_0^T\left[q+\bar{q}(1-s)^2+(r+\bar{r}(1-\bar{s})^2)(c^{MFG}\bar{\eta}_t^{MFG})^2\right](\bar{x}_t^{MFG})^2dt,
    \label{eq:SC_MFG_1} \\ [8pt]
    SC^{MKV}&=&\frac{1}{2}(q_T+\bar{q}_T)v_T+\frac{1}{2}\left(q_T+\bar{q}_T(1-s_T)^2\right)(\bar{x}_T^{MKV})^2+ \frac{1}{2}\int_0^T \left[q+\bar{q}+(r+\bar{r})(a\eta_t)^2\right]v_t dt \nonumber \\
    &&+\frac{1}{2}\int_0^T\left[q+\bar{q}(1-s)^2+(r+\bar{r}(1-\bar{s})^2)(c^{MKV}\bar{\eta}_t^{MKV})^2\right)](\bar{x}_t^{MKV})^2dt.
	\label{eq:SC_MKV_1}
\end{IEEEeqnarray}
In the following, we intend to simplify the explicit solutions (\ref{eq:SC_MFG_1}) and (\ref{eq:SC_MKV_1}) for the social costs in the LQEMFG and LQEMKV problems. First, consider the quantity $\int_0^T (\bar{\eta}^{MFG}_t)^2 (\bar{x}^{MFG}_t)^2 dt$. Using equation (\ref{eq:eta_bar_MFG}), we have:
\begin{align*}
&\int_0^T (\bar{\eta}^{MFG}_t)^2 (\bar{x}^{MFG}_t)^2 dt\\
&= - \frac{1}{c^{MFG}(b_2 + \bar{b}_2)} \left[ \int_0^T \dot{\bar{\eta}}^{MFG}_t (\bar{x}^{MFG}_t)^2 dt +  \int_0^T \left[(2b_1 + \bar{b}_1) \bar{\eta}^{MFG}_t +(q+\bar{q}(1-s))\right] (\bar{x}^{MFG}_t)^2 dt \right],
\end{align*}
then using integration by parts for the first term in the bracket:
\begin{align*}
 =  -\frac{ 1}{c^{MFG}(b_2 + \bar{b}_2)} \Bigg[  & \bar{\eta}^{MFG}_T(\bar{x}^{MFG}_T)^2 -\bar{\eta}^{MFG}_0(\bar{x}^{MFG}_0)^2 -2\int_0^T\bar{\eta}^{MFG}_t\bar{x}^{MFG}_t \cdot \dot{\bar{x}}^{MFG}_t dt \\
& + \int_0^T \left[(2b_1 + \bar{b}_1) \bar{\eta}^{MFG}_t +(q+\bar{q}(1-s))\right] (\bar{x}^{MFG}_t)^2 dt \Bigg],
\end{align*}
and together with equation (\ref{eq:x_bar_MFG}) yields:
\begin{align*}
= 2 \int_0^T (\bar{\eta}^{MFG}_t)^2 (\bar{x}^{MFG}_t)^2dt  - \frac{1}{c^{MFG}(b_2 + \bar{b}_2)} \Bigg[ & \bar{\eta}^{MFG}_T (\bar{x}^{MFG}_T)^2 - \bar{\eta}^{MFG}_0 (\mathbb{E}(\xi))^2 \\
&+ \int_0^T  \left[-\bar{b}_1 \bar{\eta}^{MFG}_t +(q+\bar{q}(1-s))\right] (\bar{x}^{MFG}_t)^2 dt \Bigg].
\end{align*}
Finally, we arrive at:
\begin{align*}
\int_0^T (\bar{\eta}^{MFG}_t)^2 (\bar{x}^{MFG}_t)^2 dt= \frac{ \bar{\eta}^{MFG}_T (\bar{x}^{MFG}_T)^2 - \bar{\eta}^{MFG}_0 (\mathbb{E}(\xi))^2 + \int_0^T  \left[-\bar{b}_1 \bar{\eta}^{MFG}_t+(q+\bar{q}(1-s))\right] (\bar{x}^{MFG}_t)^2 dt}{c^{MFG}(b_2 + \bar{b}_2)}.
\end{align*}
If we denote:
\begin{equation*}
h_{var}: = \frac{1}{2}\int_0^T \left[q+\bar{q}+(r+\bar{r})(a\eta_t)^2\right]v_tdt + \frac{1}{2}(q_T+\bar{q}_T)v_T,
\end{equation*}
and use the terminal condition for $\bar{\eta}^{MFG}_T$,  
then equation (\ref{eq:SC_MFG_1}) can be rewritten as:
\begin{IEEEeqnarray}{rCl}
SC^{MFG} & = & h_{var} +  \frac{1}{2} \int_0^T \left[ \bar{b}_1 \lambda \bar{\eta}_t^{MFG} + (q+ \bar{q}(1-s)^2) - \lambda (q + \bar{q}(1-s)) \right] (\bar{x}_t^{MFG})^2 dt, \label{eq:SC_MFG_2} \\
& + & \frac{1}{2} \lambda \left[ \bar{\eta}^{MFG}_0 (\mathbb{E}(\xi))^2 - (q_T+\bar{q}_T(1-s_T)) (\bar{x}_T^{MFG})^2 \right] + \frac{1}{2} (q_T+\bar{q}_T(1-s_T)^2)(\bar{x}_T^{MFG})^2. \nonumber
\end{IEEEeqnarray}
Similarly, equation (\ref{eq:SC_MKV_1}) can be rewritten as:
\begin{equation}
SC^{MKV} = h_{var} + \frac{1}{2} \bar{\eta}^{MKV}_0 (\mathbb{E}(\xi))^2.
\label{eq:SC_MKV_2}
\end{equation}

Let's denote the (weighted) difference between the solutions of the Riccati equations associated with $\bar{\eta}_t^{MFG}$ and $\bar{\eta}_t^{MKV}$ by:
\begin{equation}
	\Delta \bar{\eta}_t := \lambda \bar{\eta}_t^{MFG} - \bar{\eta}_t^{MKV} = u_t - w_t.
\label{eq:delta_eta}
\end{equation} 

\begin{proposition}
	Under assumption (\ref{eq:assumptions}), the difference in the social costs in the LQEMFG and LQEMKV problems can be represented by:
	\begin{equation}
		\Delta SC := SC^{MFG} - SC^{MKV} =  \frac{1}{2}\cdot \frac{(b_2 + \bar{b}_2)^2}{r + \bar{r}(1-\bar{s})^2} \int_0^T (\Delta \bar{\eta}_t \cdot \bar{x}_t^{MFG})^2 dt.
	\label{eq:prop_2}
	\end{equation}
\label{prop:diff_SC}
\end{proposition}
\begin{proof}
	The solutions $\bar{\eta}^{MFG}_t$ and $\bar{\eta}^{MKV}_t$ for the Riccati equations (\ref{eq:eta_bar_MFG}) and (\ref{eq:eta_bar_MKV}), respectively, are well defined under assumption (\ref{eq:assumptions}) (see Appendix \ref{appendix_solving_fbsdes}). We notice that $\Delta\bar{\eta}_t$ defined in (\ref{eq:delta_eta}) satisfies the following linear first-order differential equation:
	\begin{equation*}
		\frac{d( \Delta{\bar{\eta}}_t)}{dt} = \gamma_t \Delta\bar{\eta}_t + \beta_t,
		\qquad \Delta{\bar{\eta}}_T = \lambda \bar{\eta}_T^{MFG} - \bar{\eta}_T^{MKV},
	\end{equation*}
	with coefficients:
	\begin{equation*} 
		\begin{array}{rcl}
			\gamma_t &=& \displaystyle -2 b_1 - 2\bar{b}_1 + \frac{(b_2 + \bar{b}_2)^2}{r + \bar{r}(1-\bar{s})^2} \left(\lambda \bar{\eta}_t^{MFG} + \bar{\eta}_t^{MKV} \right),\\ [8pt]
			\beta_t &=& \bar{b}_1 \lambda \bar{\eta}_t^{MFG} + (q + \bar{q}(1-s)^2) - \lambda (q + \bar{q}(1-s)).
		\end{array}
	\end{equation*} 
	Since $q_T + \bar{q}_T(1-s_T) = \bar{\eta}_T^{MFG}$, $q_T + \bar{q}_T(1-s_T)^2 = \bar{\eta}_T^{MKV}$ and $\lambda \bar{\eta}_0^{MFG} - \bar{\eta}_0^{MKV} = \Delta\bar{\eta}_0$, we deduce from equations (\ref{eq:SC_MFG_2}) and (\ref{eq:SC_MKV_2}) that:
	\begin{equation*}
		\begin{split}
			SC^{MFG} - SC^{MKV}&= \frac{1}{2} \left[ \Delta\bar{\eta}_0 (\mathbb{E}(\xi))^2 - \Delta\bar{\eta}_T (\bar{x}_T^{MFG})^2 + \int_0^T \beta_t (\bar{x}_t^{MFG})^2 dt \right] \\
			&= \frac{1}{2} \int_0^T \left[ - \frac{d (\Delta\bar{\eta}_t (\bar{x}_t^{MFG})^2) }{dt}  + \left( 	\frac{d( \Delta{\bar{\eta}}_t)}{dt} - \gamma_t \Delta\bar{\eta}_t  \right)(\bar{x}_t^{MFG})^2 \right] dt\\
			&= \frac{1}{2}\int_0^T \left[- 2\Delta{\bar{\eta}}_t \bar{x}_t^{MFG} \dot{\bar{x}}_t^{MFG} - \gamma_t \Delta{\bar{\eta}}_t (\bar{x}_t^{MFG})^2 \right]dt \\
			&= \frac{1}{2} \int_0^T \Delta{\bar{\eta}}_t (\bar{x}_t^{MFG})^2 \left[-2\left(b_1 + \bar{b}_1 - \frac{(b_2+\bar{b}_2)^2}{r + \bar{r}(1-\bar{s})^2} \lambda \bar{\eta}_t^{MFG}\right) - \gamma_t\right] dt\\
			&= \frac{1}{2}\cdot \frac{(b_2 + \bar{b}_2)^2}{r + \bar{r}(1-\bar{s})^2} \int_0^T (\Delta \bar{\eta}_t \cdot \bar{x}_t^{MFG})^2 dt,
		\end{split}
	\end{equation*}
	where we use equation (\ref{eq:x_bar_MFG}) for the fourth equality.	
	
\end{proof}

\begin{remark}
	We can see directly from Proposition \ref{prop:diff_SC} that the social cost in the LQEMFG problem is larger than (or possibly equal to) the social cost in the LQEMKV problem. This result is consistent with the definition of the price of anarchy in section \ref{Price_Of_Anarchy}.
\end{remark}

Note that we can write:
\begin{equation}
PoA =1+ \frac{\Delta SC}{SC^{MKV}}.
\label{eq:PoA_LQ}
\end{equation}

It will be useful for us to note here the scalar Riccati equations associated with $u_t = \lambda \bar{\eta}_t^{MFG}$, $w_t = \bar{\eta}_t^{MKV}$ and $\eta_t$:
    \begin{IEEEeqnarray}{lLLLClClCl}
		\dot{u}_t &- 2A^u u_t &- B u_t^2 &+ C^u &= & 0, &\qquad& u_T &=& D^u,
		\label{eq:riccati_ut} \\
	    \dot{w}_t &- 2A^w w_t &- B w_t^2 &+ C^w &=&0, &\qquad &w_T &= & D^w,
		\label{eq:riccati_wt}\\
	    \dot{\eta}_t &- 2A^\eta \eta_t &- B^\eta \eta_t^2 &+ C^\eta &= & 0, &\qquad& \eta_T &= &D^\eta,
	\label{eq:riccati_etat}
	\end{IEEEeqnarray}
	with:
	\begin{IEEEeqnarray}{rClcrClcrCl}
		A^u &=& - \left(b_1 + \frac{\bar{b}_1}{2}\right), & \qquad & 
		A^w &=& - (b_1 + \bar{b}_1), & \qquad & 
		A^\eta &=& - b_1, \nonumber \\
		B^u &=& \frac{(b_2+\bar{b}_2)^2}{r + \bar{r}(1-\bar{s})^2}, & \qquad &	
		B^w &=& \frac{(b_2+\bar{b}_2)^2}{r + \bar{r}(1-\bar{s})^2}, &\qquad & 
		B^\eta &=& \frac{b_2^2}{r+ \bar{r}}, \nonumber \\
		C^u &=& \lambda (q + \bar{q}(1-s)),& \qquad & 
		C^w &=& q + \bar{q}(1-s)^2, & \qquad &
		C^\eta &=& q + \bar{q}, \nonumber \\
		D^u &=& \lambda(q_T + \bar{q}_T(1-s_T)), & \qquad & 
		D^w &=& q_T + \bar{q}_T(1-s_T)^2, & \qquad & 
		D^\eta &=& q_T + \bar{q}_T .
		\label{eq:ABCD}
	\end{IEEEeqnarray}
If $B^u\neq 0$, $B^u D^u \geq 0$ and $B^uC^u >0$, we have (see equation (\ref{eq:riccati_ut_sol_app}) in Appendix A) the existence and uniqueness for $u_t$ which can be expressed by:
\begin{equation}
		u_t= \frac{C^u(1-e^{-(\delta_u^+ - \delta_u^-)(T-t)})+D^u(\delta^+_u -\delta^-_ue^{-(\delta_u^+ - \delta_u^-)(T-t)} )}{B^uD^u(1-e^{-(\delta_u^+ - \delta_u^-)(T-t)})+ \delta^+_ue^{-(\delta_u^+ - \delta_u^-)(T-t)} -\delta^-_u  },
	\label{eq:riccati_ut_sol}
\end{equation}
with $\delta^\pm_u = -A^u \pm \sqrt{(A^u)^2 + B^u C^u}$. Under assumption (\ref{eq:assumptions}), the above conditions on $B^u$, $C^u$, and $D^u$ are satisfied, and we have $\delta^-_u<0<\delta^+_u$, $u_t > 0$ for all $t\in[0,T)$, and $u_T \geq 0$. We have analogous expressions for $w_t$ and $\eta_t$, in terms of $\delta^\pm_w$ and $\delta^\pm_{\eta}$, respectively. Note that $B^u=B^w=:B$.

It will also be useful to compute the derivative of $u_t$ with respect to time $t$ from the explicit form in equation (\ref{eq:riccati_ut_sol}):
\begin{equation}
\frac{d u_t}{dt} = \frac{\left(B (D^{u})^2 + 2 A^u D^{u} - C^{u} \right) \cdot \left(\delta_u^{+} - \delta_u^{-}\right)^2 e^{- (\delta_u^{+} - \delta_u^{-})(T-t) } }{\left[B D^{u}(1-e^{- (\delta_u^{+} - \delta_u^{-})(T-t)})   + \delta_u^{+}e^{- (\delta_u^{+} - \delta_u^{-})(T-t)}- \delta_u^{-}  \right]^2 }.
\label{eq:riccati_ut_sol_deriv}
\end{equation}
Note that $u_t$ is increasing if $B (D^{u})^2 + 2 A^u D^{u} - C^{u}>0$, and likewise, decreasing if $B (D^{u})^2 + 2 A^u D^{u} - C^{u}<0$.

\begin{theorem}
	Assume (\ref{eq:assumptions}) and the initial condition $\xi$ satisfies $\mathbb{E}(\xi)  \neq 0$. Let $A^u$, $A^w$, $B$, $C^u$, $C^w$, $D^u$, and $D^w$ as defined in equation (\ref{eq:ABCD}).
	\begin{itemize}
		\item When $\bar{b}_1 > 0$, we have $PoA = 1$ if and only if:
		\begin{equation}
		D^u = D^w =:D \quad \text{and} \quad B D^2 + 2 A^{u} D - C^{u} = B D^2 + 2 A^{w} D - C^w  = 0.
		\label{eq:PoA=1_condition_b1_bar>0}
		\end{equation}
		\item When $\bar{b}_1 = 0$, then $A^u = A^w$ and we have $PoA = 1$ if and only if:
		\begin{equation}
			D^u = D^w \quad \text{and} \quad C^u = C^w. 
		\label{eq:PoA=1_condition_b1_bar=0}
		\end{equation}
	\end{itemize}
	\label{thm:PoA_equivalent_1}
\end{theorem}

\begin{proof}
	From an analogous equation to (\ref{eq:riccati_ut_sol}) for $w_t$, we know that under assumption (\ref{eq:assumptions}), $w_0>0$. Thus, with the assumption $\mathbb{E}(\xi) \neq 0$, we have $0 < SC^{MKV} < \infty$. Hence, $PoA =1$ if and only if $\Delta SC = 0$. Since $\bar{x}_t^{MFG} \neq 0$ for all $t \in [0,T]$, from Proposition \ref{prop:diff_SC} and the continuity of $u_t$ and $w_t$, we deduce that:
	\begin{equation*}
		PoA = 1 \qquad \text{if and only if} \qquad u_t = w_t, \quad \forall  t \in [0,T]. 
	\end{equation*}
	
	From equation (\ref{eq:riccati_ut_sol_deriv}) and the uniqueness of solutions to Riccati equations (\ref{eq:riccati_ut}) and (\ref{eq:riccati_wt}), it is easy to check that if the conditions in (\ref{eq:PoA=1_condition_b1_bar>0}) and (\ref{eq:PoA=1_condition_b1_bar=0}) are satisfied, then $u_t=D^u=D^w = w_t$ for all $ t\in [0,T]$, and thus, $PoA = 1$.	
	
	Suppose now that $PoA =1$. Then $u_t = w_t$ for all $t \in [0,T]$ and clearly:
	$$D^u = u_T = w_T = D^w.$$
	Now, if we take the difference between the two Riccati equations (\ref{eq:riccati_ut}) and (\ref{eq:riccati_wt}), and by using $u_t = w_t$ for all $t \in [0,T]$, we obtain:
	\begin{equation}
		2 (A^u - A^w) w_t = C^u - C^w, \quad \forall \  t \in [0,T].
	\label{eq:wt=constant}
	\end{equation}
	Since $2(A^u - A^w) = \bar{b}_1$, in the case when $\bar{b}_1 =0$ we must have $C^u = C^w$. Otherwise, equation (\ref{eq:wt=constant}) implies that $u_t = w_t = (C^u - C^w) / \bar{b}_1 $, are constant for all $t \in [0,T]$. Thus, the time derivatives of $u_t$ and $w_t$ should be zero. From equation (\ref{eq:riccati_ut_sol_deriv}), and the fact that $\delta_u^+ - \delta_u^- > 0$, we deduce: 
	$$ B (D^{u})^2 + 2 A^u D^u - C^u = 0.$$
	Similarly we also have $B (D^{w})^2 + 2 A^w D^w - C^w = 0$. 
	
\end{proof}

\begin{corollary}
	Assume (\ref{eq:assumptions}) and $\lambda=1$. Then the sufficient condition (equations (\ref{eq:prop_1_1})-(\ref{eq:prop_1_3})) from Proposition \ref{prop: lq_subtract_FBSDEs} is also a necessary condition to have $PoA=1$.
\label{corollary_3}
\end{corollary}
\begin{proof}
	Assume $PoA=1$. Since $\lambda=1$, we have $c^{MFG}=c^{MKV}$ and thus, condition (\ref{eq:prop_1_2}) holds. If $\mathbb{E}(\xi)=0$, clearly conditions (\ref{eq:prop_1_1})-(\ref{eq:prop_1_3}) hold, as noted in Corollary \ref{corollary_2}. If $\mathbb{E}(\xi) \neq 0$ and $\bar{b}_1=0$, from Theorem \ref{thm:PoA_equivalent_1}, we have $D^u=D^w$ and $C^u=C^w$ which together with $\lambda=1$ imply conditions (\ref{eq:prop_1_1}) and (\ref{eq:prop_1_3}), similarly as Corollary \ref{corollary_1}. Now, if $\mathbb{E}(\xi) \neq 0$ and $\bar{b}_1>0$, from Theorem \ref{thm:PoA_equivalent_1}, we have $D^u=D^w$ which implies condition (\ref{eq:prop_1_3}). Finally, the condition $B(D)^2+2A^uD-C^u=B(D)^2+2A^wD-C^w=0$ implies $\bar{\eta}^{MKV}_t=w_t=D=\frac{C^u-C^w}{2(A^u-A^2)}=\frac{s\bar{q}(1-s)}{\bar{b}_1}$ and thus, we have condition (\ref{eq:prop_1_1}).
\end{proof}

We study in the following the variation of $PoA$ by letting only one of the coefficients tend to zero or to infinity. In order to make the following computations easier to follow, we repeat equations (\ref{eq:prop_2}), (\ref{eq:SC_MKV_2}), (\ref{eq:x_bar_MFG_explicit}), and (\ref{eq:v_t}), which we recall is equivalent to  equation (\ref{eq:v_t_MKV}), using the above notations. Assuming (\ref{eq:assumptions}), we have:
\begin{IEEEeqnarray}{l}
	\displaystyle \Delta SC= \frac{1}{2}B\int_0^T (u_t-w_t)^2 \cdot (\bar{x}_t^{MFG})^2 dt, \label{eq:Delta_SC_LQ_new_notation} \\[3pt]
	\displaystyle SC^{MKV}= \frac{1}{2}\int_0^T \left[q+\bar{q} + B^{\eta} \eta_t^2 \right] v_t dt + \frac{(q_T + \bar{q}_T)}{2} v_T + \frac{1}{2}w_0 (\mathbb{E}(\xi))^2, \label{eq:SC_MKV_LQ_new_notation} \\[3pt]
	\displaystyle \bar{x}^{MFG}_t =\mathbb{E}(\xi) e^{\int_0^t(b_1+\bar{b}_1-Bu_s)ds}, \label{eq:x_bar_MFG_explicit_new_notation} \\[3pt]	
	\displaystyle v_t = Var(\xi)e^{\int_0^t 2(b_1-B^{\eta} \eta_s)ds}+\sigma^2 \int_0^t e^{2 \int_s^t (b_1-B^{\eta}\eta_u) du}ds. \label{eq:v_t_new_notation}
\end{IEEEeqnarray}
Also for convenience, recall the definition from equation (\ref{eq:lambda}):
$
\displaystyle
\lambda = \frac{b_2}{b_2 + \bar{b}_2}\cdot \frac{ r + \bar{r}(1-\bar{s})^2 }{r + \bar{r}(1-\bar{s})}.
$

\noindent In the following propositions, we utilize the following assumption to make their proofs simpler.	
\begin{assumption}\label{assumption}
	Assume (\ref{eq:assumptions}). In addition, assume: $b_1 > 0$, $D^{u} > 0$, $D^w > 0$, $D^{\eta} >0$ and the initial condition satisfies $\mathbb{E}(\xi) \neq 0$.
\end{assumption}

\begin{proposition}
	Assuming Assumption \ref{assumption}, then:
	\begin{equation*}
			\lim_{r \to \infty} PoA = 1 \qquad \text{and} \qquad
			\lim_{\bar{r} \to \infty} PoA = 1.
	\end{equation*}
	\label{prop:r_bar_r}
\end{proposition}
\begin{proof}
	First, we consider $r \to \infty$. For every given $r >0$, we have existence and uniqueness of the solutions $u^{r}_t$, $w^{r}_t$ and $\eta^{r}_t$ to the scalar Riccati equations (\ref{eq:riccati_ut})-(\ref{eq:riccati_etat}). Note that we have added the superscript $r$ to emphasize the dependence on this parameter. 
	
	When $r \to \infty$, we have:
	\begin{equation*}
		\lambda^r \longrightarrow \lambda^{r\to\infty} := \frac{b_2}{b_2 + \bar{b}_2},
	\end{equation*}
	and
	\begin{equation*}
	\begin{array}{l}
		B^r \longrightarrow 0,\quad \ B^{\eta,r} \longrightarrow 0,\\
		C^{u,r} \longrightarrow C^{u,r\to \infty} := \lambda^{r \to \infty}(q + \bar{q}(1-s)),\quad \ D^{u,r} \to D^{u,r \to \infty} := \lambda^{r \to \infty}(q_T+ \bar{q}_T(1-s_T)).
	\end{array}
	\end{equation*}

	Let $u^{r \to \infty}:[0,T] \to \mathbb{R}$ be the solution to the linear first-order differential equation:
	\begin{equation*}
		\left({u}^{r \to \infty}_t\right)' - 2A^u u_t^{r \to \infty} + C^{u,r \to \infty} = 0, \qquad u_T^{r \to \infty} = D^{u,r \to \infty}.
	\end{equation*}
	Then we have:
	\begin{equation*}
	u_t^{r \to \infty} = \left(D^{u,r \to \infty} - \frac{C^{u,r \to \infty}}{2A^u} \right) e^{-2A^u (T-t)} +  \frac{C^{u,r \to \infty}}{2A^u}.
	\end{equation*}
	It is easy to show directly from their explicit solutions (see equation (\ref{eq:riccati_ut_sol})) that for every time $t\in [0,T]$, $$\lim_{r \to \infty}u^r_t=u^{r \to \infty}_t, \quad \text{and thus,} \quad \lim_{r \to \infty}B^ru^r_t=0.$$
	
	Next, our goal is to bound the $u^r_t$ uniformly over $t \in [0,T]$ for large $r$.
	Note that  $A^u < 0$, $B^r$, $C^{u,r}$, $\lambda^{r\to \infty}$, $C^{u,r \to \infty}$, $D^{u,r}$, $D^{u,r \to \infty}> 0$, and $\delta^{-,r}_u < 0 < \delta^{+,r}_u$. Let $\epsilon>0$. Then there exists a $r^*>0$ such that $\max\{ B^{r}, C^{u,r}, D^{u,r} \} \leq \max\{C^{u,r \to \infty},D^{u,r \to \infty} \}+\epsilon=: \zeta$ for $r \geq r^*$. Thus, we can deduce that for $r \geq r^*$, and for every $t \in [0,T]$:
	\begin{equation*}
	\left\vert u_t^{r} \right\vert \leq \frac{C^{u,r} + D^{u,r} (\delta^{+,r}_u - \delta^{-,r}_u)}{\delta^{+,r}_u e^{- (\delta^{+,r}_u - \delta^{-,r}_u)(T-t)}} 
	\leq  \frac{ \zeta + 2 \zeta \sqrt{(A^u)^2 + \zeta^2}  }{-2A^u e^{-2T \sqrt{(A^u)^2 + \zeta^2} } }.
	\end{equation*}
	From equation (\ref{eq:x_bar_MFG_explicit_new_notation}) and by the bounded convergence theorem, we have for every $t \in [0,T]$:
	\begin{equation*}
		\lim_{r\to \infty}\bar{x}_t^{MFG,r} = \mathbb{E}(\xi) e^{(b_1+ \bar{b}_1)t} =: \bar{x}_t^{MFG,r\to \infty}.
	\end{equation*}
	Moreover, $\bar{x}_t^{MFG,r}$ is uniformly bounded for $t \in [0,T]$. From the non-negativity of $u^r_t$, we have:
	$$ \left\vert \bar{x}^{MFG,r}_t \right\vert \leq \vert \mathbb{E}(\xi) \vert e^{ (b_1 + \bar{b}_1 ) T }, \quad \forall t \in [0,T]. $$  	
	Similarly, for every $t\in [0,T]$:
	$$\lim_{r\to \infty}w^r_t =: w_t^{r\to \infty},\qquad \lim_{r\to \infty} \eta^r_t =: \eta_t^{r\to \infty},$$
	and the functions $w^{r}_t$ and $\eta^{r}_t$ are uniformly bounded over $t \in[0,T]$ and large $r$. By the bounded convergence theorem we have for every $t \in [0,T]$:
	\begin{equation*}
	\lim_{r \to \infty} \int_0^T (u^r_t - w^r_t)^2 (\bar{x}_t^{MFG,r})^2 dt = \int_0^T (u^{r \to \infty}_t- w^{r \to \infty}_t)^2 (\bar{x}_t^{MFG,r\to \infty})^2 dt \  < \infty,
	\end{equation*}
	and thus, from equation (\ref{eq:Delta_SC_LQ_new_notation}):
	\begin{equation*}
	\lim_{r \to \infty} \Delta SC^r = \lim_{r \to \infty} \frac{1}{2}\cdot B^r\int_0^T (u^r_t - w^r_t)^2 ( \bar{x}_t^{MFG,r})^2 dt = 0.
	\end{equation*}
	
	From equation (\ref{eq:v_t_new_notation}) and by the bounded convergence theorem, we have for every $t \in [0,T]$: 
	$$	\lim_{r\to \infty} v^r_t = Var(\xi)e^{2b_1 t} + \sigma^2 \int_0^t e^{2 b_1 (t-s) } ds =: v^{r \to \infty}_t.
	$$
	We also have $w^{r \to \infty }_0 >0$ and $v^{r \to \infty}_t >0$ for $t>0$. Hence, from equation (\ref{eq:SC_MKV_LQ_new_notation}):
	\begin{equation*}
		\lim_{r \to \infty} SC^{MKV,r} = \frac{1}{2} \left( \int_0^T (q+\bar{q}) v^{r \to \infty}_t dt + (q_T + \bar{q}_T) v^{r\to \infty}_T + w^{r \to \infty}_0 (\mathbb{E}(\xi))^2 \right)  >0.
	\end{equation*}
	Therefore, from equation (\ref{eq:PoA_LQ}), we have: 
	$$\displaystyle \lim_{r \to \infty} PoA^r = 1.$$
	
	By replacing $\lambda^{r \to \infty}$ with $\lambda^{\bar{r} \to \infty}: = \frac{b_2}{b_2 + \bar{b}_2} (1 - \bar{s})$, the proof can be repeated, and we obtain $\displaystyle \lim_{\bar{r} \to \infty} PoA^{\bar{r}} = 1$.
\end{proof}

\begin{proposition}
	Assume Assumption \ref{assumption}. If:
	\begin{equation*}
	\frac{q+\bar{q}(1-s)}{r+\bar{r}(1-\bar{s})} =\frac{q+\bar{q}(1-s)^2}{r+\bar{r}(1-\bar{s})^2},
	\end{equation*}
	then:
	$$ \lim_{b_2 \to \infty} PoA = 1.$$
	Otherwise,
	$$ \lim_{b_2 \to \infty} PoA > 1.$$
\label{prop:b2_to_infty}
\end{proposition}
\begin{proof}
	When $b_2 \to \infty$, we have:
	\begin{equation*}
	\begin{array}{l}
	\lambda^{b_2} \to \frac{r + \bar{r}(1-\bar{s})^2}{r + \bar{r}(1-\bar{s})} =: \lambda^{b_2 \to \infty}, \quad	B^{b_2} \to \infty, \qquad B^{\eta,b_2} \to \infty\\
	C^{u,b_2} \to \lambda^{b_2 \to \infty} ( q + \bar{q}(1-s)) =: C^{u,b_2 \to \infty}, \qquad D^{u,b_2} \to \lambda^{b_2 \to \infty} (q_T + \bar{q}_T(1-s_T)) =: D^{u,b_2 \to \infty} >0,
	\end{array}
	\end{equation*}
	and $A^{u}, (A^{w}, C^w, D^w), (A^{\eta}, C^{\eta}, D^{\eta})$ are independent of  $b_2$. Moreover, we notice that:
	\begin{equation*}
		\frac{\delta_u^{\pm,b_2}}{b_2 + \bar{b}_2} = -\frac{A^u}{b_2 + \bar{b}_2} \pm \sqrt{\frac{(A^u)^2}{(b_2+\bar{b}_2)^2}+\frac{C^{u,b_2}}{r + \bar{r}(1-\bar{s})^2} } \xrightarrow[b_2 \to \infty]{} \displaystyle \pm \sqrt{\frac{q + \bar{q}(1-s)}{r + \bar{r}(1-\bar{s})} } =: \pm c_{\delta_u},
	\end{equation*}
	and thus, $\lim_{b_2 \to \infty}\delta^{+,b_2}_u - \delta^{-,b_2}_u = +\infty$.
	For the sake of simplicity, denote $h_u(b_2,t) := (b_2 + \bar{b}_2) u_{t}^{b_2}$ and $h_w(b_2,t) := (b_2 + \bar{b}_2)w_t^{b_2}$. From equation (\ref{eq:riccati_ut_sol}), for all $t \in[0,T)$, we deduce:
	\begin{equation*}
	\begin{split}
		h_u(b_2,t) &= \frac{ \displaystyle  \left(\frac{C^{u,b_2}}{b_2 + \bar{b}_2} + D^{u,b_2} \cdot \frac{\delta^{+,b_2}_u}{b_2 + \bar{b}_2}\right) - \left(\frac{C^{u,b_2}}{b_2 + \bar{b}_2} + D^{u,b_2} \cdot \frac{\delta^{-,b_2}_u}{b_2 + \bar{b}_2}\right)e^{- (\delta^{+,b_2}_u - \delta^{-,b_2}_u)(T-t)} }
		{\displaystyle \left(\frac{- \delta^{-,b_2}_u}{(b_2 + \bar{b}_2)^2} + \frac{D^{u,b_2}}{r + \bar{r}(1-\bar{s})^2} \right) + \left(\frac{\delta^{+,b_2}_u}{(b_2 + \bar{b}_2)^2} - \frac{D^{u,b_2} }{r + \bar{r}(1-\bar{s})^2} \right)e^{- (\delta^{+,b_2}_u - \delta^{-,b_2}_u)(T-t)}  } \nonumber \\
		\\
		&\xrightarrow[b_2 \to \infty]{} \quad (r + \bar{r}(1-\bar{s})^2) c_{\delta_u}  =: c_u.
	\end{split}
	\end{equation*}
	Similarly, for all $t \in [0,T)$:
	$$\lim_{b_2 \to \infty} h_w(b_2,t)=  (r + \bar{r}(1-\bar{s})^2) c_{\delta_w} =: c_w,\qquad \text{with } \quad c_{\delta_w} := \sqrt{ \frac{q + \bar{q}(1-s)^2}{r + \bar{r}(1-\bar{s})^2} },$$
	and
	$$ \lim_{b_2 \to \infty}b_2 \eta_t^{b_2} = (r + \bar{r}) c_{\delta_{\eta}} =: c_\eta, \qquad \text{with} \quad c_{\delta_\eta} := \sqrt{ \frac{q + \bar{q}}{r +\bar{r}} }.$$
	
	Next, we derive a strictly positive uniform lower bound for $(b_2 + \bar{b}_2) u_t^{b_2}$ over $[0,T]$ and large $b_2$. Let $\zeta_1 :=\frac{1}{2} \min \left\{ c_{\delta_u}, D^{u,b_2 \to \infty} \right\}$. Then there exists a $b_2^{*,u,lower}>0$ such that for all $b_2 \geq b_2^{*,u,lower}$:
	$$ 
	\max \left\{  \left\vert \frac{\delta_u^{+,b_2}}{b_2 + \bar{b}_2}  - c_{\delta_u} \right\vert, \   \left\vert \frac{\delta_u^{-,b_2}}{b_2 + \bar{b}_2}  - \left(-c_{\delta_u}\right) \right\vert, \left\vert D^{u,b_2} - D^{u,b_2\to \infty} \right\vert,  \left\vert \frac{1}{b_2 + \bar{b}_2} \right\vert  \right\} \leq \zeta_1,
	$$ 
	and thus for all $t \in [0,T]$:
	\begin{IEEEeqnarray}{rCl}
		h_u(b_2,t) &\geq&  (b_2 + \bar{b}_2) \frac{D^{u,b_2} \delta_u^{+, b_2} }{ (\delta_u^{+,b_2} - \delta_u^{-,b_2}) + B^{b_2}D^{u,b_2} } \nonumber \\
		&\geq& \frac{(D^{u,b_2 \to \infty} - \zeta_1 )(c_{\delta_u} - \zeta_1) }{ \zeta_1  (2 c_{\delta_u} + 2 \zeta_1) + (D^{u,b_2 \to \infty} + \zeta_1) / (r+ \bar{r}(1-\bar{s})^2) } =: m_u>0.
	\label{eq:ut_uniform_lower_bound_b2_infinity}
	\end{IEEEeqnarray}
	Then, by the same technique in inequality (\ref{eq:ut_uniform_lower_bound_b2_infinity}), there exists a $b_2^{*,\eta, lower} > 0$ and $m_\eta >0$ such that for all $b_2 \geq b_2^{*,\eta,lower}$ and all $t\in[0,T]$:
	\begin{equation}
	   b_2 \eta_t^{b_2} \geq m_\eta.
	\label{eq:eta_t_uniform_lower_bound_b2_infinity}
	\end{equation}

	From equation (\ref{eq:riccati_ut_sol_deriv}), we see that $t \mapsto u_t^{b_2}$ is increasing if $B^{b_2} (D^{u,b_2})^2 + 2 A^u D^{u,b_2} - C^{u,b_2} > 0$. Since 
	$\lim_{b_2 \to \infty} B^{b_2} (D^{u,b_2})^2 + 2 A^u D^{u,b_2} - C^{u,b_2} = \infty,$ there exists a $b_2^{*,u,upper} >0$ such that for all $b_2 \geq b_2^{*,u,upper}$, we have $|D^{u,b_2} -D^{u,b_2 \to \infty} | \leq 1$ and $t \mapsto u_t^{b_2}$ is increasing.
	Therefore,
	\begin{equation*}
	u_t^{b_2}\leq u_T^{b_2} = D^{u,b_2} \leq D^{u,b_2 \to \infty} + 1, \quad \forall \  t\in[0,T],\ b_2 \geq b_2^{*,u,upper}.
	\label{eq:u_t_bound}
	\end{equation*}
	By the same argument for $w_t^{b_2}$ and $\eta_t^{b_2}$, there exists a $b_2^{*,upper} \geq b_2^{*,u,upper}$ such that:
	\begin{equation}
	\max \left\{ \left\vert u_t^{b_2} \right\vert , \left\vert w_t^{b_2} \right\vert, \left\vert \eta_t^{b_2} \right\vert \right\} \leq M, \quad \forall \  t\in[0,T],\ b_2 \geq b_2^{*,upper},
	\label{eq:ut_wt_upper_bound_b2_infty}
	\end{equation}
	and such that the functions $t \mapsto u_t^{b_2}$, $t \mapsto w_t^{b_2}$ and $t \mapsto \eta_t^{b_2}$ are increasing on $[0,T]$. \\
	
	\textbf{Case 1:} Assume:
	\begin{equation*}
	\frac{q+\bar{q}(1-s)}{r+\bar{r}(1-\bar{s})}=\frac{q+\bar{q}(1-s)^2}{r+\bar{r}(1-\bar{s})^2}.
	\end{equation*}
	Then $c_{\delta_u}=c_{\delta_w}$ and therefore, $c_u = c_w =: c$. 
	We want to show that $\lim_{b_2 \to \infty} \frac{\Delta SC^{b_2}}{SC^{MKV,b_2}} = 0$. Our approach is to split the interval $[0,T]$ into two parts: $[0,T/2]$ and $[T/2,T]$.	
	Since $v^{b_2}_t \geq 0$ for all $t\in [0,T]$, from equations (\ref{eq:Delta_SC_LQ_new_notation}) and (\ref{eq:SC_MKV_LQ_new_notation}), we have $SC^{MKV,b_2} \geq w_0^{b_2} (\mathbb{E}(\xi))^2$ and thus:
	\begin{equation}
	\begin{split}
		\frac{\Delta SC^{b_2}}{SC^{MKV,b_2}} &\leq \frac{1}{w_0^{b_2} \mathbb{E}(\xi)^2 } \left( B^{b_2} \int_{0}^{\frac{T}{2}} (u_t^{b_2} - w_t^{b_2})^2 (\bar{x}_t^{MFG,b_2})^2 dt + B^{b_2} \int_{\frac{T}{2}}^{T} (u_t^{b_2} - w_t^{b_2})^2 (\bar{x}_t^{MFG,b_2})^2 dt \right) \\ 
		&=  \frac{1}{(b_2 + \bar{b}_2) w_0^{b_2} } \left( I_1^{b_2} + I_2^{b_2} \right),
	\end{split}
	\label{eq:I_1_plus_I_2}
	\end{equation}
	with:
	\begin{IEEEeqnarray*}{rCl}
	I_1^{b_2} &=& \frac{(b_2 + \bar{b}_2)^3}{r + \bar{r}(1-\bar{s})^2}  \int_0^{\frac{T}{2}} (u_t^{b_2} - w_t^{b_2})^2 e^{2(b_1 + \bar{b}_1)t} \exp \left( -2B^{b_2} \int_0^{t} u_s^{b_2} ds 
	\right) dt \\
	&=& \frac{b_2 + \bar{b}_2}{r + \bar{r}(1-\bar{s})^2} \int_0^{\frac{T}{2}} [ h_u(b_2,t)- h_w(b_2,t) ]^2 \cdot e^{2(b_1 + \bar{b}_1)t} \exp \left( - \frac{2 (b_2 + \bar{b}_2)}{r+ \bar{r}(1-\bar{s})^2}  \int_0^t h_u(b_2,s) ds \right) dt, 
	\end{IEEEeqnarray*}
	and:
	$$
		I_2^{b_2} = \frac{(b_2 + \bar{b}_2)^3}{r + \bar{r}(1-\bar{s})^2}  \int_{\frac{T}{2}}^T (u_t^{b_2} - w_t^{b_2})^2 e^{2(b_1 + \bar{b}_1)t} \exp \left( -\frac{2 (b_2 + \bar{b}_2)}{r + \bar{r}(1-\bar{s})^2} \int_0^{t} h_u(b_2,s) ds \right) dt.
	$$
	
	Fix $\epsilon>0$. In the following, we show that $I_1^{b_2} \leq \epsilon$ and $I_2^{b_2} \leq \epsilon$ for large $b_2$. First, consider $I_1^{b_2}$. Recall that for $t \in[0,T/2]$, we have $\lim_{b_2 \to \infty}h_u(b_2,t) = \lim_{b_2 \to \infty}h_w(b_2,t) = c$, and for all $b_2 \geq b_2^{*,upper}$, the functions $[0,T/2] \ni s \mapsto u_s^{b_2}$ and $[0,T/2] \ni s \mapsto w_s^{b_2}$ are increasing, and thus, $[0,T/2] \ni s \mapsto h_u(b_2,t)$ and $[0,T/2] \ni s \mapsto h_w(b_2,t)$ are increasing. (Note that $T/2<T$ is chosen arbitrarily, since the above limits do not hold at $T$.) Let
	$ \zeta_2: = \min \left\{ \frac{c}{2},  \frac{1}{2} e^{-T(b_1+\bar{b}_2)} \sqrt{\epsilon c} \right\}.$
	Then there exists a $b_2^{*,I_1}  \geq b_2^{*,upper}$ such that for all $b_2 \geq b_2^{*,I_1}$ and all $s\in [0, T/2]$ we have:
	\begin{equation*}
	\begin{array}{lclclclcl}
		c - \zeta_2 &\leq& h_u(b_2,0) &\leq& h_u(b_2,s) &\leq& h_u(b_2, T/2 ) &\leq& c + \zeta_2, \\
		c - \zeta_2 &\leq& h_w(b_2,0) &\leq& h_w(b_2,s) &\leq& h_w(b_2, T/2 ) &\leq& c + \zeta_2.
	\end{array}
	\end{equation*}
	Thus, for any $t \in [0,T/2]$ and $b_2 \geq b_2^{*,I_1}$:
	$$
		\left\vert h_u(b_2, t) - h_w(b_2,t) \right\vert^2 \leq  4 \zeta_2^2 \quad \text{and} \quad \int_0^t h_u(b_2,s)ds \geq (c - \zeta_2) t \geq \frac{c}{2}\cdot t.
	$$
	Therefore,
	\begin{IEEEeqnarray}{rCl}
		I_1^{b_2} &\leq& 4\zeta_2^2e^{2T(b_1 + \bar{b}_1)} \cdot\frac{(b_2 + \bar{b}_2)}{r+ \bar{r}(1-\bar{s})^2} \int_0^{\frac{T}{2}} \exp \left(  - \frac{2 (b_2 + \bar{b}_2)}{r + \bar{r}(1-\bar{s})^2} \cdot \frac{c}{2}\cdot t \right) dt \nonumber \\
		&=& \frac{4e^{2T(b_1 + \bar{b}_1)}}{c} \left( 1 - e^{ - \frac{(b_2 + \bar{b}_2) c}{r + \bar{r}(1-\bar{s})^2}  \frac{T}{2} } \right)\zeta_2^2 \leq \epsilon, 
	\label{eq:I1_ineq_b2_infinity}
	\end{IEEEeqnarray}
	where the last inequality comes from the definition of $\zeta_2$.
	
	Next, consider $I_2^{b_2}$. Since $u_t^{b_2}$ is positive over $[0,T]$, we know from the inequalities (\ref{eq:ut_uniform_lower_bound_b2_infinity}) and (\ref{eq:ut_wt_upper_bound_b2_infty}) that for all $b_2 \geq \max\{ b_2^{*,upper}, b_2^{*,u,lower} \}$ and all $t \in [T/2,T]$:
	\begin{equation*}
		\left\vert u_t^{b_2} - w_t^{b_2} \right\vert \leq \sup_{0 \leq s \leq T} \left\vert u_s^{b_2} \right\vert + \left\vert w_s^{b_2} \right\vert  \leq 2M,
		\qquad \text{and} \qquad
		\int_0^{t} h_u(b_2,s) ds \geq \int_0^{\frac{T}{2}} h_u(b_2,s) ds \geq \frac{T}{2} m_u >0.
	\end{equation*}
	Hence, there exists a $b_2^{*,I_2} \geq \max \{ b_2^{*,upper}, b_2^{*,u,lower} \}$ such that for all $b_2 \geq b_2^{*,I_2}$:
	\begin{IEEEeqnarray}{rCl}
		I_2^{b_2} &\leq&  \frac{(b_2+ \bar{b}_2)^3}{r + \bar{r}(1- \bar{s})^2} \cdot 4M^2 e^{2(b_1 + \bar{b}_1) T} \int_{\frac{T}{2}}^{T} \exp\left( - \frac{T(b_2 + \bar{b}_2)}{r + \bar{r}(1-\bar{s})^2} \cdot m_u \right) dt \nonumber\\	
		&=& \kappa_1 (b_2 + \bar{b}_2)^3 e^{- \kappa_2 (b_2 + \bar{b}_2)} \leq \epsilon,
	\label{eq:I2_b2_infty}
	\end{IEEEeqnarray}
	with $\kappa_1: = \frac{2 TM^2 e^{2(b_1 + \bar{b}_1) T} }{r + \bar{r}(1-\bar{s})^2} >0$ and $\kappa_2 :=  \frac{T  m_u}{r + \bar{r}(1-\bar{s})^2}>0$ are constants independent of $b_2$. 	
	
	Let $b_2^* := \max \{ b_2^{*,I_1}, b_2^{*,I_2}\}$. Then inequalities (\ref{eq:I_1_plus_I_2}), (\ref{eq:I1_ineq_b2_infinity}) and (\ref{eq:I2_b2_infty}) give for $b_2 \geq b_2^*$:
	\begin{equation*}
		\frac{\Delta SC^{b_2}}{SC^{MKV,b_2}} \leq \frac{I_1^{b_2} + I_2^{b_2}}{(b_2 + \bar{b}_2) w_0^{b_2}} \leq \frac{ \epsilon + \epsilon }{h_w(b_2,0)} \leq \frac{2\epsilon}{c/2}=\frac{4\epsilon}{c}.
	\end{equation*}
	Since the proof holds for arbitrary $\epsilon>0$, and $c=\sqrt{(q+\bar{q}(1-s)^2)(r+\bar{r}(1-\bar{s})^2)}>0$ is independent of $b_2$ and $\epsilon$, we conclude:
	\begin{equation*}
		\lim_{b_2 \to \infty} \frac{\Delta SC^{b_2}}{SC^{MKV,b_2}} = 0, 
	\end{equation*}
	and thus, from equation (\ref{eq:PoA_LQ}):
	$$ \lim_{b_2 \to \infty} PoA^{b_2} = 1.$$\\
	
	\textbf{Case 2:} Assume:
	\begin{equation*}
	\frac{q+\bar{q}(1-s)}{r+\bar{r}(1-\bar{s})} \neq \frac{q+\bar{q}(1-s)^2}{r+\bar{r}(1-\bar{s})^2}.
	\end{equation*}
	Then $c_u \neq c_w$. We want to show that $\lim_{b_2 \to \infty}PoA^{b_2}>1$. To do so, we will show that $(b_2 + \bar{b_2}) \Delta SC^{b_2}\geq c_{num}>0$ and $b_2 SC^{MKV,b_2} \leq M_{den}<\infty$ for large $b_2$, where $c_{num}$ and $M_{den}$ are two constants independent of $b_2$. We assume in the following that $b_2 \geq b_2^{*,basic} := \max \{b_2^{*,u,lower}, b_2^{*,\eta,lower}, b_2^{*,upper} \}$, as defined prior to Case 1. Therefore, $s \mapsto u_s^{b_2}$, $s \mapsto w_s^{b_2}$ and $s \mapsto \eta_s^{\eta}$ are increasing functions. Moreover, from inequalities (\ref{eq:ut_uniform_lower_bound_b2_infinity})-(\ref{eq:ut_wt_upper_bound_b2_infty}), we have $h_u(b_2,t) \geq m_u>0$, $b_2 \eta_t^{b_2} \geq m_\eta>0$, and $\eta_t^{b_2} \leq M<\infty$, for all $t\in[0,T]$.
	
	Step 1: We derive a lower bound for $(b_2+\bar{b}_2)\Delta SC^{b_2}$ by adapting the techniques used in inequality (\ref{eq:I1_ineq_b2_infinity}).  We have shown that for every $t \in[0,T/2]$,  $\lim_{b_2 \to \infty} h_u(b_2, t)= c_u$ and $\lim_{b_2 \to \infty} h_w(b_2, t)= c_w$. Let $\zeta_3 = \ln(2) \frac{r + \bar{r}(1-\bar{s})^2}{T (c_u + \vert c_u - c_w \vert / 4)} - \bar{b}_2$. Then, there exists a $b_2^{*,num} \geq \max \{ b_2^{*,basic}, \zeta_3 \}$ such that for all $b_2 \geq b_2^{*,num}$ and all $s \in [0, T/2]$:
	\begin{equation*}
	\begin{array}{lclclclcl}
		c_u - \left\vert c_u - c_w  \right\vert /4 &\leq& h_u(b_2, 0) &\leq& h_u(b_2,s) &\leq& h_u(b_2, T/2) &\leq& c_u + \left\vert c_u - c_w \right\vert /4, \\
		c_w -  \left\vert c_u - c_w  \right\vert /4 &\leq& h_w(b_2, 0) &\leq& h_w(b_2,s) &\leq& h_w(b_2, T/2) &\leq& c_w + \left\vert c_u - c_w  \right\vert /4,
	\end{array}
	\end{equation*}
	which implies that for all $t \in [0,T/2]$:
	$$ \left\vert h_u(b_2, t) - h_w(b_2,t)\right\vert \geq \frac{1}{2} \left\vert c_u - c_w \right\vert, \qquad \int_0^t h_u(b_2, s) ds \leq \left( c_u + \frac{1}{4} \left\vert c_u - c_w \right\vert \right) t.
	$$
	Thus, similar to inequality (\ref{eq:I1_ineq_b2_infinity}), for all $b_2 \geq b_2^{*,num}$, we deduce:
	\begin{IEEEeqnarray}{rCl}
		(b_2 + \bar{b}_2) \Delta SC ^{b_2} &\geq & \frac{b_2 + \bar{b}_2}{2} B^{b_2} \int_0^{\frac{T}{2}} (u^{b_2}_t - w^{b_2}_t)^2 (\bar{x}^{MFG, b_2})^2 dt \nonumber \\
		&\geq& \frac{\left\vert c_u - c_w \right\vert^2 (\mathbb{E}(\xi))^2}{8 (r+ \bar{r}(1-\bar{s})^2 )} \cdot (b_2 + \bar{b}_2) \int_0^{\frac{T}{2}} \exp \left[- \frac{2 (b_2 + \bar{b}_2)}{r +\bar{r}(1-\bar{s})^2} \left( c_u + \frac{1}{4} \left\vert c_u - c_w \right\vert \right) t \right] dt \nonumber \\
		&= & \frac{\left\vert c_u - c_w \right\vert^2 (\mathbb{E}(\xi))^2}{16 (c_u +  \vert c_u - c_w \vert / 4)} \left( 1- e^{- \frac{T (c_u + \vert c_u - c_w \vert/4 )}{r+\bar{r}(1-\bar{s})^2}(b_2 + \bar{b}_2)} \right) \nonumber \\
		&\geq & \frac{\left\vert c_u - c_w \right\vert^2 (\mathbb{E}(\xi))^2}{16 (c_u +  \vert c_u - c_w \vert / 4)}\cdot \frac{1}{2} =: c_{num}.
	\label{eq:numerator_ineq_b2_infinity}	
	\end{IEEEeqnarray}
	
	Step 2: We derive an upper bound for $b_2 SC^{MKV,b_2}$. From equation (\ref{eq:SC_MKV_LQ_new_notation}), we have: 
	\begin{equation}
	b_2 SC^{MKV,b_2} 	=  \underbrace{ \left[ \frac{q + \bar{q}}{2} \int_0^T b_2 v_t dt + \frac{q_T + \bar{q}_T}{2} b_2 v_T + \frac{b_2}{2} w_0^{b_2} (\mathbb{E}(\xi) )^2 \right] }_\text{\large{$=:J^{b_2}_1$}}  + \underbrace{ \frac{1}{2} \int_{0}^T  B^{\eta,b_2} (\eta_t^{b_2})^2 \cdot b_2 v_t dt }_\text{\large{$=:J^{b_2}_2$}}.
	\label{eq:J1_plus_J2_b2_infinity}
	\end{equation}	
	We derive the following two results which are useful for deriving upper bounds for $J^{b_2}_1$ and $J^{b_2}_2$. First, let $\kappa_3 := \frac{2 m_{\eta}}{r + \bar{r}}>0$ and $\displaystyle l(b_2,t): = \int_0^t e^{- 2 B^{\eta,b_2} \int_s^t \eta^{b_2}_u du} ds$. Since $b_2 \eta_t \geq m_\eta$ for all $t \in [0,T]$:
	\begin{equation}
		b_2 \cdot l(b_2,t) \leq b_2 \int_0^t \exp \left(- 2 \frac{b_2}{r + \bar{r}} m_\eta \cdot (t- s) \right) ds = \frac{1}{\kappa_3} (1 - e^{- \kappa_3 b_2 t}) \leq \frac{1}{\kappa_3}, \quad \forall \  t \in [0,T].
	\label{eq:l(t)_upper_bound_b2_infinity}
	\end{equation}
	Next, using equation (\ref{eq:riccati_etat}), integration by parts, and since $-A^\eta,\eta^{b_2}_t,C^\eta,l(b_2,t)\geq 0$ we deduce:
	\begin{IEEEeqnarray*}{rCL}
		 \int_0^T B^{\eta, b_2} (\eta_t^{b_2} )^2 \cdot l(b_2,t) dt 
		&= & \int_0^T  (\eta_t^{b_2})' \cdot l(b_2,t) dt + \int_0^T \left( -2A^{\eta} \eta_t^{b_2} + C^{\eta} \right) \cdot l(b_2,t) dt \nonumber \\
		&\geq & \left[ \eta_T^{b_2} l(b_2,T) - \eta_0^{b_2} l(b_2,0) - \int_0^T \eta_t^{b_2} \cdot \frac{\partial(l(b_2,t))}{\partial t} dt \right]  \nonumber \\
		&= & - \int_0^T \eta_t^{b_2} \cdot \left[ l(b_2,t)\cdot( - 2 B^{\eta,b_2} \eta_t^{b_2} ) + 1 \right] dt + l(b_2,T) D^{\eta} \nonumber \\
		&\geq& 2 \int_0^T B^{\eta,b_2} (\eta_t^{b_2})^2 \cdot l(b_2,t) dt - \int_0^T \eta_t^{b_2} dt.
	\end{IEEEeqnarray*}
	After rearranging terms, we obtain:
	\begin{equation}
		b_2 \int_0^T B^{\eta, b_2} ( \eta_t^{b_2} )^2 \cdot \left( \int_0^t e^{-2B^{\eta,b_2} \int_{s}^{t} \eta^{b_2}_u du} ds \right) dt  \leq  \int_0^T b_2 \eta_t^{b_2} dt.
	\label{eq:l(t)_transformation_b2_infinity}
	\end{equation}
	
	First, consider $J^{b_2}_1$. From equation (\ref{eq:v_t_new_notation}) and inequalities (\ref{eq:eta_t_uniform_lower_bound_b2_infinity}) and (\ref{eq:l(t)_upper_bound_b2_infinity}), we have that for all $b_2 \geq b_2^{*,basic}$ and for all $t \in [0,T]$:
	\begin{IEEEeqnarray*}{rCl}
		b_2 v_t &= & Var(\xi) e^{2b_1 t} \cdot b_2 e^{- 2 B^{\eta,b_2} \int_0^t \eta_s ds } + \sigma^2 \cdot b_2 \int_0^t e^{2b_1 (t-s)} \cdot e^{-2B^{\eta,b_2} \int_s^t \eta_u du} ds \nonumber \\	
		&\leq& Var(\xi) e^{2b_1 T}\cdot b_2 e^{- \frac{2 b_2}{r + \bar{r}} m_\eta \cdot t} + \sigma^2 e^{2 b_1 T} b_2 \cdot l(t)  \nonumber \\
		&\leq& e^{2b_1 T} \left( Var(\xi) b_2 e^{-\kappa_3 b_2 t} + \frac{\sigma^2}{\kappa_3} \right).
	\end{IEEEeqnarray*}
	Let $b^{*,J_1} \geq b^{*,basic}$ such that
	for all $b_2 \geq b_2^{*,J_1}$:
	$$ (b_2 + \bar{b}_2) w_0^{b_2} \leq c_w + \frac{1}{2} c_w, \quad \text{and} \quad b_2 e^{- \kappa_3 b_2 T} \leq \frac{1}{\kappa_3}.$$
	Then for all $b_2 \geq b_2^{*,J_1}$:
	\begin{IEEEeqnarray}{rCl}
		J^{b_2}_1 &\leq& \frac{q+\bar{q}}{2}  \cdot e^{2b_1 T} \int_0^T \left( Var(\xi) b_2 e^{-\kappa_3 b_2 t} + \frac{\sigma^2}{\kappa_3} \right) dt  \nonumber \\
		&& \quad + \frac{q_T+\bar{q}_T}{2} \cdot e^{2b_1 T} \left( Var(\xi) b_2 e^{-\kappa_3 T \cdot b_2 } + \frac{\sigma^2}{\kappa_3} \right) + \frac{b_2}{2(b_2 + \bar{b}_2)} (\mathbb{E}(\xi))^2 \cdot (b_2 + \bar{b}_2) w_0^{b_2}\nonumber \\
		&\leq & \frac{(q+\bar{q}) e^{2b_1 T} }{2}\left[ \frac{Var(\xi)}{\kappa_3}  + \frac{\sigma^2 T}{\kappa_3} \right] + \frac{(q_T+\bar{q}_T) e^{2b_1 T}}{2} \left(  \frac{Var(\xi)}{\kappa_3} + \frac{\sigma^2}{\kappa_3} \right)  + \frac{3 c_w (\mathbb{E}(\xi))^2 }{4} \nonumber \\
		& =:& M_{J_1}.
	\label{eq:J1_upper_bound_b2_infinity}
	\end{IEEEeqnarray}

	Next, consider the quantity $J^{b_2}_2$, which can be written as:
	\begin{IEEEeqnarray*}{rCl}
		J^{b_2}_2 &=& \frac{1}{2 (r + \bar{r}) }\int_0^T \left(b_2 \eta_t^{b_2} \right)^2 Var(\xi) e^{2b_1 t} \cdot b_2 \exp \left( - 2 B^{\eta,b_2} \int_0^t \eta_s ds \right) dt \nonumber \\
		&& \hspace{10em}  + \frac{\sigma^2}{2} \cdot  b_2 \int_0^T B^{\eta,b_2} (\eta_t^{b_2})^2 \left( \int_0^t e^{2b_1(t-s) - 2B^{\eta,b_2} \int_s^t \eta_u^{b_2} du} ds \right)  dt .
	\end{IEEEeqnarray*}
	We will make use of Lemma \ref{lemma:int_b2_eta_t}, which appears below, and states that there exists a $b_2^{*,l}>0$ and a constant $M_{l}$ independent of $b_2$ such that for all $b_2 \geq b_2^{*,l}$: 
	\begin{equation}
	\int_0^T b_2 \eta_t^{b_2} dt \leq M_{l}.
	\label{eq:intgeral_b2_eta_t_bound_b2_infinity}
	\end{equation}
	Since $\lim_{b_2 \to \infty} b_2 \eta_t^{b_2}= c_\eta$ for all $t \in [0, T/2]$, there exists a $b_2^{*,J_2} \geq \max\{ b_2^{*,l}, b_2^{*,basic} \}$ such that for all $b_2 \geq b_2^{*,J_2}$:
	$$ b_2 \eta_t  \leq c_\eta + \frac{1}{2} c_\eta, \quad \forall \  t \in [0,T/2],  \quad \text{and} \quad b_2^{3} e^{-\kappa_3 b_2 \frac{T}{2} } \leq \frac{1}{\kappa_3}.$$
	Thus, together with inequalities (\ref{eq:eta_t_uniform_lower_bound_b2_infinity}), (\ref{eq:ut_wt_upper_bound_b2_infty}), (\ref{eq:l(t)_transformation_b2_infinity}), and (\ref{eq:intgeral_b2_eta_t_bound_b2_infinity}), and for $\kappa_4 := \frac{Var(\xi) e^{2b_1 T} }{2 (r+\bar{r})}$, we deduce that:
	\begin{IEEEeqnarray}{rCl}
		J^{b_2}_2 &\leq& \frac{Var(\xi) e^{2b_1 T} }{2 (r+\bar{r})} \left[ \int_0^{\frac{T}{2}} \left(b_2 \eta^{b_2}_t \right)^2 b_2 e^{- \kappa_3 b_2 t} dt  +  \int_{\frac{T}{2}}^{T} (\eta_t^{b_2})^2 b_2^3  e^{ - 2 B^{\eta, b_2} \int_0^{t} \eta_s^{b_2} ds } dt \right]  + \frac{\sigma^2  e^{2 b_1 T}}{2} \int_0^T b_2 \eta_t^{b_2} dt  \nonumber \\
		&\leq& 	\kappa_4 \left[ \left(\frac{3 c_\eta}{2} \right)^2 \frac{1}{\kappa_3}(1- e^{-\kappa_3 b_2 \frac{T}{2}}) + M^2 \cdot b_2^3 e^{- \kappa_3 b_2 \frac{T}{2}} \cdot \frac{T}{2} \right] + \frac{\sigma^2  e^{2 b_1 T}}{2} M_{l} \nonumber \\
		&\leq& \frac{\kappa_4}{\kappa_3} \left( \frac{9}{4} c_\eta^2 + \frac{M^2 T}{2} \right) + \frac{\sigma^2  e^{2 b_1 T}  M_{l}}{2} =: M_{J_2}.
	\label{eq:J2_upper_bound_b2_infinity}
	\end{IEEEeqnarray} 
	Now, from equation (\ref{eq:J1_plus_J2_b2_infinity}) and inequalities (\ref{eq:J1_upper_bound_b2_infinity}) and (\ref{eq:J2_upper_bound_b2_infinity}), we have for all $b_2 \geq b_2^{*,den} := \max \{ b_2^{*,J_1}, b_2^{*,J_2} \}$:
	\begin{equation} 
		b_2 SC^{MKV, b_2}  = J^{b_2}_1 + J^{b_2}_2 \leq M_{J_1} + M_{J_2} =: M_{den}.
	\label{eq:demumerator_ineq_b2_infinity}
	\end{equation}

	Finally, putting together inequalities (\ref{eq:numerator_ineq_b2_infinity}) and (\ref{eq:demumerator_ineq_b2_infinity}) and for $b_2^{*,case 2} := \max \{ \bar{b}_2,  b_2^{*,num}, b_2^{*,den} \}$, we have for all $b_2 \geq b_2^{*,case 2}$:
	\begin{equation*}
		\frac{b_2+\bar{b}_2}{b_2} \leq 2, \quad \text{and thus,} \quad \frac{\Delta SC^{b_2}}{SC^{MKV,b_2}} = \frac{(b_2 + \bar{b}_2) \Delta SC^{b_2} }{\left(\frac{b_2 + \bar{b}_2}{b_2} \right) (b_2 SC^{MKV,b_2} ) } \geq \frac{c_{num}}{ 2 M_{den}} > 0.
	\end{equation*}
	Therefore, from equation (\ref{eq:PoA_LQ}), we conclude: 
	$$ \lim_{b_2 \to \infty} PoA^{b_2} > 1.$$ 
\end{proof}

The following lemma was used in Proposition \ref{prop:b2_to_infty}.
\begin{lemma}
	Assume Assumption \ref{assumption}. There exists a $b_2^{*,l}>0$ and a $M^l>0$, such that for all $b_2 \geq b_2^{*,l}$:
	$$\int_0^T b_2 \eta^{b_2}_t dt \leq M^l.$$
	\label{lemma:int_b2_eta_t}
\end{lemma}

\begin{proof}
	Since $\lim_{b_2 \to \infty}b_2\eta^{b_2}_t= c_\eta,$ there exists a $b_2^{*,1}>0$ such that for $b_2 \geq b_2^{*,1}$, we have
	$b_2\eta^{b_2}_0 \leq c_\eta+\frac{1}{2}.$
	Clearly there exists a $b_2^{*,2}>0$ such that for $b_2 \geq b_2^{*,2}$:
	$$b_2\eta^{b_2}_T=b_2D^\eta=b_2(q_T+\bar{q}_T)\geq c_\eta+2.$$
	Let $b_2^{*,3}$ such that for $b_2 \geq b_2^{*,3}$, the function $t \mapsto b_2 \eta^{b_2}_t$ is increasing. Since for $b_2 \geq \max \{b_2^{*,1},b_2^{*,2},b_2^{*,3} \}$, the function $t \mapsto b_2 \eta^{b_2}_t$ is increasing and continuous, by the intermediate value theorem, there exists a $t^*_{b_2}\in[0,T]$ such that
	$b_2 \eta^{b_2}_{t^*_{b_2}} =c_\eta+1$, and $b_2 \eta^{b_2}_t \leq c_\eta+1$, for all $t\leq t^*_{b_2}.$
	From an analogous equation to (\ref{eq:riccati_ut_sol}) for $\eta^{b_2}_{t^*_{b_2}}$, we have for $b_2 \geq \max \{b_2^{*,1},b_2^{*,2},b_2^{*,3} \}$:
	$$
	1=b_2\eta^{b_2}_{t^*_{b_2}}-c_\eta\leq \frac{ \frac{C^\eta}{b_2}+D^\eta\frac{\delta^{+,b_2}_\eta}{b_2}-\frac{D^\eta}{r+\bar{r}}c_\eta +\frac{\delta^{-,b_2}_\eta}{b_2^2}c_\eta+\left(-D^\eta\frac{\delta^{-,b_2}_\eta}{b_2} +\frac{D^\eta}{r+\bar{r}}c_\eta\right)e^{-(\delta^{+,b_2}_\eta-\delta^{-,b_2}_\eta)(T-t^*_{b_2})} }{ \frac{D^\eta}{r+\bar{r}} -\frac{\delta^{-,b_2}_\eta}{b_2^2} -\frac{D^\eta}{r+\bar{r}} e^{-(\delta^{+,b_2}_\eta-\delta^{-,b_2}_\eta)(T-t^*_{b_2})}}.
	$$
	After rearranging terms:
	\begin{equation}
	\begin{split}
	\frac{D^\eta}{r+\bar{r}}-\left[\frac{\delta^{-,b_2}_\eta}{b_2^2}(1+c_\eta)+\frac{C^\eta}{b_2}+D^\eta \left(\frac{\delta^{+,b_2}_\eta}{b_2}-c_{\delta_\eta} \right)\right]\\
	\leq  \left(-D^\eta\frac{\delta^{-,b_2}_\eta }{b_2}+\frac{D^\eta}{r+\bar{r}}(1+c_\eta) \right) e^{-\left(\delta^{+,b_2}_\eta -\delta^{-,b_2}_\eta\right)(T-t^*_{b_2})}.
	\end{split}
	\label{eq:lemma_1}
	\end{equation}
	Since we have
	$\lim_{b_2 \to \infty}\frac{\delta^{-,b_2}_\eta}{b_2^2}=\lim_{b_2 \to \infty}\frac{C^\eta}{b_2}=\lim_{b_2 \to \infty}\left(\frac{\delta^{+,b_2}_\eta}{b_2}-c_{\delta_\eta} \right) =0,$ there exists a $b_2^{*,4}>0$ such that for $b_2 \geq b_2^{*,4}$,
	$\left[\frac{\delta^{-,b_2}_\eta}{b_2^2}(1+c_\eta)+\frac{C^\eta}{b_2}+D^\eta \left(\frac{\delta^{+,b_2}_\eta}{b_2}-c_{\delta_\eta} \right)\right] \leq \frac{D^\eta}{2(r+\bar{r})},$
	and thus returning to inequality (\ref{eq:lemma_1}), for $b_2 \geq \max \{b_2^{*,1},b_2^{*,2},b_2^{*,3},b_2^{*,4} \}$:
	$$\frac{D^\eta}{2(r+\bar{r})}\leq \left(-D^\eta\frac{\delta^{-,b_2}_\eta }{b_2}+\frac{D^\eta}{r+\bar{r}}(1+c_\eta) \right) e^{-\left(\delta^{+,b_2}_\eta -\delta^{-,b_2}_\eta\right)(T-t^*_{b_2})}.$$
	After rearranging terms:
	\begin{equation}
	b_2(T-t^*_{b_2})\leq \frac{1}{\frac{\left(\delta^{+,b_2}_\eta -\delta^{-,b_2}_\eta\right)}{b_2}} \ln \left( -\frac{2\delta^{-,b_2}_\eta(r+\bar{r})}{b_2}+2(1+c_\eta) \right).
	\label{eq:lemma_2}
	\end{equation}
	Since $\lim_{b_2 \to \infty}\frac{\delta^{+,b_2}_\eta}{b_2}=c_{\delta_\eta}=-\lim_{b_2 \to \infty}\frac{\delta^{-,b_2}_\eta}{b_2}$, there exists a $b_2^{*,5}>0$ such that for $b_2 \geq b_2^{*,5}$:
	$$\frac{c_{\delta_\eta}}{2} \leq \frac{-\delta^{-,b_2}_\eta}{b_2} \leq \frac{3c_{\delta_\eta}}{2}, \quad \text{and} \quad c_{\delta_\eta} \leq \frac{\delta^{+,b_2}_\eta-\delta^{-,b_2}_\eta}{b_2} \leq 3c_{\delta_\eta}.$$	
	Returning to inequality (\ref{eq:lemma_2}), for $b_2 \geq b_2^{*,l}:=\max \{b_2^{*,1},b_2^{*,2},b_2^{*,3},b_2^{*,4},b_2^{*,5}\}$:
	$$ b_2(T-t^*_{b_2})\leq \frac{1}{c_{\delta_\eta}} \ln \left(3 (r+\bar{r})c_{\delta_\eta}+2(1+c_\eta) \right)=\frac{1}{c_{\delta_\eta}} \ln \left(5c_\eta+2 \right).$$
	Finally, for $b_2 \geq b_2^{*,l}$:
	\begin{equation*}
	\begin{split}
	\int_0^T b_2 \eta^{b_2}_t dt=\int_0^{t^*_{b_2}} b_2 \eta^{b_2}_t dt+\int_{t^*_{b_2}}^T b_2 \eta^{b_2}_t dt \leq t^*_{b_2}(c_\eta+1)+b_2 (T-t^*_{b_2})D^\eta \\
	\leq T(c_\eta+1)+\frac{D^\eta}{c_{\delta_\eta}} \ln \left( 5c_\eta+2 \right)=:M^l.
	\end{split}
	\end{equation*}
\end{proof}

\begin{proposition}
	Assume Assumption \ref{assumption}. If $\bar{b}_2= 0$, then:
	\begin{equation*}
	\lim_{b_2 \to 0} PoA = 1,
	\end{equation*}
	whereas if $\bar{b}_2 > 0$, then:
	\begin{equation*}
	\lim_{b_2 \to 0} PoA > 1.
	\end{equation*}
\label{prop:b2_to_0}
\end{proposition}

\begin{proof}
	\textbf{Case 1:} First, consider the case $\bar{b}_2 =0$. As $b_2 \to 0$, we have:
	\begin{equation*}
		B^{b_2} \to 0, \quad B^{\eta,b_2} \to 0,
	\end{equation*}
	and $\lambda = \frac{r + \bar{r}(1- \bar{s})^2}{r + \bar{r}(1-\bar{s})}$, $(A^{u},C^{u}, D^{u}), (A^{w}, C^{w}, D^{w}), (A^\eta, C^\eta, D^\eta)$ are all independent of $b_2$. We can then use the same technique shown in Proposition \ref{prop:r_bar_r} to conclude that $\lim_{b_2 \to 0} PoA = 1.$\\
	
	\textbf{Case 2:} Now, let's assume $\bar{b}_2 > 0$.	As $b_2 \to 0$, we have:
	\begin{equation*}
	\begin{array}{l}
		\lambda \longrightarrow 0, \quad \displaystyle B^{b_2} \longrightarrow{}  \frac{\bar{b}_2^2}{r + \bar{r}(1-\bar{s})^2}=:B^{b_2 \to 0} >0, \qquad B^{\eta,b_2} \longrightarrow{} 0, \quad C^{u,b_2} \longrightarrow{} 0, \  D^{u,b_2} \longrightarrow{} 0,
	\end{array}
	\end{equation*}
	and $A^u, (A^w, C^w, D^w), (A^\eta, C^\eta, D^\eta)$ are independent of $b_2$.
	Moreover, we have:
	$$\lim_{b_2 \to 0}\delta^{+,b_2}_u =-2A^u >0, \quad \text{and} \quad \lim_{b_2 \to 0}\delta^{-,b_2}_u = 0.$$
	Thus, from equation (\ref{eq:riccati_ut_sol}) we deduce that for every fixed time $t \in [0,T]$, $\lim_{b_2 \to 0}u^{b_2}_t = 0.$
	
	Similar to Proposition \ref{prop:r_bar_r}, we can derive a uniform bound for $u_t^{b_2}$ over $[0,T]$ for small $b_2$. Indeed, for any fixed $\epsilon >0$ there exists a $b_2^* >0$ such that for any $b_2 \leq b_2^*$:
	$$\max \left\{B^{b_2}, C^{u,b_2}, D^{u,b_2} \right\} \leq B^{b_2 \to \infty} + \epsilon =: \zeta, \quad \text{and thus,} \quad
	\vert u_t^{b_2} \vert \leq \frac{\zeta + 2 \epsilon \sqrt{(A^u)^2 + \zeta^2} }{- 2 A^u e^{-2T \sqrt{(A^u)^2 + \zeta^2} } } ,\ \forall  \  t \in [0,T].$$
	
	From equation (\ref{eq:x_bar_MFG_explicit_new_notation}), the assumption $\mathbb{E}(\xi)\neq 0$, and by the bounded convergence theorem, we derive that for any fixed $t \in [0,T]$:
	$$\lim_{b_2 \to 0} \bar{x}_t^{MFG,b_2} = \mathbb{E}(\xi) e^{(b_1+\bar{b}_1)t} =: \bar{x}_t^{MFG,b_2 \to 0} \neq 0.$$
	It can also be shown that $\left| \bar{x}_t^{MFG,b_2}\right| \leq \vert \mathbb{E}(\xi) \vert e^{(b_1 + \bar{b}_1) T}$ for any $t \in [0,T]$ and $b_2 > 0$.

	Moreover, since $B^{b_2 \to 0}>0$, $B^{b_2 \to 0}C^w >0$, $B^{b_2\to 0}D^w > 0$, we have $\lim_{b_2 \to 0} w_t^{b_2}=:w_t^{b_2 \to 0}$, and $w_t^{b_2 \to 0}$ is strictly positive over $[0,T)$. It is easy to check that $w_t^{b_2}$ is also uniformly bounded over $[0,T]$ for small $b_2$. Hence, from equation (\ref{eq:Delta_SC_LQ_new_notation}) and the bounded convergence theorem, we deduce:
	\begin{equation*}
	\lim_{b_2 \to 0} \Delta SC^{b_2} = \frac{1}{2} B^{b_2\to 0} \int_0^T (w_t^{b_2 \to 0} \cdot \bar{x}_t^{MFG,b_2 \to 0})^2 dt >0.
	\end{equation*}

	Since $B^{\eta,b_2} \to 0, A^{\eta} < 0, C^{\eta} >0$, and $D^{\eta} > 0$, using the same argument shown in Proposition \ref{prop:r_bar_r}, 
	we deduce that $\eta_t^{b_2}$ is uniformly bounded over $[0,T]$ for small $b_2$ and for all $t \in[0,T]$:
	$$\lim_{b_2 \to 0}\eta_t^{b_2}=\left(D^{\eta} - \frac{C^{\eta}}{2A^\eta}\right) e^{-2A^\eta(T-t)} + \frac{C^\eta}{2A^{\eta} }=:\eta_t^{b_2 \to 0}.$$
	
	From equation (\ref{eq:v_t_new_notation}) and the bounded convergence theorem, for all $t \in [0,T]$:
	$$ \lim_{b_2 \to 0}v_t^{b_2}=Var(\xi) e^{2b_1 t} + \sigma^2 \int_0^t e^{2b_1 (t-s)} ds=:v_t^{b_2 \to 0}>0,$$
	and thus, $ 0 < \lim_{b_2 \to 0} SC^{MKV,b_2} < \infty.$ We conclude $\lim_{b_2 \to 0} PoA^{b_2} > 1.$
\end{proof}

\begin{proposition}
	Assuming Assumption \ref{assumption}, then:
	$$ \qquad\lim_{\bar{b}_2 \to \infty} PoA = 1.$$
	Furthermore, if	$\frac{r + \bar{r}(1- \bar{s})^2}{r + \bar{r}(1-\bar{s})} \neq \frac{q_T+\bar{q}_T(1-s_T)^2}{q_T+\bar{q}_T(1-s_T)}$ then:
	\begin{equation*}
	\lim_{\bar{b}_2 \to 0} PoA > 1. 
	\end{equation*}	
	\label{prop:b2bar}
\end{proposition}

\begin{proof}
	\textbf{Case 1:} When $\bar{b}_2 \to \infty$,  we have:
	\begin{equation*}
	\begin{array}{l}
	\lambda^{\bar{b}_2} \to 0, \quad B^{\bar{b}_2} \to \infty,\ C^{u,\bar{b}_2} \to 0,\ D^{u,\bar{b}_2} \to 0,
	\end{array}
	\end{equation*}
	and $A^{u}, (A^w,C^w,D^w), (A^\eta, B^\eta, C^\eta, D^\eta)$ are independent of $\bar{b}_2$.
	Following the same technique used in Proposition \ref{prop:b2_to_infty}, we can show that:
	$$ \lim_{\bar{b}_2 \to \infty} \frac{\delta^{\pm,\bar{b}_2}_u}{\sqrt{b_2 + \bar{b}_2}} =\pm \sqrt{ \frac{b_2 (q + \bar{q}(1-s))}{r + \bar{r}(1-\bar{s})} } =: \pm c_{\delta_u},
	\
	\lim_{\bar{b}_2 \to \infty} \frac{\delta^{\pm, \bar{b}_2}_w}{b_2 + \bar{b}_2} = \pm \sqrt{\frac{q+\bar{q}(1-s)^2}{r+ \bar{r}(1-\bar{s})^2}} =: \pm c_{\delta_w},$$
	and, for all $t \in [0,T)$:
	\begin{equation*}
	\begin{array}{l}
	\lim_{\bar{b}_2 \to \infty}(b_2 + \bar{b}_2)^{\frac{3}{2}} u^{\bar{b}_2}_t = (r + \bar{r}(1-\bar{s})^2) c_{\delta_u} =: c_u, \\
	\lim_{\bar{b}_2 \to \infty}(b_2 + \bar{b}_2) w_t^{\bar{b}_2} = (r + \bar{r}(1-\bar{s}^2) ) c_{\delta_w} =: c_w.
	\end{array}
	\end{equation*}
	
	Next, we provide a uniform upper bound for $u_t^{\bar{b}_2}$ over $[0,T]$ and large $\bar{b}_2$.
	
	\hspace{-7mm} Let $\zeta_1 := \frac{1}{2} \min \left\{ c_{\delta_u}, \  c_{\delta_w} \right\}$. Then there exists a $\bar{b}_2^{*,u} >0 $ such that for all $\bar{b}_2 \geq \bar{b}_2^{*,u}$,
	$$ \max\left\{ \left\vert \frac{\delta_u^{+,\bar{b}_2}}{\sqrt{b_2 + \bar{b}_2}} - c_{\delta_u} \right\vert , \ \left\vert \frac{\delta_u^{-,\bar{b}_2}}{\sqrt{b_2 + \bar{b}_2}} - (-c_{\delta_u})\right\vert,\    \left\vert \frac{C^{u,\bar{b}_2}}{\sqrt{b_2 + \bar{b}_2}} \right\vert, \   \left\vert D^{u,\bar{b}_2} \right\vert, \  \frac{1}{\sqrt{b_2 + \bar{b}_2}}  \right\} \leq \zeta_1. $$ 
	Then with equation (\ref{eq:riccati_ut_sol}), for any $t \in [0,T]$ and $\bar{b}_2 \geq \bar{b}_2^{*,u}$,
	\begin{equation}
	\left\vert u_t^{\bar{b}_2} \right\vert 
	\leq \frac{C^{u,\bar{b}_2}+ D^{u,\bar{b}_2} \left(\delta_u^{+,\bar{b}_2} - \delta_u^{-,\bar{b}_2} \right)}{- \delta_u^{-,\bar{b}_2}} \\
	\leq  \frac{\zeta_1 + \zeta_1 (2 c_{\delta_u} + 2 \zeta_1)}{ c_{\delta_u} - \zeta_1}.
	\label{eq:ut_uniform_upper_bound_b2_bar_infinity}
	\end{equation}
	By the same argument for $w_t^{b_2}$ and together with inequality (\ref{eq:ut_uniform_upper_bound_b2_bar_infinity}), there exists a $\bar{b}_2^{*,upper} \geq \bar{b}_2^{*,u}$ and $M>0$ such that:
	\begin{equation}
	\max \left\{ \left\vert u_t^{\bar{b}_2}\right\vert, \left\vert w_t^{\bar{b}_2} \right\vert \right\} \leq M, \quad \forall t \in[0,T],\ \bar{b}_2 \geq \bar{b}_2^{*,upper} .
	\label{eq:ut_wt_uniform_upper_bound_b2_bar_infty}
	\end{equation}
	
	Furthermore, we can get a uniform lower bound for $(b_2 + \bar{b}_2)^{\frac{3}{2}} u_t^{\bar{b}_2}$.
	
	\hspace{-7mm}Denote $\zeta_2: = \frac{b_2 (r + \bar{r}(1-\bar{s})^2) (q_T + \bar{q}_T(1-s_T))}{r + \bar{r}(1-s)}$. Then for all $t \in [0,T]$ and $\bar{b}_2 \geq \bar{b}_2^{*,u}$ we have:
	\begin{equation}
	\left\vert (b_2 + \bar{b}_2)^{\frac{3}{2}} u_t^{\bar{b}_2} \right\vert \geq    \frac{ (b_2 + \bar{b}_2)^{\frac{3}{2}} D^{u,\bar{b}_2} \delta_u^{+,\bar{b}_2}}{(\delta_u^{+,\bar{b}_2} - \delta_u^{-,\bar{b}_2}) + B^{\bar{b}_2} D^{u,\bar{b}_2} } 
	\geq  \frac{\zeta_2 \left(c_{\delta_u} - \zeta_1\right)}{ \zeta_1 \left(2 c_{\delta_u} + 2 \zeta_1\right) + \zeta_2 / (r+ \bar{r}(1-\bar{s})^2) } =: m_u.
	\label{eq:ut_lower_bound_b2_bar_infty}
	\end{equation}
	
	Now, we adapt the method used in Proposition \ref{prop:b2_to_infty} to prove $\lim_{\bar{b}_2 \to \infty} \Delta SC^{\bar{b}_2} = 0$. Consider the two quantities:
	$$ I^{\bar{b}_2}_1 := \frac{1}{2} B^{\bar{b}_2} \int_{0}^{\frac{T}{2}} (u^{\bar{b}_2}_s- w^{\bar{b}_2}_s)^2 (\bar{x}_s^{MFG, \bar{b}_2} )^2 ds \quad \text{ and } \quad I^{\bar{b}_2}_2: = \frac{1}{2} B^{\bar{b}_2} \int_{\frac{T}{2}}^T (u^{\bar{b}_2}_s- w^{\bar{b}_2}_s)^2 (\bar{x}_s^{MFG, \bar{b}_2} )^2 ds. $$
	
	Fix $\epsilon>0$. In the following, we will show that $\Delta SC^{\bar{b}_2} = I^{\bar{b}_2}_1 + I^{\bar{b}_2}_2 \leq 2\epsilon$ for large $\bar{b}_2$. First, consider $I^{\bar{b}_2}_1$. Let $\zeta_3 := c_u / 2$. Because $ \lim_{\bar{b}_2 \to \infty} B^{\bar{b}_2} (D^{u,\bar{b}_2})^2+ 2 A^{u} D^{u,\bar{b}_2} - C^{u,\bar{b}_2} >0$, from equation (\ref{eq:riccati_ut_sol_deriv}), there exists a $\bar{b}_2^{*,inc} \geq \bar{b}_2^{*,upper}$ so that for all $\bar{b}_2 \geq \bar{b}_2^{*,inc}$, the functions $s \mapsto u_s^{\bar{b}_2}$ and $s \mapsto w_s^{\bar{b}_2}$ are increasing, and so that for all $s \in [0, T/2]$:
	\begin{IEEEeqnarray*}{c}
		c_u - \zeta_3 \leq (b_2 + \bar{b}_2)^{\frac{3}{2}} u_0^{\bar{b}_2} \leq  (b_2 + \bar{b}_2)^{\frac{3}{2}} u_s^{\bar{b}_2} \leq (b_2 + \bar{b}_2)^{\frac{3}{2}} u_{\frac{T}{2}}^{\bar{b}_2} \leq c_u + \zeta_3,\\
		\left\vert (b_2 + \bar{b}_2) u_s^{\bar{b}_2} \right\vert \leq \left\vert (b_2 + \bar{b}_2) u_{\frac{T}{2}}^{\bar{b}_2} \right\vert \leq  \zeta_3, \quad \text{and} \quad 
		\left\vert (b_2 + \bar{b}_2) w_s^{\bar{b}_2} \right\vert \leq \left\vert (b_2 + \bar{b}_2) w_{\frac{T}{2}}^{\bar{b}_2} \right\vert \leq c_w + \zeta_3.
	\end{IEEEeqnarray*}
	Thus, for any $\bar{b}_2 \geq \bar{b}_2^{*,inc}$ we have:
	\begin{IEEEeqnarray}{rCl}
		I^{\bar{b}_2}_1 &=& \frac{\mathbb{E}(\xi)^2}{2 (r + \bar{r}(1-\bar{s})^2) } \int_0^{\frac{T}{2}} \left( (b_2 + \bar{b}_2) u_t^{\bar{b}_2} - (b_2 + \bar{b}_2) w_t^{\bar{b}_2} \right)^2 e^{2(b_1 + \bar{b}_1) t} \cdot e^{ -\frac{2 (b_2 + \bar{b}_2)^{1/2}}{r + \bar{r}(1-\bar{s})^2}  \int_0^t (b_2 + \bar{b}_2)^{3/2} u_s^{\bar{b}_2} ds} dt \nonumber \\
		&\leq& \kappa_1  \frac{1}{\sqrt{b_2 + \bar{b}_2} } \left( 1 - e^{- \kappa_2 (b_2 + \bar{b}_2)^{\frac{1}{2}} }\right) \xrightarrow[\bar{b}_2 \to \infty]{} 0, \nonumber
		\label{eq:I1_b2_bar_infinity}
	\end{IEEEeqnarray} 
	with $\kappa_1: = \frac{\mathbb{E}(\xi)^2 [\zeta_3^2 + (c_w + \zeta_3)^2 ] e^{2(b_1 + \bar{b}_1) \frac{T}{2}} }{2 (c_u - \zeta_3)}$ and $ \kappa_2 := \frac{(c_u - \zeta_3)T}{r+ \bar{r}(1-\bar{s})^2}$ are independent of $\bar{b}_2$. Therefore, there exists a $\bar{b}_2^{*,I_1} \geq \bar{b}_2^{*,0}$ such that for $\bar{b}_2 \geq \bar{b}_2^{*,I_1}$, we have $I_1^{\bar{b}_1} \leq \epsilon$.
	
	Now, we consider the quantity $I^{\bar{b}_2}_2$. Since $u_t^{\bar{b}_2}$ is positive over $[0,T]$ and from inequalities (\ref{eq:ut_wt_uniform_upper_bound_b2_bar_infty}) and (\ref{eq:ut_lower_bound_b2_bar_infty}), we know that for all $\bar{b}_2 \geq \bar{b}_2^{*,upper} \geq \bar{b}_2^{*,u}$ and $t \in [T/2,T]$,
	\begin{equation*}
	\left\vert u_t^{\bar{b}_2} - w_t^{\bar{b}_2} \right\vert \leq 2M,\qquad \text{and} \qquad \int_0^t (b_2 + \bar{b}_2)^{\frac{3}{2}} u_s^{\bar{b}_2} ds \geq \frac{T}{2}m_u.
	\end{equation*}
	Thus, similar to inequality (\ref{eq:I2_b2_infty}), there exists a $\bar{b}_2^{*,I_2} \geq \bar{b}_2^{*,upper}$ such that for all $\bar{b}_2 \geq \bar{b}_2^{*,I_2}$:
	\begin{equation*}
	I^{\bar{b}_2}_2  \leq  \kappa_3 (b_2 + \bar{b}_2)^2 e^{ - \kappa_4 \sqrt{b_2 + \bar{b}_2} } \leq \epsilon,
	\end{equation*}
	where  $\kappa_3: =\frac{T \mathbb{E}(\xi)^2 e^{2 (b_1 + \bar{b}_1) T} M^2 }{r + \bar{r}(1-\bar{s})^2}$ and $\kappa_4 := \frac{T m_u}{r + \bar{r}(1-\bar{s})^2}$ are independent of $\bar{b}_2$.
	
	Hence, for all $\bar{b}_2 \geq \bar{b}_2^* := \max \{ \bar{b}_2^{*,I_1}, \bar{b}_2^{*,I_2}\}$ we have:
	\begin{equation*}
	\Delta SC^{\bar{b}_2} = I^{\bar{b}_2}_1 + I^{\bar{b}_2}_2 \leq 2\epsilon.
	\end{equation*} 
	Since the proof holds for arbitrary $\epsilon > 0$, we obtain:
	\begin{equation*}
	\lim_{\bar{b}_2 \to \infty} \Delta SC^{\bar{b}_2} = 0.
	\end{equation*}
	Moreover, recall that $\eta_t$ and $v_t$ are invariant with respect to $\bar{b}_2$ and $0<v_t<\infty$ for $t>0$. Clearly we also have $w^{\bar{b}_2}_0 \geq 0$ and $\lim_{\bar{b}_2 \to \infty}w^{\bar{b}_2}_0 = 0$. Thus, we obtain:
	$0< \lim_{\bar{b}_2 \to \infty} SC^{MKV,\bar{b}_2} <\infty,$
	and conclude that:
	$$ \lim_{\bar{b}_2 \to \infty} PoA^{\bar{b}_2} = 1.$$\\
	
	\textbf{Case 2: } When $\bar{b}_2 \to 0$, we have:
	\begin{equation*}
	\begin{array}{l}
		 \lambda^{\bar{b}_2} \to \lambda^{\bar{b}_2 \to 0} := \frac{r + \bar{r}(1- \bar{s})^2}{r + \bar{r}(1-\bar{s})}, \qquad 
		 B^{\bar{b}_2} \to \frac{b_2^2}{r + \bar{r}(1-\bar{s})^2}=:B^{\bar{b}_2 \to 0} > 0,\\
		 C^{u,\bar{b}_2} \to \lambda^{\bar{b}_2 \to 0}(q+\bar{q}(1-s))=:C^{u,\bar{b}_2 \to 0} >0,\qquad 
		 D^{u,\bar{b}_2} \to \lambda^{\bar{b}_2 \to 0}(q_T+\bar{q}_T(1-s_T))=:D^{u,\bar{b}_2 \to 0} > 0,
	\end{array}
	\end{equation*}
	and $A^u, (A^w, C^w, D^w), (A^\eta, B^\eta, C^\eta, D^\eta)$ are independent of $\bar{b}_2$.	
	Let $u^{\bar{b}_2 \to 0}:[0,T]\to \mathbb{R}$ be the solution to the limiting Riccati equation:
	\begin{equation}
	\left({u}^{{\bar{b}_2 \to 0}}_t\right)' - 2A^u u_t^{\bar{b}_2 \to 0} -B^{\bar{b}_2 \to 0} (u_t^{\bar{b}_2 \to 0})^2+ C^{u,\bar{b}_2 \to 0} = 0, \qquad u_T^{\bar{b}_2 \to 0} = D^{u,\bar{b}_2 \to 0},
	\label{eq:riccati_limiting_u}
	\end{equation}
	which we recall has an explicit solution. It is easy to show directly from the explicit solutions that for every time $t\in [0,T]$, $\lim_{\bar{b}_2 \to 0}u^{\bar{b}_2}_t=u^{\bar{b}_2 \to 0}_t$. Next, our goal is to bound $u^{\bar{b}_2}_t$ uniformly over $t \in [0,T]$ for small $\bar{b}_2$, following the methodology of the proof of Proposition \ref{prop:r_bar_r}. For any $\epsilon>0$, there exists a $\bar{b}^*_2>0$ such that $\max\{ B^{\bar{b}_2},C^{u,\bar{b}_2},D^{u,\bar{b}_2} \}<\max \{B^{\bar{b}_2 \to 0},C^{u,\bar{b}_2 \to 0},D^{u,\bar{b}_2 \to 0} \}+\epsilon=:\zeta_4$ for all $\bar{b}_2 \leq \bar{b}^*_2$. Thus, for all $\bar{b}_2\leq \bar{b}^*_2$ and for every $t \in [0,T]$:
	\begin{equation*}
		\left \vert u_t^{\bar{b}_2} \right\vert \leq \frac{\zeta_4 + 2 \zeta_4 \sqrt{(A^u)^2 + \zeta_4^2}}{-2A^u e^{-2T \sqrt{(A^u)^2 + \zeta_4^2}}}.
	\end{equation*}
	
	Similarly, for every time $t\in [0,T]$, $\lim_{\bar{b}_2 \to 0}w^{\bar{b}_2}_t=w^{\bar{b}_2 \to 0}_t$, and  $w^{\bar{b}_2}$ is uniformly bounded over $[0,T]$ and small $\bar{b}_2$. From equation (\ref{eq:x_bar_MFG_explicit_new_notation}), the assumption $\mathbb{E}(\xi)\neq 0$, and by the bounded convergence theorem, we have for every $t \in [0,T]$:
   	\begin{equation}
   	\lim_{\bar{b}_2 \to 0} \bar{x}_t^{MFG,\bar{b}_2} = \mathbb{E}(\xi) e^{\int_0^t (b_1 + \bar{b}_1 - B^{\bar{b}_2 \to 0}u^{\bar{b}_2 \to 0}_s) ds} =: \bar{x}^{MFG,\bar{b}_2 \to 0}_t\neq 0.
   	\label{eq:x_bar_neq_0}
   	\end{equation}
	Moreover, $\bar{x}^{MFG,\bar{b}_2}$ is uniformly bounded for all $\bar{b}_2 \leq \bar{b}^*_2$ and for all $t \in [0,T]$. From the non-negativity of $u_t$, we have:
	$$ \left\vert \bar{x}^{MFG,\bar{b}_2}_t \right\vert \leq \left\vert \mathbb{E}(\xi) \right\vert e^{ (b_1 + \bar{b}_1 ) T }, \quad \forall t \in [0,T], \ \forall \bar{b}_2 \leq \bar{b}^*_2. $$ 
   	
   	By the assumption $\frac{r + \bar{r}(1- \bar{s})^2}{r + \bar{r}(1-\bar{s})} \neq \frac{q_T+\bar{q}_T(1-s_T)^2}{q_T+\bar{q}_T(1-s_T)}$, we have $D^{u,\bar{b}_2 \to 0} \neq D^{w}$, and thus by continuity, $u^{\bar{b}_2 \to 0}_t\neq w^{\bar{b}_2 \to 0}_t$ on a set of positive Lebesgue measure. 
   	 Thus, by the bounded convergence theorem, we deduce:
	\begin{equation}
	\lim_{\bar{b}_2 \to 0} \Delta SC^{\bar{b}_2} = \frac{1}{2} B^{\bar{b}_2 \to 0} \int_0^T \left( (u_t^{\bar{b}_2 \to 0} - w_t^{\bar{b}_2 \to 0}\right)^2 \left (\bar{x}_t^{MFG,\bar{b}_2 \to 0} \right)^2 dt \  > 0.
	\label{eq:prop_6_bounded_conv}
	\end{equation}	
	
	Meanwhile, $\eta_t$ does not depend on $\bar{b}_2$, and therefore, the variance $v_t$ also does not depend on $\bar{b}_2$. Clearly $0<v_t<\infty$ for $t>0$ and $0 \leq w^{\bar{b}_2 \to 0}_0 < \infty$, and thus,
	$ 0<\lim_{\bar{b}_2 \to 0}SC^{MKV,\bar{b}_2} <\infty.$
	Hence, we deduce:
	$$ \lim_{\bar{b}_2 \to 0} PoA^{\bar{b}_2} > 1. $$
\end{proof}

\begin{remark}
	Consider Assumption \ref{assumption} and the case when $\bar{b}_2$ tends to zero. We have $0< SC^{MKV, \bar{b}_2 \to 0} < \infty$, and therefore, $\lim_{\bar{b}_2 \to 0} PoA^{\bar{b}_2 } =1$ if and only if $\lim_{\bar{b}_2 \to 0} \Delta SC^{\bar{b}_2} = 0$. Since we can pass the limit as in equation (\ref{eq:prop_6_bounded_conv}), we have an analogous result as Theorem \ref{thm:PoA_equivalent_1} but for the limiting coefficients $A^u$, $A^w$, $B^{\bar{b}_2 \to 0}$, $C^{u,\bar{b}_2 \to 0}$, $C^w$, $D^{u,\bar{b}_2 \to 0}$, and $D^w$. Therefore, the assumption in Proposition 4 Case 2, $\frac{r + \bar{r}(1- \bar{s})^2}{r + \bar{r}(1-\bar{s})} \neq \frac{q_T+\bar{q}_T(1-s_T)^2}{q_T+\bar{q}_T(1-s_T)}$, which is equivalent to $D^{u,\bar{b}_2 \to 0} \neq D^w$, is sufficient, but not necessary, in order to have $\lim_{b_2 \to 0}PoA^{b_2}>1$. 
\label{remark_6}
\end{remark}

\begin{proposition}\label{prop:b1_b1bar}
	Assuming Assumption \ref{assumption}, then:
	\begin{equation*}
	\lim_{b_1 \to \infty} PoA = 1 \qquad \text{and} \qquad\lim_{\bar{b}_1 \to \infty} PoA =\infty.
	\end{equation*}	
	Furthermore, if:
	\begin{equation*}
	\frac{b_2}{b_2+\bar{b}_2}\cdot \frac{r + \bar{r}(1- \bar{s})^2}{r + \bar{r}(1-\bar{s})}\cdot (q_T+\bar{q}_T(1-s_T)) \neq q_T+\bar{q}_T(1-s_T)^2,
	\end{equation*}
	then:
	\begin{equation*}
	    \lim_{b_1 \to 0} PoA > 1\qquad \text{and} \qquad \lim_{\bar{b}_1 \to 0} PoA > 1.
	\end{equation*}
\end{proposition}
\begin{proof}
	\textbf{Case 1:} Consider $b_1 \to \infty$. Since $\lim_{b_1 \to \infty} A^{u,b_1} = \lim_{b_1 \to \infty} A^{w,b_1} = - \infty$ we have:
	$$\lim_{b_1 \to \infty}B (D^{u})^2 + 2 A^{u,b_1} D^{u} - C^{u} = \lim_{b_1 \to \infty}B (D^{w})^2 + 2 A^{w,b_1} D^{w} - C^{w}=-\infty.$$
	Denote $g_u(b_1,t): = B \frac{u_t^{b_1}}{b_1}$ and $g_w(b_1,t) := B \frac{w_t^{b_1}}{b_1}$. For all $t \in [0,T)$, we have the limits:
	\begin{equation}
		\lim_{b_1 \to \infty} g_u(b_1,t) =  \lim_{b_1 \to \infty} g_w(b_1,t) = 2.
	\label{eq:ut_wt_lim_b1_infty}
	\end{equation}
	From equation (\ref{eq:riccati_ut_sol_deriv}), and together with equation (\ref{eq:ut_wt_lim_b1_infty}), there exists a $b_1^{*,upper} >0$ such that for all $b_1 \geq b_1^{*,upper}$, the functions $t \mapsto u_t^{b_1}$ and $t \mapsto w_t^{b_1}$ are decreasing and such that:
	$$ \sup_{0 \leq t \leq T} \left\{ g_u(b_1, t), g_w(b_1,t)\right\} \leq \max \left\{ g_u(b_1, 0), g_w(b_1,0) \right\} \leq 3.$$
	
	Fix $\epsilon>0$. There exists a $b^{*,I_1}_1\geq b_1^{*,upper}$ such that for all $b_1 \geq b_1^{*,I_1} $, and for all $t \in [0, 3T/4]$,
	\begin{equation*}
	\begin{split}
	\vert g_u(b_1, t) - g_w(b_1,t) \vert  &=\max \{g_u(b_1, t)-g_w(b_1, t),g_w(b_1, t)-g_u(b_1, t) \}\\
	&\leq \max \{g_u(b_1, 0)-g_w(b_1, 3T/4),g_w(b_1, 0)-g_u(b_1, 3T/4) \} \\
	&\leq \epsilon,
	\end{split}
	\end{equation*}
	and
	$$g_w(b_1,0) \geq 2 - \frac{1}{3} = \frac{5}{3},\quad  g_u(b_1,t) \geq g_u(b_1, 3T/4) \geq \frac{5}{3}.$$
	We adapt the methodology in Proposition \ref{prop:b2_to_infty} and split the interval $[0,T]$ into two parts: $[0, 3T/4]$ and $[3T/4, T]$. For all $b_1 \geq b_1^{*,I_1}$, similar to inequality (\ref{eq:I_1_plus_I_2}), we have:
	\begin{equation*}
		\frac{\Delta SC^{b_1}}{SC^{MKV,b_1}} \leq \frac{I^{b_1}_1 + I^{b_1}_2 }{g_w(b_1,0)},
	\end{equation*}
	where
	\begin{IEEEeqnarray*}{rCl}
	I^{b_1}_1 &=& b_1 \int_0^{\frac{3}{4}T} \left( g_u(b_1,t) - g_w(b_1,t) \right)^2 e^{2\bar{b}_1 t} \cdot \exp \left( 2b_1 t - 2 b_1 \int_0^t g_u(b_1,s) ds \right) dt \nonumber \\
	&\leq & \epsilon^2 e^{2\bar{b}_1 \cdot \frac{3}{4} T } \cdot b_1 \int_0^{\frac{3}{4}T} \exp \left( 2b_1 t - 2 b_1 \cdot \frac{5}{3} t \right) dt  \nonumber \\
	&= & \epsilon^2 e^{\frac{3\bar{b}_1 T}{2} } \cdot \frac{3}{4}(1-e^{-b_1T})  \nonumber \\	
	& \leq & \kappa_1 \epsilon^2 ,
	\end{IEEEeqnarray*} 	
	in which $\kappa_1 := \frac{3}{4} e^{\frac{3\bar{b}_1 T}{2} } $, and
	\begin{IEEEeqnarray*}{rCl}
		 I^{b_1}_2 &=& b_1 \int_{\frac{3}{4}T}^T \left( g_u(b_1,t) - g_w(b_1,t) \right)^2 e^{2\bar{b}_1 t} \cdot \exp \left( 2b_1 t - 2 b_1 \int_0^t g_u(b_1,s) ds \right) dt \nonumber \\
		 &\leq & b_1 (3+3)^2 e^{2 \bar{b}_1 T} \cdot \int_{\frac{3}{4}T}^{T}  \exp \left( 2b_1 t - 2b_1 \int_0^{\frac{3}{4}T} g_u(b_1,s) ds \right) dt \nonumber \\
		 &\leq & 36 e^{2 \bar{b}_1 T} b_1 \cdot \frac{T}{4} \exp\left(2b_1 \cdot T - 2 b_1 \cdot  \frac{3}{4} T \cdot \frac{5}{3} \right) \nonumber \\
		 &=&  9 T e^{2 \bar{b}_1 T} b_1 e^{- \frac{1}{2} b_1 T} \xrightarrow[b_1 \to \infty]{}0.
	\end{IEEEeqnarray*}
	Therefore, there exists a $b_1^{*} \geq b_1^{*,I_1}$, so that for $b_1 \geq b_1^*$, we have $I^{b_1}_2 \leq \epsilon$. Hence, for all $b_1 \geq b_1^{*}$,
	$$ 	\frac{\Delta SC^{b_1}}{SC^{MKV,b_1}} \leq \frac{I^{b_1}_1 + I^{b_1}_2 }{g_w(b_1,0)} \leq \frac{3}{5} (\kappa_1 \epsilon^2 + \epsilon).$$
	Since the proof holds for arbitrary $\epsilon>0$, and $\kappa_1 = \frac{3}{4} e^{\frac{3\bar{b}_1 T}{2} }$ is independent of $b_1$ and $\epsilon$, we conclude that:
	$$ \lim_{b_1 \to \infty} PoA^{b_1} = 1.$$ \\
	
	\textbf{Case 2:} Now, consider $\bar{b}_1 \to \infty$.
	Since $A^{u,\bar{b}_1}<0$ and $\lim_{\bar{b}_1 \to \infty} |A^{u,\bar{b}_1}|=\infty$, we have the following limits:
	\begin{equation*}
	\delta^{+,\bar{b}_1}_u-\delta^{-,\bar{b}_1}_u=2 \sqrt{(A^{u,\bar{b}_1})^2+BC^u} \xrightarrow[\bar{b}_1 \to \infty]{}\infty, \quad -\delta^{-,\bar{b}_1}_u=A^{u,\bar{b}_1}+\sqrt{(A^{u,\bar{b}_1})^2+BC^u}\xrightarrow[\bar{b}_1 \to \infty]{}0,
	\end{equation*}
	\begin{equation*}
	\frac{\delta^{+,\bar{b}_1}_u}{\bar{b}_1}=\frac{b_1+\frac{\bar{b}_1}{2}}{\bar{b}_1} +\frac{\sqrt{\left(b_1+\frac{\bar{b}_1}{2}\right)^2+BC^u}}{\bar{b}_1}\xrightarrow[\bar{b}_1 \to \infty]{} \frac{1}{2}+\frac{1}{2}=1.
	\end{equation*}
	We also have for $t\in [0,T)$:
	$$\delta^{+,\bar{b}_1}_ue^{-(\delta^{+,\bar{b}_1}_u-\delta^{-,\bar{b}_1}_u)(T-t)}\leq(\delta^{+,\bar{b}_1}_u-\delta^{-,\bar{b}_1}_u)e^{-(\delta^{+,\bar{b}_1}_u-\delta^{-,\bar{b}_1}_u)(T-t)}
	\xrightarrow[\bar{b}_1 \to \infty]{}0,$$
	which implies:
	$
	\lim_{\bar{b}_1\to \infty}\delta^{+,\bar{b}_1}_ue^{-(\delta^{+,\bar{b}_1}_u-\delta^{-,\bar{b}_1}_u)(T-t)}=0.
	$
	Therefore, for $t \in [0,T)$:
	$\lim_{\bar{b}_1 \to \infty}\frac{u^{\bar{b}_1}_t}{\bar{b}_1}=\frac{1}{B}$. By the same argument, we have for $t \in [0,T)$: $\lim_{\bar{b}_1 \to \infty}\frac{w^{\bar{b}_1}_t}{\bar{b}_1}=\frac{2}{B}$.

	Since $\lim_{\bar{b}_1 \to \infty}B (D^{u})^2 + 2 A^{u,\bar{b}_1} D^{u} - C^{u}=-\infty$ and $\lim_{\bar{b}_1 \to \infty}B (D^{w})^2 + 2 A^{w,\bar{b}_1} D^{w} - C^{w}=-\infty$,   from equation (\ref{eq:riccati_ut_sol_deriv}) there exists a $\bar{b}_1^{*,lower}$ such that for $\bar{b}_1 \geq \bar{b}_1^{*,lower}$:
	$$\max \left\{ \left|\frac{u^{\bar{b}_1}_0}{\bar{b}_1}-\frac{1}{B} \right|,\left|\frac{w^{\bar{b}_1}_{T/2}}{\bar{b}_1}-\frac{2}{B} \right| \right\} \leq \frac{1}{4B},$$
	and such that the functions $t \mapsto u_t^{\bar{b}_1}$ and $ t \mapsto w_t^{\bar{b}_1}$ are decreasing. Thus, for all $t \in [0, T/2]$ we have:
	\begin{equation*}
	    0< \frac{u_t^{\bar{b}_1}}{\bar{b}_1} \leq \frac{u_0^{\bar{b}_1}}{\bar{b}_1} \leq \frac{1}{B} + \frac{1}{4B}  \leq \frac{2}{B} - \frac{1}{4 B} \leq \frac{w_{T/2}^{\bar{b}_1}}{ \bar{b}_1} \leq \frac{w_t^{\bar{b}_1}}{\bar{b}_1}. 
	\end{equation*}
	Thus, for all $\bar{b}_1 \geq \bar{b}_1^{*,lower}$ and all $t \in [0,T/2]$,
	\begin{equation}
	  \left\vert  \frac{w^{\bar{b}_1}_t}{\bar{b}_1}-\frac{u^{\bar{b}_1}_t}{\bar{b}_1} \right\vert \geq \frac{1}{2B}.
	\label{eq:lower_bound_u_w}
	\end{equation}
	
	Note that $\eta_t$, and therefore, $v_t$, are independent of $\bar{b}_1$. Thus,
	\begin{equation*}
	\frac{1}{\bar{b}_1^2}SC^{MKV,\bar{b}_1}=\frac{1}{2\bar{b}_1^2}\left[\int_0^T \left[q+\bar{q} + B^{\eta} \eta_t^2 \right] v_t dt + (q_T + \bar{q}_T) v_T\right] + \frac{w^{\bar{b}_1}_0}{2\bar{b}_1^2} (\mathbb{E}(\xi))^2 \xrightarrow[\bar{b}_1 \to \infty]{}0.	
	\end{equation*}

	Now, consider:
	\begin{equation}
	\frac{1}{\bar{b}_1^2}\Delta SC^{\bar{b}_1}\geq \frac{B}{2}(\mathbb{E}(\xi))^2\int_{0}^{\frac{T}{2}} \left(\frac{u_t^{\bar{b}_1}}{\bar{b}_1} - \frac{w_t^{\bar{b}_1}}{\bar{b}_1}\right)^2 e^{2\int_0^t (b_1+\bar{b}_1-B u_s^{\bar{b}_1})ds} dt.
	\label{eq:prop_7_case_4_numerator}
	\end{equation}
	Since $u^{\bar{b}_1}_t$ is decreasing for $\bar{b}_1 \geq \bar{b}_1^{*,lower}$, we have $\bar{b}_1-B u_t^{\bar{b}_1} \geq \bar{b}_1-B u_0^{\bar{b}_1}$. We have the following limits for $t\in[0,T)$:
	\begin{equation*}
	\lim_{\bar{b}_1 \to \infty}\bar{b}_1-\delta^{+,\bar{b}_1}_u=-2b_1, \quad \lim_{\bar{b}_1 \to \infty}-\delta^{-,\bar{b}_1}_u\bar{b}_1=BC^u,
	\end{equation*}
	\begin{equation*}
	\lim_{\bar{b}_1 \to \infty}\delta^{+,\bar{b}_1}_u\bar{b}_1e^{-(\delta^{+,\bar{b}_1}_u-\delta^{-,\bar{b}_1}_u)(T-t)}=0, \quad \lim_{\bar{b}_1 \to \infty}(\delta^{-,\bar{b}_1}_u-\bar{b}_1)e^{-(\delta^{+,\bar{b}_1}_u-\delta^{-,\bar{b}_1}_u)(T-t)}=0,
	\end{equation*}
	and thus,
	\begin{equation*}
	\begin{split}
	\bar{b}_1-Bu^{\bar{b}_1}_0&=\frac{-BC^u+BD^u(\bar{b}_1-\delta^{+,\bar{b}_1}_u)-\delta^{-,\bar{b}_1}_u\bar{b}_1+(\delta^{+,\bar{b}_1}_u\bar{b}_1+BC^u+BD^u(\delta^{-,\bar{b}_1}_u-\bar{b}_1))e^{-(\delta^{+,\bar{b}_1}_u-\delta^{-,\bar{b}_1}_u)T}}{BD^u-\delta^{-,\bar{b}_1}_u+(\delta^{+,\bar{b}_1}_u-BD^u) e^{-(\delta^{+,\bar{b}_1}_u-\delta^{-,\bar{b}_1}_u)T}} \\
	&\xrightarrow[\bar{b}_1 \to \infty]{} -2b_1 .
	\end{split}
	\end{equation*}
	Since $\lim_{\bar{b}_1 \to \infty}(\bar{b}_1-B u_0^{\bar{b}_1}) = -2b_1 < 0$, there exists a $\bar{b}_1^{*}\geq\bar{b}_1^{*,lower}$, such that for $\bar{b}_1 \geq \bar{b}_1^{*}$, $(\bar{b}_1-B u_0^{\bar{b}_1}) \geq -3b_1$. Returning to inequality (\ref{eq:prop_7_case_4_numerator}), and using inequality (\ref{eq:lower_bound_u_w}) we have for $\bar{b}_1 \geq \bar{b}^*_1$:
	\begin{equation*}
	\frac{1}{\bar{b}_1^2}\Delta SC^{\bar{b}_1}
	\geq \frac{B}{2}\mathbb{E}(\xi)^2 \cdot \frac{1}{4B^2} \cdot \int_{0}^{\frac{T}{2}} e^{2b_1 t + 2(\bar{b}_1 - B u_0^{\bar{b}_1}) t } dt \geq \frac{\mathbb{E}(\xi)^2}{8 B} \int_0^{\frac{T}{2}} e^{-4b_1 t} dt  > 0.
	\end{equation*}
	Therefore,
	$
	\lim_{\bar{b}_1 \to \infty}\frac{1}{\bar{b}_1^2}\Delta SC^{\bar{b}_1}>0,
	$
	and thus,
	\begin{equation*}
	\begin{split}
	\lim_{\bar{b}_1 \to \infty}\frac{\Delta SC^{\bar{b}_1}}{SC^{MKV,\bar{b}_1}}=\lim_{\bar{b}_1 \to \infty}\frac{\frac{1}{\bar{b}_1^2}\Delta SC^{\bar{b}_1}}{\frac{1}{\bar{b}_1^2}SC^{MKV,\bar{b}_1}}=\infty.
	\end{split}
	\end{equation*}
	We conclude $\lim_{\bar{b}_1 \to \infty} PoA^{\bar{b}_1}=\infty.$\\

	\textbf{Cases 3 and 4:} First, we consider $b_1 \to 0$. We have $A^{u,b_1} \to A^{u,b_1\to 0}$, $\delta^{+,b_1}_u \to \delta^{+,b_1\to 0}_u>0$, and $\delta^{-,b_1}_u \to \delta^{-,b_1\to 0}_u<0$, and similarly for $A^{w,b_1}$, $A^{\eta,b_1}$, $\delta^{\pm,b_1}_w$, and $\delta^{\pm,b_1}_\eta$. Clearly we have for all $t\in[0,T]$, $\lim_{b_1 \to 0}u^{b_1}_t=:u^{b_1\to 0}_t$, $\lim_{b_1 \to 0}w^{b_1}_t=:w^{b_1 \to 0}_t$, and $\lim_{b_1 \to 0}\eta^{b_1}_t=:\eta^{b_1 \to 0}_t$. Next, we show that the three sequences are uniformly bounded. Let $0<\epsilon<-\delta^{-,b_1\to 0}_u$. There exists a $b_1^*>0$ such that $\max \left \{\left| \delta^{+,b_1}_u-\delta^{+,b_1\to0}_u \right|,\left| \delta^{-,b_1}_u-\delta^{-,b_1\to0}_u \right| \right\}<\epsilon$ for all $b_1 \leq b_1^*$. Then for all $b_1 \leq b_1^*$ and $t\in[0,T]$:
	\begin{equation*}
	\left|u^{b_1}_t\right| \leq \frac{C^u+D^u\left(\delta^{+,b_1}_u-\delta^{-,b_1}_u\right) }{-\delta^{-,b_1}_u} \\
	\leq \frac{C^u+D^u\left(\delta^{+,b_1\to 0}_u-\delta^{-,b_1 \to 0}_u+2\epsilon\right) }{-\delta^{-,b_1\to 0}_u-\epsilon},
	\end{equation*}
	and similarly for $\left|w^{b_1}_t\right|$ and $\left|\eta^{b_1}_t\right|$. From the assumption $\frac{b_2}{b_2+\bar{b}_2}\cdot \frac{r + \bar{r}(1- \bar{s})^2}{r + \bar{r}(1-\bar{s})}\cdot q_T+\bar{q}_T(1-s_T) \neq q_T+\bar{q}_T(1-s_T)^2$, we have $D^{u} \neq D^{w}$ and thus by continuity, $u^{b_1\to 0}_t \neq w^{b_1\to 0}_t$ on a set of positive Lebesgue measure.
	
	From equation (\ref{eq:x_bar_MFG_explicit_new_notation}), the assumption $\mathbb{E}(\xi)\neq 0$, and by the bounded convergence theorem, we have for every $t \in [0,T]$:
	\begin{equation*}
	\lim_{b_1\to 0}\bar{x}_t^{MFG,b_1} = \mathbb{E}(\xi) e^{\int_0^t(\bar{b}_1-B u^{b_1\to 0}_s)ds} =: \bar{x}_t^{MFG,b_1\to 0}\neq 0.
	\end{equation*}
	Moreover, $\bar{x}_t^{MFG,b_1}$ is uniformly bounded over $b_1 \leq b_1^*$ and $t \in [0,T]$, i.e. 
	$$ \left\vert \bar{x}^{MFG,b_1}_t \right\vert \leq \left\vert \mathbb{E}(\xi) \right\vert e^{ \bar{b}_1 T }, \quad \forall t \in [0,T],\ b_1 \leq b_1^*. $$ 
	
	Therefore, by the bounded convergence theorem, $0<\lim_{b_1 \to 0} \Delta SC^{b_1}<\infty$ and $\lim_{b_1 \to 0}v^{b_1}_t=:v_t^{b_1 \to 0}$, which is bounded over $t\in [0,T]$. Thus, $0<\lim_{b_1 \to 0} SC^{MKV,b_1}<\infty$ and we conclude $\lim_{b_1 \to 0} PoA^{b_1} > 1$. The proof can be repeated to show $\lim_{\bar{b}_1 \to 0} PoA^{\bar{b}_1}  > 1$.
\end{proof}

\subsection{\textbf{Numerical Results}}
The price of anarchy for the class of linear quadratic extended mean field games that we consider is given by the ratio of the two quantities given by equations (\ref{eq:SC_MFG_1}) and (\ref{eq:SC_MKV_1}), which are explicit, up to evaluating integrals. Using the simple rectangle rule to estimate integrals, we numerically compute the price of anarchy when the coefficients are time-independent, non-negative, and satisfy Assumption \ref{assumption}. In particular, when we allow for full interaction (i.e. through the states and the controls), we choose the following default values:

\begin{equation*}
\begin{split}
    &\xi\equiv 1,\ T=1,\ b_1=1,\  \bar{b}_1=1,\ b_2=1,\ \bar{b}_2=1,\ \sigma=1, \\
    &q=1,\ \bar{q}=1,\ s=0.5,\ r=1,\ \bar{r}=1,\ \bar{s}=0.5,\ q_T=1,\  \bar{q}_T=1,\ s_T=0.5.
\end{split}
\end{equation*}

Unless otherwise stated, the parameters stay at these default values. For results involving only interaction through the states, we set $\bar{b}_2=0$ and $\bar{r}=0$. For results involving only interaction through the controls, we set $\bar{b}_1=0$, $\bar{q}=0$, and $\bar{q}_T=0$. Figures \ref{fig:r_rbar}-\ref{fig:b1bar} show the price of anarchy as we vary one parameter at a time for each of three interaction cases: full interaction (i.e. through the states and the controls), interaction only through the states, and interaction only through the controls. The results show various limiting behaviors, such as some of the cases proved in the previous section.

\begin{figure}[!htb]
	\centering
	\begin{subfigure}{.45\textwidth}
		\includegraphics[scale=0.5]{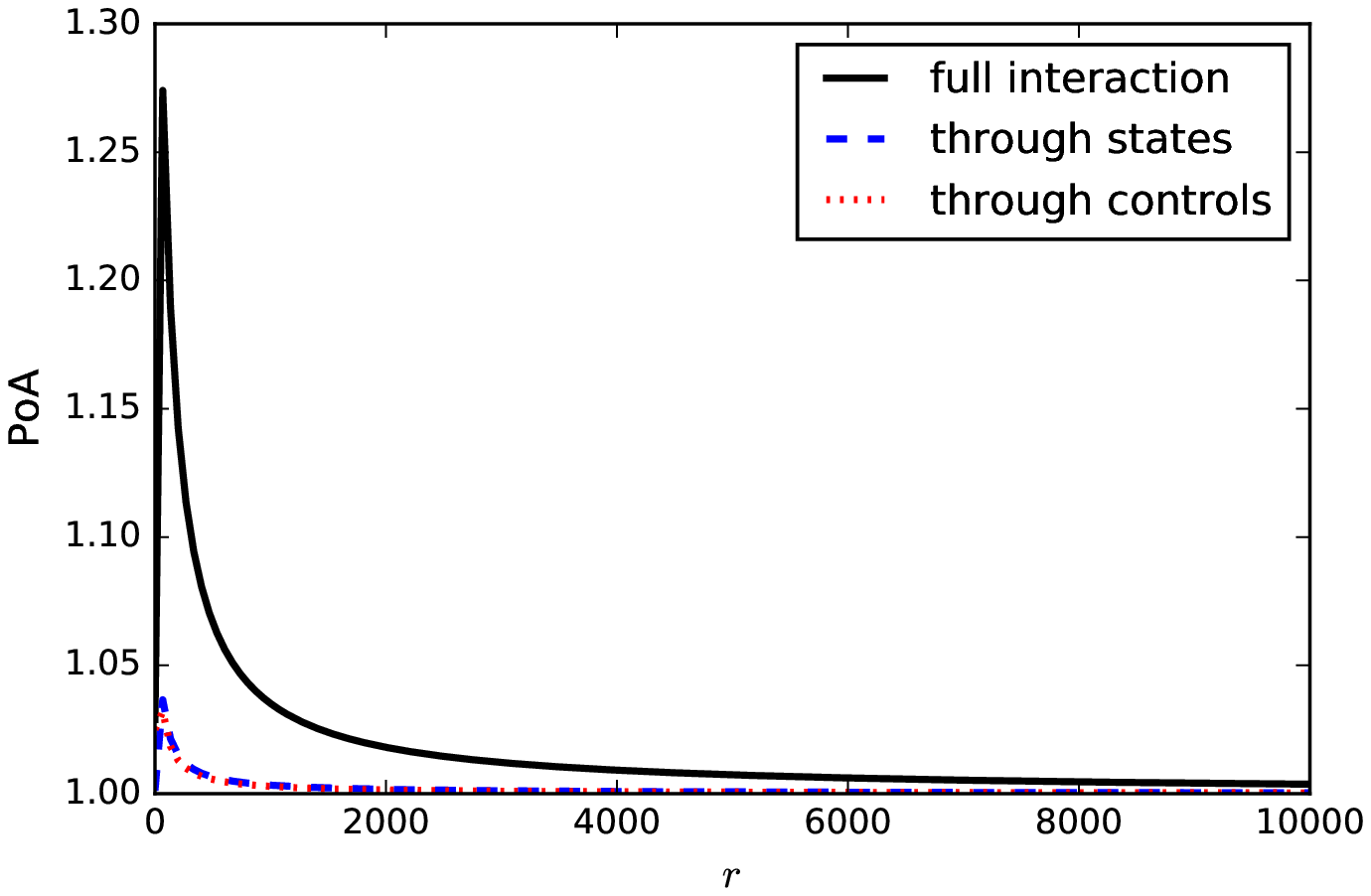}
	\end{subfigure}
	\begin{subfigure}{.45\textwidth}
		\includegraphics[scale=0.5]{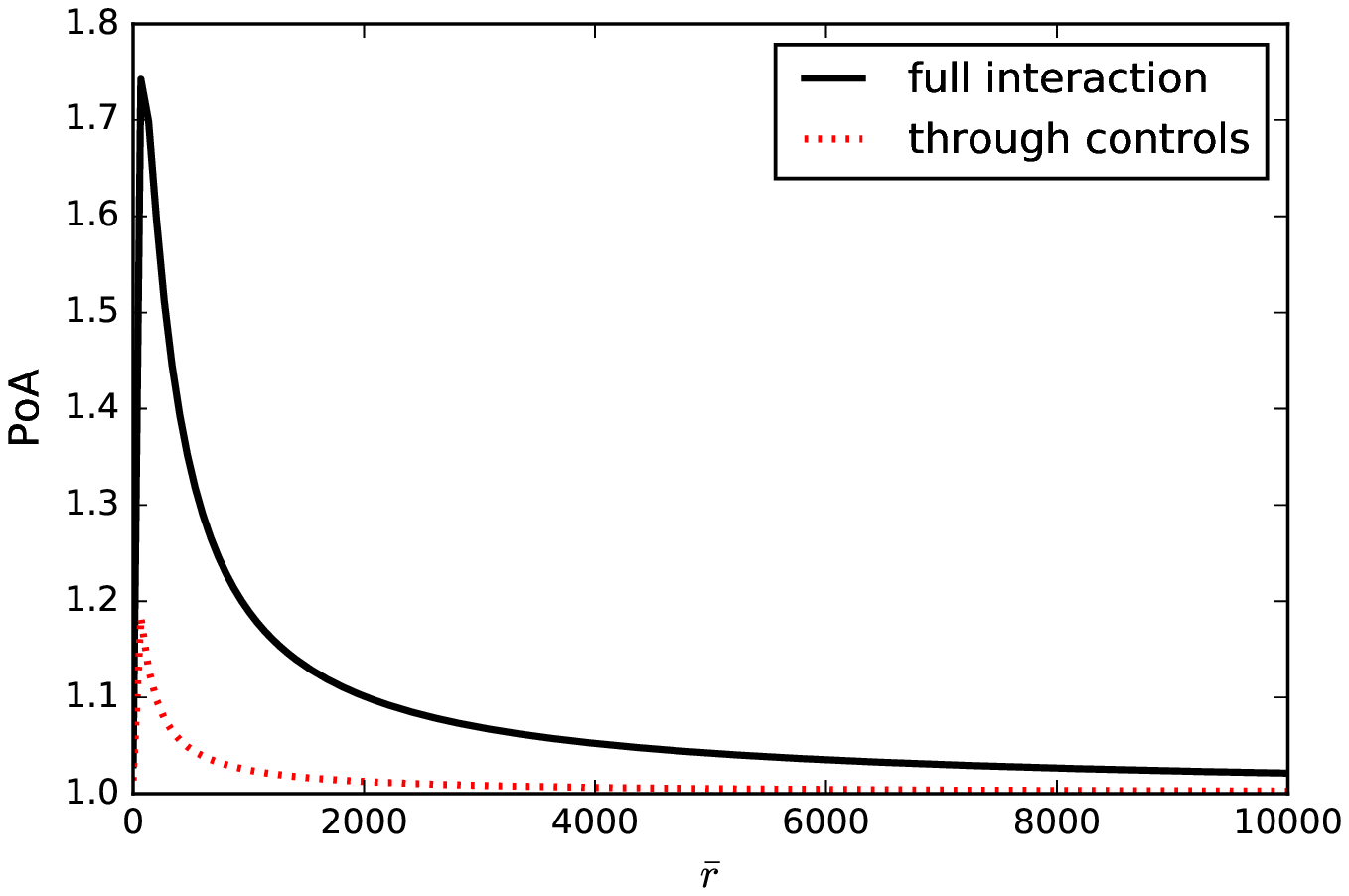}
	\end{subfigure}
	\caption{$PoA$ as we vary $r$ (left) and $\bar{r}$ (right).}
	\label{fig:r_rbar}
\end{figure}
\begin{figure}[!htb]
    \centering
    \begin{subfigure}{.45\textwidth}
        \includegraphics[scale=0.5]{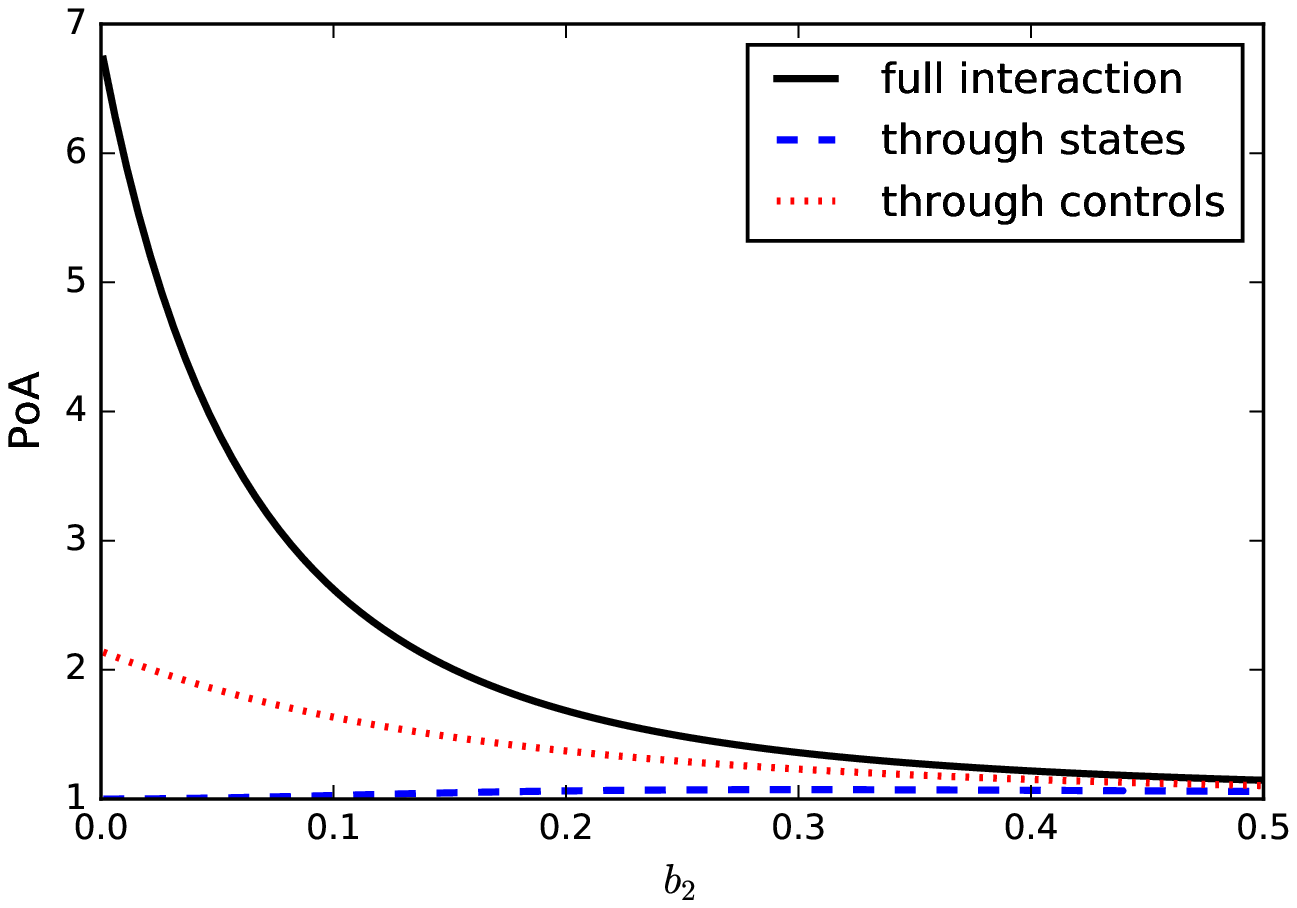}
    \end{subfigure}
    \begin{subfigure}{.45\textwidth}
        \includegraphics[scale=0.5]{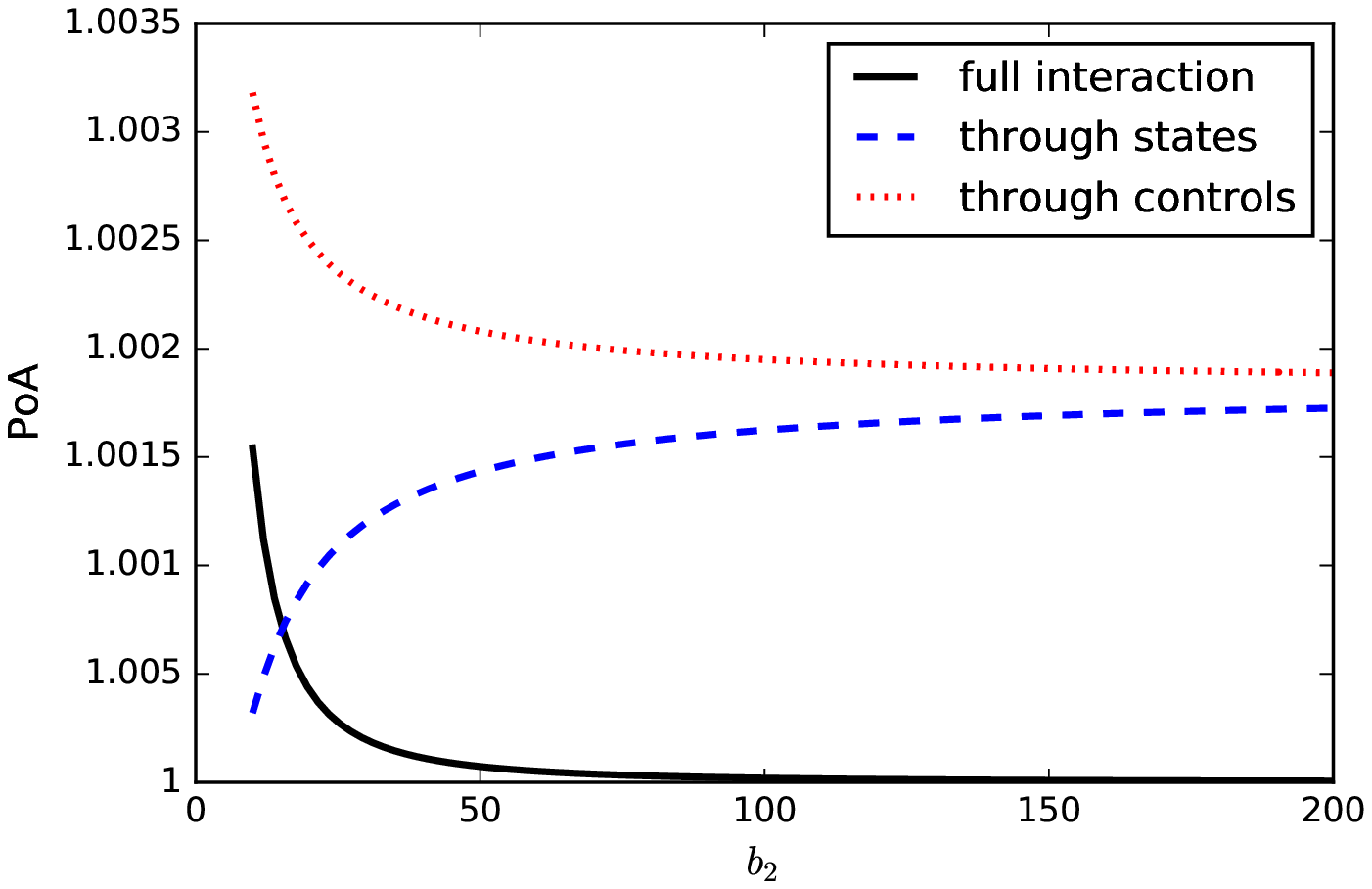}
    \end{subfigure}
    \caption{$PoA$ as we vary $b_2$.}
    \label{fig:b2}
\end{figure}
\begin{figure}[!htb]
	\centering
	\begin{subfigure}{.45\textwidth}
		\includegraphics[scale=0.5]{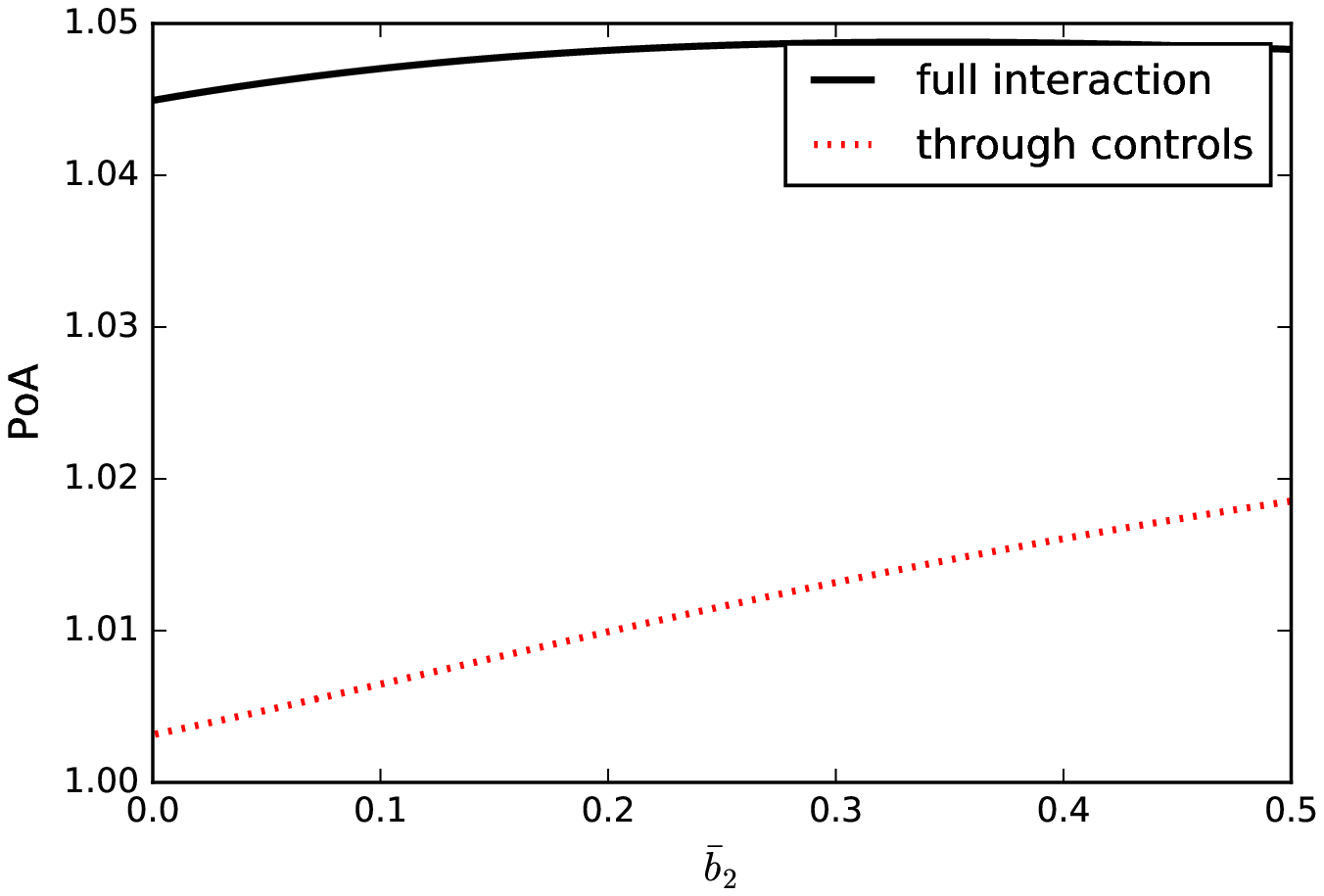}
	\end{subfigure}
	\begin{subfigure}{.45\textwidth}
		\includegraphics[scale=0.5]{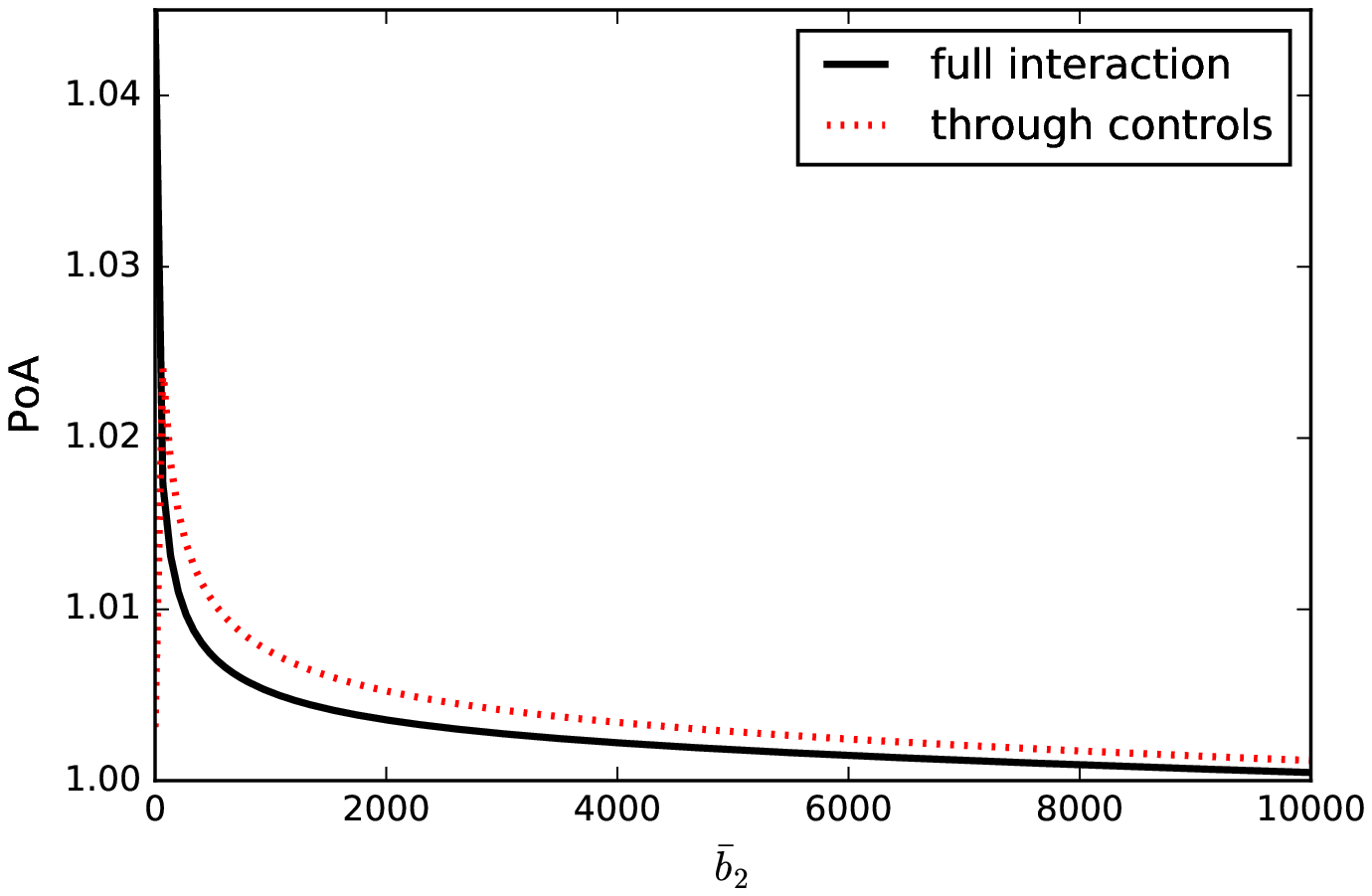}
	\end{subfigure}
	\caption{$PoA$ as we vary $\bar{b}_2$.}
	\label{fig:b2bar}
\end{figure}
\begin{figure}[!htb]
	\centering
	\begin{subfigure}{.45\textwidth}
		\includegraphics[scale=0.5]{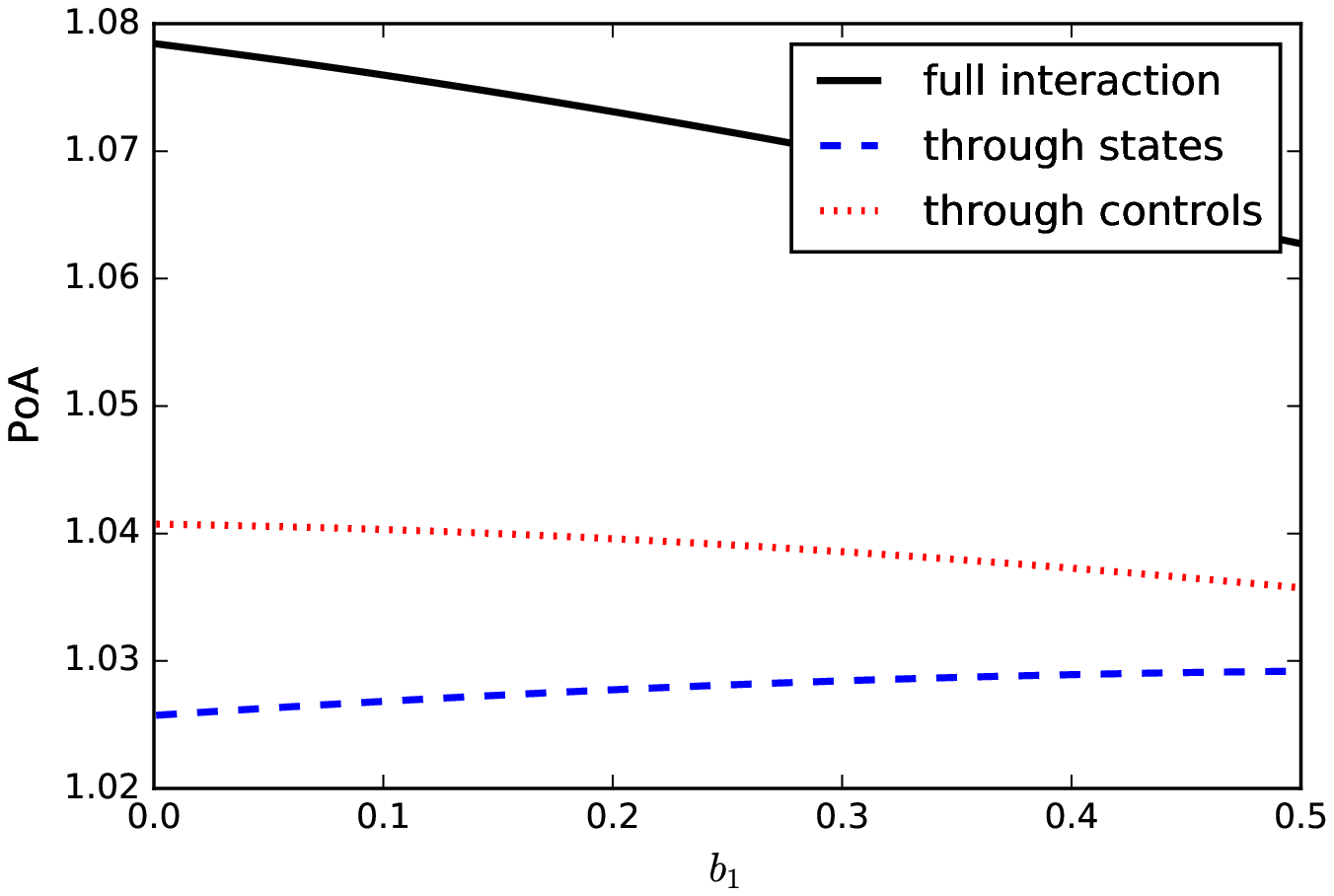}
	\end{subfigure}
	\begin{subfigure}{.45\textwidth}
		\includegraphics[scale=0.5]{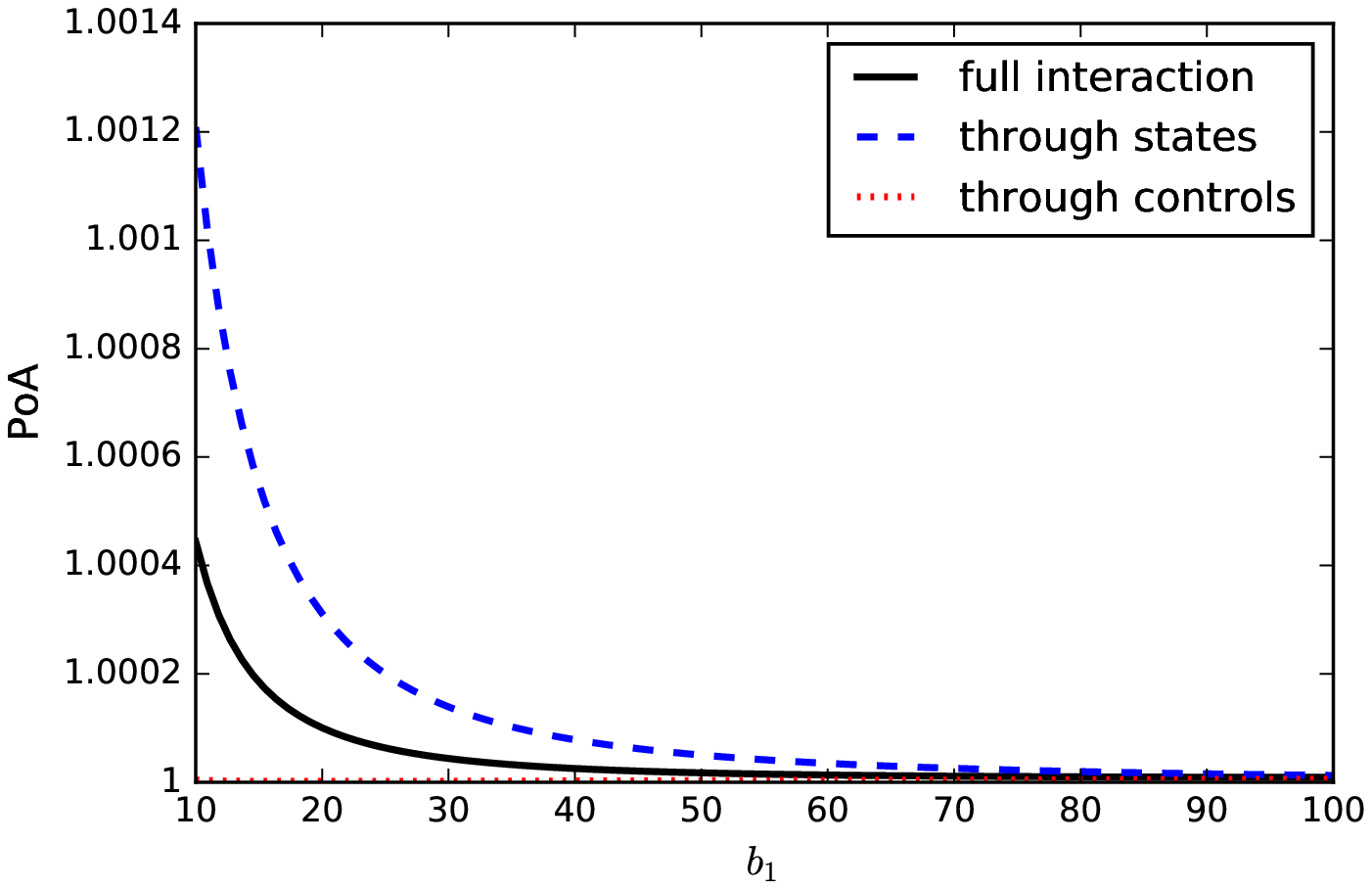}
	\end{subfigure}
	\caption{$PoA$ as we vary $b_1$.}
	\label{fig:b1}
\end{figure}
\begin{figure}[!htb]
	\centering
	\begin{subfigure}{.45\textwidth}
		\includegraphics[scale=0.5]{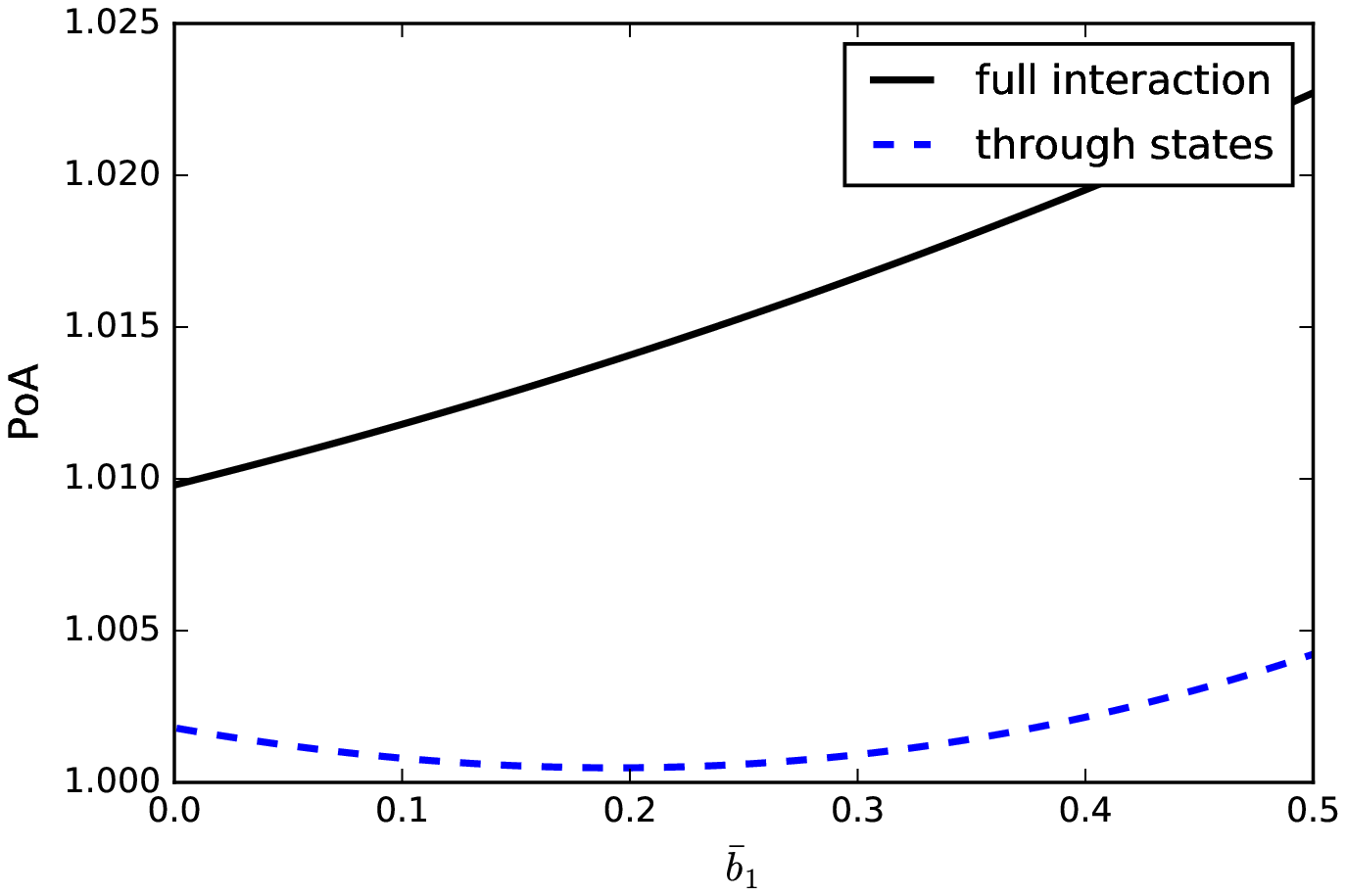}
	\end{subfigure}
	\begin{subfigure}{.45\textwidth}
		\includegraphics[scale=0.5]{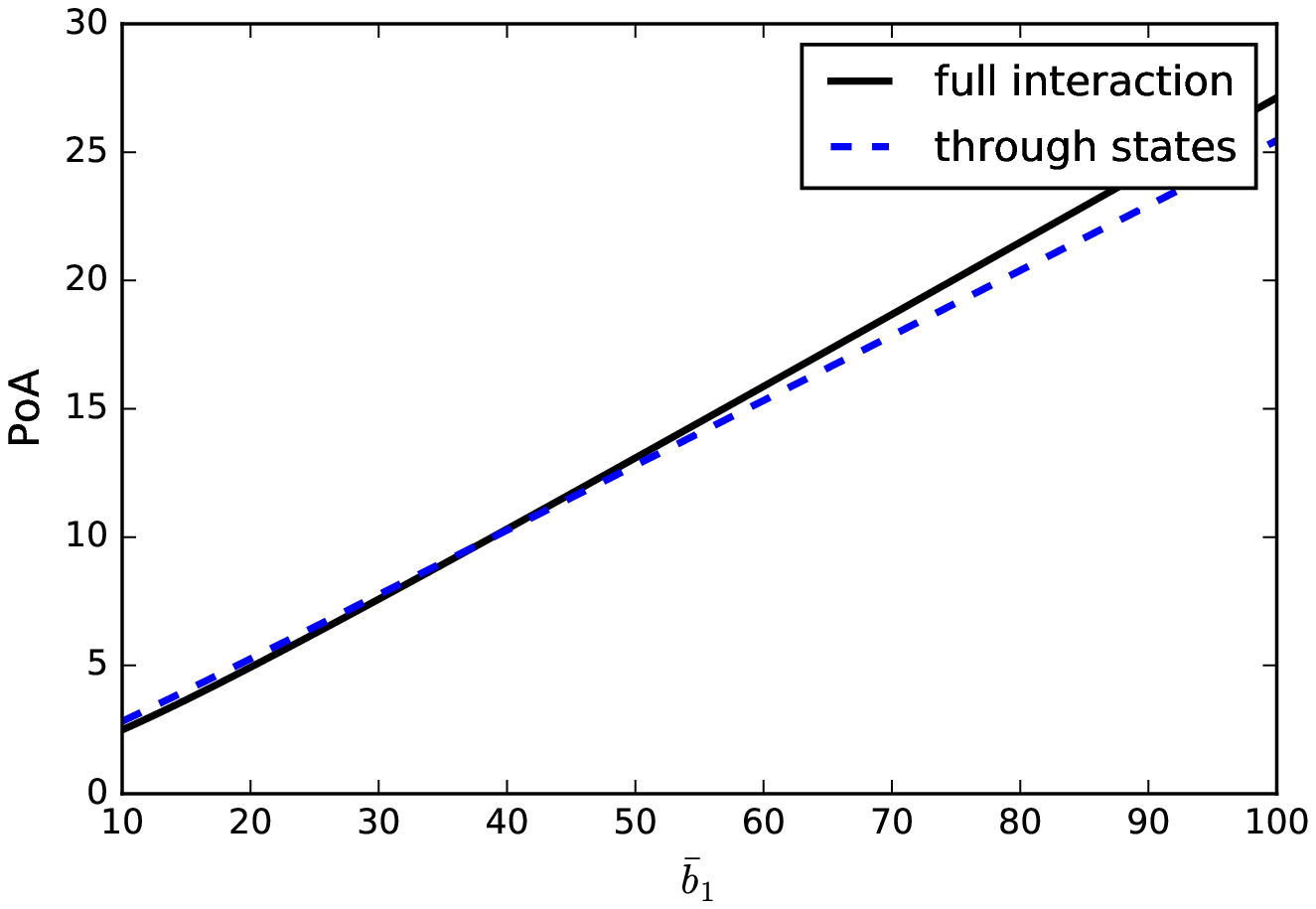}
	\end{subfigure}
	\caption{$PoA$ as we vary $\bar{b}_1$.}
	\label{fig:b1bar}
\end{figure}

Figure \ref{fig:r_rbar} confirms as in Proposition \ref{prop:r_bar_r} that $\lim_{r \to \infty}PoA= 1$ and $\lim_{\bar{r} \to \infty}PoA= 1$. Propositions \ref{prop:b2_to_infty} and \ref{prop:b2_to_0} are confirmed in Figure \ref{fig:b2}. In the case of full interaction, we have $\frac{q + \bar{q}(1-s)}{r + \bar{r}(1-\bar{s})}= \frac{q + \bar{q}(1-s)^2}{r + \bar{r}(1-\bar{s})^2}$ and we see that $\lim_{b_2 \to \infty}PoA= 1$. In the cases of interaction only through the states or interaction only through the controls, then $\frac{q + \bar{q}(1-s)}{r + \bar{r}(1-\bar{s})}\neq \frac{q + \bar{q}(1-s)^2}{r + \bar{r}(1-\bar{s})^2}$, and we see that $\lim_{b_2 \to \infty}PoA> 1$, confirming Proposition \ref{prop:b2_to_infty}. For Proposition \ref{prop:b2_to_0}, when there is only interaction through the states, then $\bar{b}_2=0$ and we see that $\lim_{b_2 \to 0}PoA= 1$. When there is full interaction or only interaction through the controls, then $\bar{b}_2>0$ and we see that $\lim_{b_2 \to 0}PoA>1$. For Proposition \ref{prop:b2bar}, Figure \ref{fig:b2bar} confirms that $\lim_{\bar{b}_2 \to 0}PoA>1$ and $\lim_{\bar{b}_2 \to \infty}PoA= 1$. Note that the condition $\frac{r + \bar{r}(1- \bar{s})^2}{r + \bar{r}(1-\bar{s})} \neq \frac{q_T+\bar{q}_T(1-s_T)^2}{q_T+\bar{q}_T(1-s_T)}$ is not satisfied for the full interaction case, and is therefore a sufficient, but not necessary, assumption. This agrees with the conclusion of Remark \ref{remark_6}. In Figures \ref{fig:b1} and \ref{fig:b1bar}, we note that Proposition \ref{prop:b1_b1bar} is confirmed. The condition $\frac{b_2}{b_2+\bar{b}_2}\cdot \frac{r + \bar{r}(1- \bar{s})^2}{r + \bar{r}(1-\bar{s})}\cdot (q_T+\bar{q}_T(1-s_T)) \neq q_T+\bar{q}_T(1-s_T)^2$ is satisfied for all three interaction cases and we see that $\lim_{b_1 \to 0}PoA >1$, $\lim_{\bar{b}_1 \to 0}PoA>1$, $\lim_{b_1 \to \infty}PoA = 1$, and $\lim_{\bar{b}_1 \to \infty}PoA = \infty$. Thus, the numerical computations confirm the results presented in Propositions \ref{prop:r_bar_r}-\ref{prop:b1_b1bar}.

\section{\textbf{Conclusion}} \label{sec:conclusion}
We defined the price of anarchy ($PoA$) in the context of extended mean field games as the ratio of the worst case social cost when the players are in a mean field game equilibrium to the social cost as computed by a central planner. Since the central planner does not require that the players be in a mean field game equilibrium, the central planner will realize a social cost that is no worse than that of a mean field game equilibrium. Thus, $PoA \geq 1$.

We computed the price of anarchy for linear quadratic extended mean field games, for which explicit computations are possible. We identify a large class of models for which $PoA=1$ (see Proposition \ref{prop: lq_subtract_FBSDEs} and Corollaries \ref{corollary_1} and \ref{corollary_2}), as well as giving a sufficient and necessary condition to have $PoA=1$ (see Theorem \ref{thm:PoA_equivalent_1} and Corollary \ref{corollary_3}). We also derive some limiting cases where $PoA \rightarrow 1$ as certain parameters tend to zero or to infinity (see Propositions \ref{prop:r_bar_r}-\ref{prop:b1_b1bar}). The numerics support our theoretical results.

\section{\textbf{Acknowledgments}}
The authors would like to thank the organizers of CEMRACS 2017 for putting together such a successful summer school. We would also like to thank the referees of the CEMRACS 2017 proceedings for their feedback.

\begin{appendices}

\section{\textbf{Solving Linear FBSDEs of McKean-Vlasov Type}} \label{appendix_solving_fbsdes}

Consider a linear FBSDE system of McKean-Vlasov type:
\begin{equation}
\begin{split}
        dX_t&=\left(a^x_tX_t+a^{\bar{x}}_t \mathbb{E}X_t+a^y_tY_t+a^{\bar{y}}_t \mathbb{E}Y_t\right)dt+\sigma dW_t, \quad X_0=\xi, \\
        dY_t&=\left(b^x_tX_t+b^{\bar{x}}_t \mathbb{E}X_t+b^y_tY_t+b^{\bar{y}}_t \mathbb{E}Y_t\right)dt +Z_t dW_t, \quad Y_T=c^xX_T+c^{\bar{x}}\mathbb{E}X_T.
\end{split}
\label{eq:FBSDE_general}
\end{equation}
For the LQEMFG model considered in section \ref{sec:EMFG}, the FBSDE system in equation (\ref{eq:FBSDE_EMFG}) is of the form of equation (\ref{eq:FBSDE_general}) if we set:
\begin{equation*}
\begin{split}
    &a^x_t=b_1(t),\ a^{\bar{x}}_t=\bar{b}_1(t),\ a^y_t=a^{MFG}(t)b_2(t),\ a^{\bar{y}}_t=b^{MFG}(t)b_2(t)+c^{MFG}(t)\bar{b}_2(t) \\
    &b^x_t=-(q(t)+\bar{q}(t)),\ b^{\bar{x}}_t=\bar{q}(t)s(t),\ b^y_t=-b_1(t),\ b^{\bar{y}}_t=0 \\
    &c^x_t=q_T+\bar{q}_T,\ c^{\bar{x}}_t=-\bar{q}_Ts_T.
\end{split}
\end{equation*}
For the LQEMKV model considered in section \ref{sec:EMKV}, the FBSDE system in equation (\ref{eq:FBSDE_EMKV}) is of the form of equation (\ref{eq:FBSDE_general}) if we set:
\begin{equation*}
\begin{split}
    &a^x_t=b_1(t),\ a^{\bar{x}}_t=\bar{b}_1(t),\ a^y_t=a^{MKV}(t)b_2(t),\ a^{\bar{y}}_t=b^{MKV}(t)b_2(t)+c^{MKV}(t)\bar{b}_2(t) \\
    &b^x_t=-(q(t)+\bar{q}(t)),\ b^{\bar{x}}_t=-s(t)\bar{q}(t)(s(t)-2),\ b^y_t=-b_1(t),\ b^{\bar{y}}_t=-\bar{b}_1 \\
    &c^x_t=q_T+\bar{q}_T,\ c^{\bar{x}}_t=s_T\bar{q}_T(s_T-2).
\end{split}
\end{equation*}

Now we return to the general FBSDE system (\ref{eq:FBSDE_general}). By taking expectations in equation (\ref{eq:FBSDE_general}), and letting $\bar{x}_t$ and $\bar{y}_t$ denote $\mathbb{E}X_t$ and $\mathbb{E}Y_t$, respectively, we get:
\begin{equation}
\begin{split}
        \dot{\bar{x}}_t&=(a^x_t+a^{\bar{x}}_t)\bar{x}_t+(a^y_t+a^{\bar{y}}_t)\bar{y}_t, \quad \bar{x}_0=\mathbb{E}(\xi), \\
        \dot{\bar{y}}_t&=(b^x_t+b^{\bar{x}}_t)\bar{x}_t+(b^y_t+b^{\bar{y}}_t)\bar{y}_t, \quad \bar{y}_T=(c^x+c^{\bar{x}})\bar{x}_T,
\end{split}
\label{eq:FBSDE_general_means}
\end{equation}
where the dot is the standard ODE notation for a derivative. We then make the ansatz $\bar{y}_t=\bar{\eta}_t\bar{x}_t+\bar{\chi}_t$ for deterministic functions $[0,T] \ni t \mapsto \bar{\eta}_t \in \mathbb{R}$ and $[0,T] \ni t \mapsto \bar{\chi}_t \in \mathbb{R}$. By plugging in the ansatz, the system in equation (\ref{eq:FBSDE_general_means}) is equivalent to the ODE system:
\begin{equation*}
\begin{split}
    \dot{\bar{\eta}}_t+(a^y_t+a^{\bar{y}}_t) \bar{\eta}_t^2+(a^x_t+a^{\bar{x}}_t-b^y_t-b^{\bar{y}}_t) \bar{\eta}_t -b^x_t-b^{\bar{x}}_t&=0, \quad \bar{\eta}_T=c^x+c^{\bar{x}}, \\
    \dot{\bar{\chi}}_t+(\bar{\eta}_t(a^y_t+a^{\bar{y}}_t)-b^y_t-b^{\bar{y}}_t)\bar{\chi}_t &=0, \quad \bar{\chi}_T=0.
\end{split}
\end{equation*}
The first equation is a Riccati equation. Note that $\bar{\chi}_t$ solves a first order homogeneous linear equation. Thus $\bar{\chi}_t=0$, $\forall t\in[0,T]$. Once the equation for $\bar{\eta}_t$ is solved, we can compute $\bar{x}_t$ by solving the linear ODE:
\begin{equation*}
\begin{split}
    \dot{\bar{x}}_t&=(a^x_t+a^{\bar{x}}_t+(a^y_t+a^{\bar{y}}_t)\bar{\eta}_t)\bar{x}_t, \quad \bar{x}_0=\mathbb{E}(\xi),
\end{split}
\end{equation*}
and thus,
\begin{equation*}
    \bar{x}_t=\mathbb{E}(\xi) e^{\int_0^t(a^x_u+a^{\bar{x}}_u+(a^y_u+a^{\bar{y}}_u)\bar{\eta}_u)du}.
\end{equation*}

Once we have computed $(\bar{x}_t)_{0\leq t\leq T}$, we can rewrite the original FBSDE system:
\begin{equation*}
\begin{split}
        dX_t&=\left(a^x_tX_t+a^y_tY_t+a^0_t\right)dt+\sigma dW_t, \quad X_0=\xi, \\
        dY_t&=\left(b^x_tX_t+b^y_tY_t+b^0_t\right)dt +Z_t dW_t, \quad Y_T=c^xX_T+c^0,
\end{split}
\end{equation*}
with:
\begin{equation*}
\begin{split}
    a^0_t=(a^{\bar{x}}_t+a^{\bar{y}}_t \bar{\eta}_t)\bar{x}_t, \quad b^0_t=(b^{\bar{x}}_t+b^{\bar{y}}_t \bar{\eta}_t)\bar{x}_t, \quad c^0=c^{\bar{x}}\bar{x}_T.
\end{split}
\end{equation*}
Now we make the ansatz: $Y_t=\eta_t X_t+\chi_t$, which reduces the problem to the ODE system:
\begin{equation*}
\begin{split}
    \dot{\eta}_t+a^y_t\eta_t^2+(a^x_t-b^y_t)\eta_t-b^x_t&=0, \quad \eta_T=c^x, \\
    \dot{\chi}_t+(-b^y_t+a^y_t\eta_t)\chi_t+a^0_t\eta_t-b^0_t&=0, \quad \chi_T=c^0, \\
    Z_t&=\sigma \eta_t.
\end{split}
\end{equation*}
Again, the first equation is a Riccati equation. Note that it is not necessary to solve for $\chi_t$ because of the relationship:
\begin{equation*}
    \bar{\eta}_t\bar{x}_t=\bar{y}_t=\mathbb{E}(Y_t)=\mathbb{E}(\eta_t X_t+\chi_t)=\eta_t \bar{x}_t+\chi_t.
\end{equation*}
Thus,
$\chi_t=(\bar{\eta}_t-\eta_t)\bar{x}_t.$

In summary, the solution to the linear FBSDE of McKean-Vlasov type is reduced to solving linear ODEs and Riccati equations. It will also be useful to compute $Var(X_t)$, which we denote by $v_t$. After we have solved the above equations, we have:
\begin{equation*}
\begin{split}
        dX_t&=\left((a^x_t+a^y_t \eta_t)X_t+a^y_t \chi_t+a^0_t\right)dt+\sigma dW_t, \quad X_0=\xi.
\end{split}
\end{equation*}
Thus,
\begin{equation*}
    v_t=Var(X_t)=Var(\xi)e^{\int_0^t 2(a^x_s+a^y_s\eta_s)ds}+\sigma^2 \int_0^t e^{2 \int_s^t (a^x_u+a^y_u\eta_u) du}ds.
\end{equation*}
In the case where the coefficients are time-independent, the Riccati equations for $\bar{\eta}_t$ and $\eta_t$ can be solved explicitly.

\subsection*{\textbf{Scalar Riccati Equation}} If the scalar Riccati equation: 
	\begin{equation*}
		\dot{\rho}_t - B \rho_t^2 - 2A \rho_t  + C = 0, 
	\end{equation*}
	with terminal condition $\rho_T = D$ satisfies:
	\begin{equation}
		B \neq 0,\ BD \geq 0,\ BC>0,
	\label{eq:Riccati_assumptions}
	\end{equation}
	then it has a unique solution:
\begin{equation}
		\rho_t= \frac{C(1-e^{-(\delta^+ - \delta^-)(T-t)})+D( \delta^+ -\delta^-e^{-(\delta^+ - \delta^-)(T-t)}) }{BD(1-e^{-(\delta^+ - \delta^-)(T-t)})+ \delta^+e^{-(\delta^+ - \delta^-)(T-t)}  -\delta^- },
	\label{eq:riccati_ut_sol_app}
\end{equation}
	with $\delta^\pm = -A \pm \sqrt{(A)^2 + B C}$.
	
	Furthermore, if $B \to 0$ and $A\neq 0$, we can deduce that the limiting solution of the scalar Riccati equation coincides with the linear first-order differential equation:
	\begin{equation*}
		\dot{\rho}_t - 2A \rho_t +C = 0,
	\end{equation*}
	with terminal condition $\rho_T = D$, namely:
	\begin{equation*}
		\rho_t = \left(D - \frac{C}{2A} \right) e^{-2A (T-t)} + \frac{C}{2A}.
	\end{equation*}
	
	If $B \to 0$ and $A = 0$, the limiting solution of the scalar Riccati equation coincides with the linear first-order differential equation:
	\begin{equation*}
	    \dot{\rho}_t +C = 0
	\end{equation*}
	with terminal condition $\rho_T = D$, namely:
	\begin{equation*}
	    \rho_t = D + C(T-t).
	\end{equation*}\\
Hence, returning to the linear FBSDE (\ref{eq:FBSDE_general}), for $\bar{\eta}_t$, we use:
\begin{equation*}
	\begin{array}{rl}
		A=-\frac{1}{2}(a^x+a^{\bar{x}}-b^y-b^{\bar{y}}), \quad, B=-(a^y+a^{\bar{y}}), \quad C=-(b^x+b^{\bar{x}}),\quad D=\ c^x+c^{\bar{x}}.
	\end{array}
\end{equation*}
The conditions (\ref{eq:Riccati_assumptions}) are satisfied if $-(a^y+a^{\bar{y}})>0$, $-(b^x+b^{\bar{x}})>0$, and $c^x+c^{\bar{x}} \geq 0$.

For $\eta_t$, we use:
\begin{equation*}
	\begin{array}{rl}
		A=-\frac{1}{2}(a^x-b^y), \quad B=-a^y, \quad C=-b^x, \quad D=\ c^x.
	\end{array}
\end{equation*}
The conditions (\ref{eq:Riccati_assumptions}) are satisfied if $-a^y>0$, $-b^x>0$, and $c^x \geq0$. Returning to the LGEMFG and LGEMKV problems, if we assume the coefficients are non-negative, we see that these conditions are exactly assumption (\ref{eq:assumptions}).

\end{appendices}

\end{document}